\documentclass[psamsfonts]{amsart}

\usepackage{amssymb,amsfonts}
\usepackage[all,arc]{xy}
\usepackage{enumerate}
\usepackage{mathrsfs}
\usepackage{url}
\usepackage{mathtools}
\usepackage[font=small,labelfont=bf]{caption}

\newtheorem{thm}{Theorem}[section]
\newtheorem{cor}[thm]{Corollary}
\newtheorem{prop}[thm]{Proposition}
\newtheorem{lem}[thm]{Lemma}

\newtheorem{quest}[thm]{Question}

\theoremstyle{definition}
\newtheorem{defn}[thm]{Definition}

\newtheorem{fact}[thm]{Fact}

\theoremstyle{remark}
\newtheorem{rem}[thm]{Remark}

\makeatletter
\let\c@equation\c@thm
\makeatother
\numberwithin{equation}{section}

\def\Ind{\setbox0=\hbox{$x$}\kern\wd0\hbox to 0pt{\hss$\mid$\hss} \lower.9\ht0\hbox to 0pt{\hss$\smile$\hss}\kern\wd0} 

\def\Notind{\setbox0=\hbox{$x$}\kern\wd0\hbox to 0pt{\mathchardef \nn=12854\hss$\nn$\kern1.4\wd0\hss}\hbox to 0pt{\hss$\mid$\hss}\lower.9\ht0 \hbox to 0pt{\hss$\smile$\hss}\kern\wd0} 

\def\ind{\mathop{\mathpalette\Ind{}}} 

\def\nind{\mathop{\mathpalette\Notind{}}} 

\title{On Kim-Independence}

\author{Itay Kaplan and Nicholas Ramsey}

\date{\today}

\thanks{The first author would like to thank the Israel Science Foundation
for partial support of this research (Grant no. 1533/14). }

\begin{document}

\begin{abstract}
We study NSOP$_{1}$ theories.  We define \emph{Kim-independence}, which generalizes non-forking independence in simple theories and corresponds to non-forking at a generic scale.  We show that Kim-independence satisfies a version of Kim's lemma, local character, symmetry, and an independence theorem and that, moreover, these properties individually characterize NSOP$_{1}$ theories.  We describe Kim-independence in several concrete theories and observe that it corresponds to previously studied notions of independence in Frobenius fields and vector spaces with a generic bilinear form.  
\end{abstract}

\maketitle

\setcounter{tocdepth}{1}
\tableofcontents

\section{Introduction}

The class of simple theories was one of the first classes of unstable theories to receive extensive study.  The starting point is \emph{Classification Theory}, where, in the course of studying stable theories, Shelah isolates \emph{local character} as a key property of non-forking independence and observes a dichotomy in the way local character can fail, a theorem we now recognize as saying that a non-simple theory must have the tree property of the first or second kind \cite[Theorem III.7.11]{shelah1990classification}.  Shortly after the publication of the first edition of \cite{shelah1990classification}, Shelah defined the class of simple theories and characterized them in terms of a certain chain condition of the Boolean algebra of non-weakly dividing formulas, which in turn led to consistency results on their saturation spectra \cite{shelah1980simple}.  The aim of that work was to obtain an `outside' set-theoretic definition of the class to support the claim that simplicity marked a dividing line.  In separate developments, questions concerning concrete examples created the need for new methods to treat unstable structures.  Hrushovski and Pillay used local stability and $S_{1}$-rank in the study of the definability of groups in pseudo-finite and PAC fields in \cite{hrushovski1994groups}, and these methods were situated in the broader context of PAC structures studied by Hrushovski \cite{hrushovski1991pseudo}, where an independence theorem was proved.  Moreover, Lachlan's far-reaching theory of smoothly approximated structures furnished examples of tame unstable theories.  After Kantor, Liebeck, and Macpherson \cite{KLM} classified the primitive smoothly approximable structures, Cherlin and Hrushovski \cite{cherlin2003finite} used stability theoretic methods concerning independence and amalgamation to describe how these primitive pieces fit together to form a quasi-finite structure.  

Kim's thesis and subsequent work by Kim and Pillay showed how to regard these developments as instances of a common theory, with non-forking independence at its center \cite{kim1998forking}, \cite{kim1997simple}.  Kim proved that in a simple theory, forking and dividing coincide, non-forking independence is symmetric and transitive, and Kim and Pillay proved that the independence theorem holds over models.  Moreover, Kim showed that symmetry and transitivity of non-forking both individually \emph{characterize} the simple theories, and Kim and Pillay showed that any independence relation satisfying the basic properties of non-forking independence must actually coincide with non-forking independence, giving both a striking characterization of the simple theories and a powerful method for showing that a particular theory is simple, namely by observing that it has an independence relation of the right kind.  

Here, we study the class of NSOP$_{1}$ theories.  These are the theories which do not have the property SOP$_{1}$, which form a class of theories that properly contain the simple theories and which are contained inside the class of theories without the tree property of the first kind.  SOP$_{1}$ was defined by D\v{z}amonja and Shelah in their study of the $\unlhd^{*}$-order \cite{dvzamonja2004maximality} and later studied by Shelah and Usvyatsov in \cite{shelah2008more}.  The NSOP$_{1}$ theories were characterized as the theories satisfying a weak independence theorem for invariant types by Chernikov and the second-named author in \cite{ArtemNick}.  This characterization provided a point of contact between the combinatorics of  model-theoretic tree properties and the study of definability in particular algebraic examples.   Chatzidakis \cite{ZoePAC2}, \cite{chatzidakis2002properties} studied independence in $\omega$-free PAC fields and, more generally, Frobenius fields and showed that the independence theorem holds for these structures even though they are not simple.  Similarly, Granger showed in his thesis that the model companion of the theory of infinite-dimensional vector spaces with a bilinear form is not simple but nonetheless comes equipped with a good notion of independence.  The amalgamation criterion of \cite{ArtemNick} established that these structures have NSOP$_{1}$ theory by appealing to the existence of these independence relations, but what was missing was a theory of independence in NSOP$_{1}$ theories more generally.  The purpose of this paper is to establish exactly such a theory.  

One central tool in the study of forking in simple theories is Kim's lemma:  in a simple theory, a formula divides over a set $A$ if and only if it divides with respect to some Morley sequence over $A$ if and only if it divides for all Morley sequences over $A$.  In \cite{chernikov2012forking}, this was shown to hold over models in NTP$_{2}$ theories, provided that the Morley sequence is a strict invariant Morley sequence.  In the setting of NSOP$_{1}$ theories, we find a new phenomenon:  forking which is \emph{never} witnessed by a generic sequence.  In fact, we show that any NSOP$_{1}$ theory with a universal witness to dividing must be simple (Proposition \ref{nomorley} below) and that forking need not equal dividing in an NSOP$_{1}$ theory.  Nonetheless, we find that, by restricting attention to the forking that is witnessed by a generic sequence, one can recover many of the properties of forking in simple theories.  We show moreover that this kind of simplicity at a generic scale is characteristic of NSOP$_{1}$ theories.

There is considerable freedom in the choice of notion of generic sequence.  One suggestion which inspired our work is due to Kim, who proposed in his 2009 talk on NTP$_{1}$ theories \cite{KimNTP1} that one might develop an independence theory for NTP$_{1}$ theories or a subclass therein by considering only formulas which divide with respect to every non-forking Morley sequence.  Compared to invariance or finite satisfiability, forking is a relatively weak notion of independence and this notion proved unwieldy at the beginning stages of developing the theory presented here.  However, Hrushovski's study of $q$-dividing \cite{hrushovski2012stable} and Malliaris and Shelah's characterization of NTP$_{1}$ theories in terms of higher formulas \cite{malliaris2015model} provided evidence that one might be able to build a theory around an investigation of formulas that divide with respect to a Morley sequence in a global invariant or finitely satisfiable type.  Building off this work, we introduce the notion of Kim-dividing -- a formula \emph{Kim-divides} over a set $A$ if it divides with respect to a Morley sequence in a global $A$-invariant type -- and the associated notion of independence, \emph{Kim-independence}.  Our first observation is that a theory is NSOP$_{1}$ if and only if Kim-dividing satisfies a version of Kim's lemma over models, where a formula divides with respect to a Morley sequence in \emph{some} global invariant type extending the type of the parameters if and only if it divides with respect to \emph{every} Morley sequence in an appropriate invariant type.  

From Kim's lemma for Kim-dividing, many familiar properties of non-forking independence follow:  Kim-forking equals Kim-dividing, Kim-independence satisfies extension and a version of the chain condition, etc.  In subsequent sections, we investigate additional properties of Kim-independence in NSOP$_{1}$ theories and prove that, in many cases, these properties are characteristic of NSOP$_{1}$.  In Section \ref{localcharsection} we observe a form of local character for Kim-independence in the context of NSOP$_{1}$ theories.  In Section \ref{symmetrysection}, we show additionally that Kim-independence is symmetric over models.  The argument there centers upon the notion of a \emph{tree Morley sequence} which is defined in terms of indiscernible trees.  We show that tree Morley sequences always witness Kim-dividing and prove a version of the chain condition for them.  In Section \ref{itsection}, we prove the independence theorem.  In Section \ref{forkingsection}, we prove that in an NSOP$_{1}$ theory a formula Kim-divides over a model if and only if it divides with respect to every non-forking Morley sequence in the parameters and this too characterizes NSOP$_{1}$ theories.  This means that Kim-independence could have been defined from the outset in essentially the way Kim proposed, but curiously, proving anything about this notion without making use of invariant types seems quite difficult.  In Section \ref{mtsection}, we state our main theorem:  Kim's lemma for Kim-dividing, symmetry over models, and the independence theorem both hold in NSOP$_{1}$ theories and individually characterize NSOP$_{1}$ theories.  We also show that the simple theories can be characterized in several new ways in terms of Kim-independence.  In particular, we show that Kim-independence coincides with non-forking over models if and only if the theory is simple, which means that our theorems imply the corresponding facts for non-forking independence in a simple theory.  

We conclude the paper with Section \ref{examplesection} where we describe Kim-independence explicitly in several concrete examples.  We show it may be described in purely algebraic terms in the case of Frobenius fields, where Kim-independence turns out to coincide with \emph{weak independence}, as defined by Chatzidakis.  We also show that in Granger's two-sorted theory of a vector space over an algebraically closed field with a generic bilinear form, Kim-independence is closely related to Granger's $\Gamma$-independence and may be given a simple algebraic description.  These results suggest the naturality and robustness of Kim-dividing, but also serve to explain the simplicity-like phenomena observed in these concrete examples on the basis of a general theory.  We additionally describe a combinatorial example of a NSOP$_{1}$ theory, based on a variant of $T^{*}_{feq}$ introduced by D\v{z}amonja and Shelah, which furnishes counter-examples to some \emph{a priori} possible strengthenings of the results we prove.  In particular, we give the first example of a simple non-cosimple type, answering a question of Chernikov \cite{ChernikovNTP2}, and the first example of an NSOP$_{3}$ theory in which every complete type has a global non-forking extension but forking does not equal dividing, answering a question of Conant \cite{conant2014forking}.    

\section{Syntax}

In this section we will define SOP$_{1}$ and prove its equivalence with a syntactic property of a different form.  This will allow us to relate SOP$_{1}$ to dividing.  We will often work with arrays and trees.  Suppose \((c_{ij})_{i < \kappa, j < \lambda}\) is an array.  Write \(\overline{c}_{i} = (c_{i,j})_{j < \lambda}\) for the $i$th row of the array and \(\overline{c}_{<i}\) for the sequence of rows with index less than $i$, i.e. \((\overline{c}_{k})_{k < i}\).  Suppose $\mathcal{T}$ is a tree, $(a_{\eta})_{\eta \in \mathcal{T}}$ is a collection of tuples indexed by $\mathcal{T}$.  We write $\unlhd$ for the tree partial order and $<_{lex}$ for the lexicographic order on $\mathcal{T}$.  For a node $\eta \in \mathcal{T}$, write $a_{\unlhd \eta}$ for the sequence $\langle a_{\nu} :\nu \unlhd \eta \rangle$, and likewise $a_{\vartriangleleft \eta}$ for $\langle a_{\nu} :\nu \vartriangleleft \eta \rangle$.  We use the notation $a_{\unrhd \eta}$ and $a_{\vartriangleright \eta}$ similarly.  If the tree $\mathcal{T}$ is contained in $2^{<\kappa}$ or $\omega^{<\kappa}$, we write $0^{\alpha}$ to denote the element of the tree of length $\alpha$ consisting of all zeros.  Throughout the paper, $T$ denotes a complete theory and $\mathbb{M} \models T$ is a monster model of $T$.  

\begin{defn} \cite[Definition 2.2]{dvzamonja2004maximality} \label{sop1def}
The formula $\varphi(x;y)$ has SOP$_{1}$ if there is a collection of tuples $(a_{\eta})_{\eta \in 2^{<\omega}}$ so that 
\begin{itemize}
\item For all $\eta \in 2^{\omega}$, $\{\varphi(x;a_{\eta | \alpha}) : \alpha < \omega\}$ is consistent.
\item For all $\eta \in 2^{<\omega}$, if $\nu \unrhd \eta \frown \langle 0 \rangle$, then $\{\varphi(x;a_{\nu}), \varphi(x;a_{\eta \frown 1})\}$ is inconsistent.
\end{itemize}
We say $T$ is SOP$_{1}$ if some formula has SOP$_{1}$ modulo $T$.  $T$ is NSOP$_{1}$ otherwise.  
\end{defn}

The following lemma is close to \cite[Lemma 5.2]{ArtemNick}, but with a key strengthening which will allow us to relax the $2$-inconstency in the definition of SOP$_{1}$ to a version with $k$-inconsistency.   

\begin{lem}\label{fancy}
Suppose \((c_{i,j})_{i < \omega, j < 2}\) is an array where $c_{i,j} = (d_{ij},e_{ij})$ for all $i,j$ and \(\chi_{1}(x;y)\) and $\chi_{2}(x;z)$ are formulas over \(C\).  Write $\psi(x;y,z)$ for $\chi_{1}(x;y) \wedge \chi_{2}(x;z)$ and suppose 
\begin{enumerate}
\item For all \(i < \omega\), \(e_{i,0} \equiv_{C c_{<i,0}e_{<i,1}} e_{i,1}\).
\item \(\{\psi(x;c_{i,0}) : i < \omega\}\) is consistent.
\item \(j \leq i \implies \{\chi_{1}(x;d_{i,0}), \chi_{2}(x;e_{j,1})\}\) is inconsistent.
\end{enumerate}
then \(T\) has SOP\(_{1}\).  
\end{lem}

\begin{proof}
By adding constants, we may assume $C = \emptyset$.  By Ramsey and compactness, we may assume $(\overline{c}_{i})_{i < \omega}$ is a $C$-indiscernible sequence.  By compactness again, we may extend the array to an array whose rows are indexed by the integers $(\overline{c}_{i})_{i \in \mathbb{Z}}$.  We will construct, for each $n < \omega$, a tree $(c_{\eta})_{\eta \in 2^{\leq n}}$ so that 
\begin{enumerate}
\item If $\nu \in 2^{n}$, then  $\{\psi(x;c_{\nu | i}) : i \leq n\}$ is consistent.
\item If $\nu \in 2^{<n}$ and $\nu \frown \langle 0 \rangle \vartriangleleft \eta$ then $\{\psi(x;c_{\eta}), \psi(x;c_{\nu \frown \langle 1 \rangle})\}$ is inconsistent.  
\item If $\eta \in 2^{n}$, $(c_{\nu})_{\nu \unlhd \eta} \equiv_{c_{< -n,0}e_{< -n,1}} (c_{i,0})_{-n \leq i \leq 0}$.
\end{enumerate}
To define $(c_{\eta})_{\eta \in 2^{\leq 0}}$, we put $c_{\emptyset} = c_{0,0}$.  Now suppose we are given $S_{n} = (c_{\eta})_{\eta \in 2^{\leq n}}$ satisfying the requirements.  There is an automorphism $\sigma$ taking $e_{-n,0}$ to $e_{-n,1}$ fixing $c_{<-n,0}e_{<-n,1}$.  Define $S_{n+1} = (c'_{\eta})_{\eta \in 2^{\leq n+1}}$ by $c'_{\emptyset} = c_{-(n+1),0}$ and, for all $\eta \in 2^{\leq n}$, $c'_{\langle 0 \rangle \frown \eta} = c_{\eta}$, $c'_{\langle 1 \rangle \frown \eta} = \sigma(c_{\eta})$.  Clearly all branches have the same type over $c_{<-(n+1),0}e_{<-(n+1),1}$ as $(c_{i,0})_{-(n+1) \leq i \leq 0}$.  Write $c'_{\eta} = (d'_{\eta},e'_{\eta})$ for all $\eta \in 2^{\leq n+1}$.  Now note that in both $S_{n}$ and $\sigma(S_{n})$ conditions (1) and (2) are preserved and that $\psi(x;c'_{\langle 1 \rangle})$ is inconsistent with $\psi(x;c'_{\langle 0 \rangle \frown \eta})$ for any $\eta \in 2^{\leq n}$ since $\chi_{2}(x;e_{-n,1}) \wedge \chi_{1}(x;d_{\eta})$ is consistent if and only if $\chi_{2}(x;e_{-n,1}) \wedge \chi_{1}(x;d_{i,0})$ is consistent, for $i = l(\eta) -n$.  Likewise, instantiating $\psi(x;y)$ along any branch through this tree yields something consistent:  any branch in $S_{n}$ or $\sigma(S_{n})$ has the same type over $c_{-(n+1),0}$ as $(c_{i,0})_{-n \leq i \leq 0}$ and $\{\psi(x;c_{i,0}) : -(n+1) \leq i \leq 0\}$ is consistent.  We conclude by compactness.  
\end{proof}

\begin{lem}\label{karyversion}
Suppose $\varphi(x;y)$ is a formula, $k$ is a natural number, and $(\overline{c}_{i})_{i \in I}$ is an infinite sequence with $\overline{c}_{i} = (c_{i,0}, c_{i,1})$ satisfying:
\begin{enumerate}
\item For all $i \in I$, $c_{i,0} \equiv_{\overline{c}_{<i}} c_{i,1}$. 
\item $\{\varphi(x;c_{i,0}) : i \in I\}$ is consistent.
\item $\{\varphi(x;c_{i,1}) : i \in I\}$ is $k$-inconsistent.
\end{enumerate}
Then $T$ has SOP$_{1}$.  
\end{lem}

\begin{proof}
By compactness and Ramsey, it suffices to prove this when $I = \mathbb{Q}$ -- so suppose $(c_{i,0},c_{i,1})_{i \in \mathbb{Q}}$ is an indiscernible sequence with $c_{i,0} \equiv_{\overline{c}_{<i}} c_{i,1}$, $\{\varphi(x;c_{i,0}) :i \in \mathbb{Q}\}$ is consistent, and $\{\varphi(x;c_{i,1}) : i \in \mathbb{Q}\}$ is $k$-inconsistent.  

For integers $l < l'$, define a partial type $\Gamma_{l,l'}(x)$ by 
$$
\{\varphi(x;c_{i,0}) : i \in (l+m, l+m+1) ,m \in \omega, m < l'-l\} \cup \{\varphi(x;c_{l+m,1}) : m < l' - l, m \in \omega\}.   
$$
Let $\Gamma_{l,l}(x) = \emptyset$.  Note that if $\Gamma_{l,l'}(x)$ is consistent then $\Gamma_{l+z,l'+z}(x)$ is consistent for any integer $z$ by indiscernibility of the sequence $(\overline{c}_{i})_{i \in \mathbb{Q}}$.  Let $n \in \omega$ be maximal so that $\Gamma_{0,n}(x)$ is consistent.  Note that $\Gamma_{0,0}(x)$ is consistent, as it is the empty partial type and we have 
$$
\Gamma_{0,k}(x) \vdash \{\varphi(x;c_{i,1}) : i \in \omega, i < k\},
$$
which is inconsistent, so $0 \leq n < k$.  So now we know $\Gamma_{-n,0}(x)$ is consistent and $\Gamma_{-n,1}(x) = \Gamma_{-n,0}(x) \cup \Gamma_{0,1}(x)$ is inconsistent.  By indiscernibility and compactness, we may fix some integer $N > 0$ so that 
$$
\Gamma_{-n,0}(x) \cup \{\varphi(x;c_{0,1})\} \cup \{\varphi(x;c_{\frac{j+1}{N}, 0}) : j \in \omega, j < N-1\}
$$
is inconsistent.  Now choose $\Delta(x) \subseteq \Gamma_{-n,0}(x)$ finite so that 
$$
\Delta(x) \cup \{\varphi(x;c_{0,1})\}  \cup \{\varphi(x;c_{\frac{j+1}{N}, 0}) : j \in \omega, j < N-1\}
$$
is inconsistent.  Let $z$ indicate the tuple of variables $(y_{0}, \ldots, y_{N-2})$ and let $\chi(x;z)$ be the formula $\chi(x;z) = \bigwedge_{i < N} \varphi(x;y_{i}) \wedge \bigwedge \Delta(x)$.  Let $(a_{i,j})_{i < \omega, j < 2}$ be defined as follows:
$$
a_{i,0} = (c_{i,0};d_{i,0}) = (c_{i,0}; c_{i+ \frac{1}{N},0}, \ldots, c_{i + \frac{N-1}{N}, 0}).
$$
Now choose $d_{i,1}$ so that $c_{i,0}d_{i,0} \equiv_{\overline{c}_{<i}} c_{i,1}d_{i,1}$ -- this is possible as $c_{i,0} \equiv_{\overline{c}_{<i}} c_{i,1}$.  Then we put $a_{i,1} = (c_{i,1}, d_{i,1})$.  Let $\psi(x;yz) = \varphi(x;y) \wedge \chi(x;z)$.  

To conclude, we have to establish the following:

\textbf{Claim}:  The array $(a_{i,j})_{i < \omega, j < 2}$ and the formulas $\varphi(x;y)$, $\chi(x;z)$ satisfy the following:
\begin{enumerate}
\item $a_{i,0} \equiv_{a_{<i,0}, c_{<i,1}} a_{i,1}$.
\item $\{\psi(x;a_{i,0}) : i < \omega\}$ is consistent.
\item If $l \leq l'$ then $\{\varphi(x;c_{l,1}) , \chi(x;d_{l',0})\}$ is inconsistent.
\end{enumerate}
\emph{Proof of claim}:  (1) follows from the fact that $a_{i,0} \equiv_{\overline{c}_{<i}} a_{i,1}$ and both $a_{<i,0}$ and $c_{<i,1}$ are enumerated in $\overline{c}_{<i}$.  Note that $\Gamma_{-n,0}(x)$ is consistent so, by indiscernibility, 
$$
\Gamma_{-n,0}(x) \cup \{\varphi(x;c_{i,0}) : i \in [0, \infty) \cap \mathbb{Q}\}
$$
is consistent, which establishes (2).  Finally, if $l \leq l'$, then $\{\varphi(x;c_{l,1}) , \chi(x;d_{l',0})\}$ implies 
$$
\{\varphi(x;c_{l,1})\} \cup \{\varphi(x;c_{l' + \frac{j+1}{N}}) : j \in \omega, j < N-1\} \cup \Delta(x).
$$  
By indiscernibility of $(\overline{c}_{i})_{i \in \mathbb{Q}}$ and the fact that $l \leq l'$, this set is consistent if and only if 
$$
\{\varphi(x;c_{0,1})\} \cup \{\varphi(x;c_{\frac{j+1}{N}, 0}) : j \in \omega, j < N-1\} \cup \Delta(x)
$$
is consistent.  As this latter set is inconsistent, this shows (3), which proves the claim.
The lemma now follows by Lemma \ref{fancy}.  
\end{proof}

Finally, we note that the criterion for SOP$_{1}$ from Lemma \ref{karyversion} is an equivalence.  This was implicit in \cite{ArtemNick}, at least in its $2$-inconsistent version, but we think that the property described by Lemma \ref{karyversion} is, in most cases, the more fruitful way of thinking about SOP$_{1}$ and therefore worth making explicit.  

\begin{prop}\label{arrayequivalent}
The following are equivalent, for a complete theory $T$:
\begin{enumerate}
\item $T$ has SOP$_{1}$.
\item There is a formula $\varphi$ and an array $(c_{i,j})_{i < \omega, j < 2}$ so that: 
\begin{enumerate}
\item $c_{i,0} \equiv_{\overline{c}_{<i}} c_{i,1}$ for all $i < \omega$.
\item $\{\varphi(x;c_{i,0}) : i < \omega\}$ is consistent.
\item $\{\varphi(x;c_{i,1}) :i < \omega\}$ is $2$-inconsistent.
\end{enumerate}
\item There is a formula $\varphi$ and an array $(c_{i,j})_{i < \omega, j < 2}$ so that: 
\begin{enumerate}
\item $c_{i,0} \equiv_{\overline{c}_{<i}} c_{i,1}$ for all $i < \omega$.
\item $\{\varphi(x;c_{i,0}) : i < \omega\}$ is consistent.
\item $\{\varphi(x;c_{i,1}) :i < \omega\}$ is $k$-inconsistent for some $k$.
\end{enumerate}
\end{enumerate}
\end{prop}

\begin{proof}
(3)$\implies$(1) is Lemma \ref{karyversion}.  

(1)$\implies$(2).  This follows from the proof of \cite[Proposition 5.6]{ArtemNick}.

(2)$\implies$(3) is obvious.  
\end{proof}

\begin{rem}
Though the configurations described in (2) and (3) are not obviously preserved by expansion, SOP$_{1}$ as defined in Definition \ref{sop1def} clearly is.  It follows, then, that one can take $(\overline{c}_{i})_{i < \omega}$ to be indiscernible with respect to some Skolemization in the language $L^{Sk}$ of $T$ and, moreover, obtain $c_{i,0} \equiv^{L^{Sk}}_{\overline{c}_{<i}} c_{i,1}$ for all $i < \omega$ (in fact, this is what the proof of \cite[Proposition 5.6]{ArtemNick} directly shows).  
\end{rem}

\section{Kim-dividing}

\subsection{Averages and Invariant Types}

\begin{defn}
A global type $q \in S(\mathbb{M})$ is called $A$\emph{-invariant} if $b \equiv_{A} b'$ implies $\varphi(x;b) \in q$ if and only if $\varphi(x;b') \in q$.  A global type $q$ is $\emph{invariant}$ if there is some small set $A$ such that $q$ is $A$-invariant.  If $q(x)$ and $r(y)$ are $A$-invariant global types, then the type $(q \otimes r)(x,y)$ is defined to be $\text{tp}(a,b/\mathbb{M})$ for any $b \models r$ and $a \models q|_{\mathbb{M}b}$.  We define $q^{\otimes n}(x_{0},\ldots, x_{n-1})$ by induction:  $q^{\otimes 1} = q$ and $q^{\otimes n + 1} = q(x_{n}) \otimes q^{\otimes n}(x_{0},\ldots, x_{n-1})$.  When $M$ is a model, write $a \ind^{i}_{M} b$ to mean $\text{tp}(a/Mb)$ extends to a global $M$-invariant type.  
\end{defn}

\begin{fact}\cite[Chapter 2]{simon2015guide} \label{tensor}
Given a global $A$-invariant type $q$ and positive integer $n$, $q^{\otimes n}$ is a well-defined $A$-invariant global type.  If $N \supset A$ is an $|A|^{+}$-saturated model and $p \in S(N)$ satisfies $\varphi(x;b) \in p \iff \varphi(x;b') \in p$ whenever $b,b' \in N$ and $b \equiv_{A} b'$, then $p$ extends uniquely to a global $A$-invariant type.  
\end{fact}

\begin{defn}
Suppose $q$ is an $A$-invariant global type and $I$ is a linearly ordered set.  By a \emph{Morley sequence in }$q$ \emph{over} $A$ \emph{of order type} $I$, we mean a sequence $(b_{\alpha})_{\alpha \in I}$ such that for each $\alpha \in I$, $b_{\alpha} \models q|_{Ab_{<\alpha}}$ where $b_{<\alpha} = (b_{\beta})_{\beta < \alpha}$.  Given a linear order $I$, we will write $q^{\otimes I} = q^{\otimes I}(x_{\alpha} : \alpha \in I)$ for the $A$-invariant global type so that if $\overline{b} \models q^{\otimes I}$ then $b_{\alpha} \models q|_{\mathbb{M}b_{<\alpha}}$ for all $\alpha \in I$.  
\end{defn}

The above definition of $q^{\otimes I}$ generalizes the finite tensor product $q^{\otimes n}$ -- given any global $A$-invariant type $q$ and linearly ordered set $I$, one may easily show that $q^{\otimes I}$ exists and is $A$-invariant, by Fact \ref{tensor} and compactness.  

\begin{defn}
Let \(I \subseteq \mathbb{M}^{n}\) be a collection of tuples, \(A \subseteq \mathbb{M}\) a set, and \(\mathcal{D}\) an ultrafilter over \(I\).  We define the \emph{average type of } \(\mathcal{D}\) \emph{over } \(A\) to be the type defined by 
\[
\text{Av}(\mathcal{D},A) = \{ \varphi(x;a) : a \in A \text{ and } \{b \in I : \mathbb{M} \models \varphi(b;a)\} \in \mathcal{D}\}.
\]
\end{defn}

\begin{fact}\cite[Lemma 4.1]{shelah1990classification}\label{average}
Let $I \subseteq \mathbb{M}^{n}$ be a collection of tuples and $\mathcal{D}$ an ultrafilter on $I$.  
\begin{enumerate}
\item For every set $C$, $\text{Av}(\mathcal{D},C)$ is a complete type over $C$.
\item The global type $\text{Av}(\mathcal{D},\mathbb{M})$ is $I$-invariant.
\item For any model $M \models T$, if $p \in S^{n}(M)$, there is some ultrafilter $\mathcal{E}$ on $M^{n}$ so that $p = \text{Av}(\mathcal{E},M)$.  
\end{enumerate}
\end{fact}

One important consequence of Fact \ref{average} for us is that every type over a model $M$ extends to a global $M$-invariant type:  given $p \in S(M)$, one chooses an ultrafilter $\mathcal{D}$ so that $\text{Av}(\mathcal{D},M) = p$.  Then $\text{Av}(\mathcal{D},\mathbb{M})$ is a global type extending $p$ which is $M$-invariant.  In the arguments below, it will often be convenient to produce global invariant types through a particular choice of ultrafilter.  

\begin{fact} \cite[Remark 2.16]{chernikov2012forking}
Write $a \ind^{u}_{A} b$ to mean that $\text{tp}(a/Ab)$ is finitely satisfiable in $A$ -- the $u$ is for ``ultrafilter" as this is equivalent to asserting $\text{tp}(a/Ab) = \text{Av}(\mathcal{D},Ab)$ for some ultrafilter $\mathcal{D}$ on $A$.  The relation $\ind^{u}$ satisfies both left and right extension over models:
\begin{enumerate}
\item (Left extension)  If $M$ is a model and $a \ind^{u}_{M} b$ then for all $d$, there is some $b' \equiv_{Ma} b$ so that $ad \ind^{u}_{M} b'$.  
\item (Right extension)  If $M$ is a model and $a \ind^{u}_{M} b$ then for all $c$, there is some $a' \equiv_{Mb} a$ so that $a' \ind^{u}_{M} bc$.  
\end{enumerate}
\end{fact}

\begin{defn}\label{indiscernibletype}
Suppose $M \models T$ and $\overline{a} = (a_{i})_{i < \omega}$ is an $M$-indiscernible sequence.  A global $M$-invariant type $q \supseteq \text{tp}(\overline{a}/M)$ is called an \emph{indiscernible type} if whenenver $\overline{a}' \models q$, $\overline{a}'$ is $\mathbb{M}$-indiscernible.  
\end{defn}

\begin{defn}
A collection of sequences $(\overline{a}_{\alpha})_{\alpha < \kappa}$ where $\overline{a}_{\alpha} = \langle a_{\alpha,i} : i < \lambda \rangle$ is called a \emph{mutually indiscernible array} over a set of parameters $C$ if, for each $\alpha < \kappa$, the sequence $\overline{a}_{\alpha}$ is an indiscernible sequence over $C\overline{a}_{\neq \alpha}$.  
\end{defn}

The following two lemmas are essentially \cite[Lemma 8]{adler2014kim}.  We include a proof for completeness.  

\begin{lem}
If $\overline{a} = (a_{i})_{i < \omega}$ is an $M$-indiscernible sequence, there is an indiscernible global $M$-invariant type $q \supseteq \text{tp}(\overline{a}/M)$.  
\end{lem}

\begin{proof}
Let $N$ be an $|M|^{+}$-saturated elementary extension of $M$ of size $\kappa$ and let $r \supseteq \text{tp}(\overline{a}/M)$ be an arbitrary global $M$-invariant type extending $\text{tp}(\overline{a}/M)$.  
Let $\overline{b} = (b_{i})_{i < \omega} \models r|_{N}$.  
By Ramsey and compactness, we may extract from $\overline{b}$ an $N$-indiscernible sequence $(c_{i})_{i < \omega}$.  Clearly $\text{tp}(\overline{c}/N)$ extends $\text{tp}(\overline{a}/M)$.  It is also $M$-invariant:  if not, there are $n \equiv_{M} n'$ in $N$, an increasing $k$-tuple $\overline{i}$ from $\omega$, and a formula $\varphi$ so that
$$
\models \varphi(c_{\overline{i}};n) \leftrightarrow \neg \varphi(c_{\overline{i}}, n').
$$
Then there is an increasing $k$-tuple $\overline{j}$ so that
$$
\models \varphi(b_{\overline{j}};n) \leftrightarrow \neg \varphi(b_{\overline{j}};n'),
$$
since the sequence $\overline{c}$ is extracted from $\overline{b}$.  This contradicts the fact that $\overline{b}$ realizes an $M$-invariant type over $N$.  
By Fact \ref{tensor}, the type $\text{tp}(\overline{c}/N)$ determines a unique $M$-invariant extension to $\mathbb{M}$.  Call it $q$.  Then $q$ is an indiscernible type.  
\end{proof}

\begin{lem}\label{pathtype}
Suppose $M \models T$, $\overline{a} = (a_{i})_{i < \omega}$ is an $M$-indiscernible sequence, and $q \supseteq \text{tp}(\overline{a}/M)$ is a global $M$-invariant indiscernible type.  Let $(\overline{a}_{i})_{i < \omega} \models q^{\otimes \omega}|_{M}$ with $\overline{a}_{0} = \overline{a}$, where $\overline{a}_{i} = (a_{i,j})_{j < \omega}$.  Then $(\overline{a}_{i})_{i < \omega}$ is a mutually indiscernible array over $M$.  
\end{lem}

\begin{proof}
We prove by induction on $n$ that $(\overline{a}_{i})_{i \leq n}$ is mututally indiscernible over $M$.  For $n=1$, there's nothing to prove.  Suppose it's been shown for $n$ and consider $(a_{i})_{i \leq n+1}$.  As $q$ is an indiscernible type, $\overline{a}_{n+1}$ is $M\overline{a}_{\leq n}$-indiscernible.  For $i \leq n$, we know, by induction, that $\overline{a}_{i}$ is $M\overline{a}_{< i}\overline{a}_{i+1}\ldots \overline{a}_{n}$-indiscernible.  As $\overline{a}_{n+1} \models q|_{M\overline{a}_{\leq n}}$, this entails $\overline{a}_{i}$ is indiscernible over $M\overline{a}_{<i}\overline{a}_{i+1}\ldots \overline{a}_{n+1}$, which completes the induction.  
\end{proof}

\subsection{Kim-dividing}

In this subsection, we define Kim-dividing and Kim-forking, the fundamental notions explored in this paper.  To start, we will need the definition of $q$-dividing, introduced by Hrushovski in \cite[Section 2.1]{hrushovski2012stable}:

\begin{defn}
Suppose $q(y)$ is an $A$-invariant global type.  The formula $\varphi(x;y)$ $q$-divides over $A$ if for some (equivalently, any) Morley sequence $\langle b_{i} : i < \omega \rangle$ in $q$ over $A$, $\{\varphi(x;b_{i}) : i < \omega \}$ is inconsistent.  
\end{defn}

We note that we will consistently use the letters $p,q,r$ to refer to types, $n,m,k,l$ to refer to numbers.  In this way, no confusion between $q$-dividing and the more familiar $k$-dividing will arise.  

The related notion of a \emph{higher formula} was introduced by Malliaris and Shelah in \cite{malliaris2015model} on the way to a new characterization of NTP$_{1}$ theories:  

\begin{defn} \cite[Definition 8.6]{malliaris2015model}
A \emph{higher formula} is a triple $(\varphi,A, \mathcal{D})$ where $\varphi = \varphi(x;y)$ is a formula, $A$ is a set of parameters, and $\mathcal{D}$ is an ultrafilter on $A^{l(y)}$ so that, if $q = \text{Av}(\mathcal{D},\mathbb{M})$ and $\langle b_{i} : i < \omega \rangle \models q^{\otimes \omega}|_{A}$ then $\{\varphi(x;b_{i}) : i < \omega\}$ is consistent.  
\end{defn}

We can rephrase the above definition as:  $(\varphi,A,\mathcal{D})$ is a higher formula if, setting $q = \text{Av}(\mathcal{D},\mathbb{M})$, $\varphi(x;y)$ does not $q$-divide over $A$.  

\begin{defn}
We say that a formula $\varphi(x;b)$ \emph{Kim-divides} over $A$ if there is some $A$-invariant global type $q \supseteq \text{tp}(b/A)$ so that $\varphi(x;y)$ $q$-divides.  The formula $\varphi(x;b)$ \emph{Kim-forks} over $A$ if $\varphi(x;b) \vdash \bigvee_{i<k} \psi_{i}(x;c^{i})$ and each $\psi_{i}(x;c^{i})$ Kim-divides over $A$.  A type Kim-forks if it implies a formula which does.  If $\text{tp}(a/Ab)$ does not Kim-fork over $A$, we write $a \ind^{K}_{A} b$.  
\end{defn}

We call this notion \emph{Kim}-dividing to make explicit the fact that this definition was inspired by a suggestion of Kim in his 2009 BIRS talk \cite{KimNTP1}, where he proposed an independence relation based on instances of dividing that are witnessed by every appropriate Morley sequence.  A rough connection between Kim's notion and ours is provided by Theorem \ref{kimslemmaforindk} below, which shows that, in an NSOP$_{1}$ theory, dividing with respect some invariant Morley sequence is equivalent to dividing with respect to all.  An even tighter connection is established by Theorem \ref{kimslemmaforforking}, which shows that we can drop the assumption that the Morley sequences are generated by an invariant type.  (We note that for technical reasons our notion is still different from Kim's -- the proposal of \cite{KimNTP1} forces a kind of base monotonicity and we do not).  

In general, we only know that a type over $A$ has a global $A$-invariant extension when $A$ is a model.  Thus, when working with Kim-independence below, we will restrict ourselves almost entirely to the case where the base is a model.  

The next two propositions explain how the notions of higher formula and $q$-dividing interact with SOP$_{1}$.  

\begin{prop}\label{higher}
Suppose $T$ has SOP$_{1}$.  Then there is a model $M\models T$, a formula $\varphi(x;b)$, and ultrafilters $\mathcal{D}_{0}, \mathcal{D}_{1}$ on $M$ with
$$
\text{Av}(\mathcal{D}_{0}, M) = \text{Av}(\mathcal{D}_{1},M) = \text{tp}(b/M),
$$
so that $(\varphi,M, \mathcal{D}_{0})$ is higher but $(\varphi,M, \mathcal{D}_{1})$ is not higher.  
\end{prop}

\begin{proof}
Fix a Skolemization $T^{\text{Sk}}$ of $M$.  As $T$ has SOP$_{1}$, there is, by Proposition \ref{arrayequivalent}, a formula $\varphi(x;y)$ and an array $(c_{i,j})_{i < \omega + 1, j < 2}$ such that 
\begin{enumerate}
\item $(\overline{c}_{i})_{i < \omega + 1}$ is an indiscernible sequence (with respect to the Skolemized language)
\item $c_{\omega,0} \equiv^{L^{Sk}}_{\overline{c}_{<\omega}} c_{\omega,1}$.
\item $\{\varphi(x;c_{i,0}) : i < \omega+1\}$ is consistent.
\item If $i < j$, then $\{\varphi(x;c_{i,1}), \varphi(x;c_{j,1})\}$ is inconsistent.
\end{enumerate}  
Put $M = \text{Sk}(\overline{c}_{<\omega})$.  For $j=0,1$, let $\mathcal{D}_{j}$ be any non-principal ultrafilter on $M$, concentrating on $\langle c_{i,j} : i < \omega \rangle$ and set $q_{j} = \text{Av}_{L}(\mathcal{D}_{j},\mathbb{M})$ for $j=0,1$.  Note that $q_{0}|_{M} = \text{tp}_{L}(c_{\omega,0}/M)= \text{tp}_{L}(c_{\omega,1}/M) = q_{1}|_{M}$ by (2).  By (3), $\varphi(x;y)$ does not $q_{0}$-divide, hence $(\varphi,M, \mathcal{D}_{0})$ is higher.  However, by (4), $\{\varphi(x;c_{j,1}) : j < \omega\}$ is $2$-inconsistent hence $\varphi(x;y)$ $q_{1}$-divides, so $(\varphi,M,\mathcal{D}_{1})$ is not higher.  
\end{proof}

\begin{prop}\label{qdiv}
Suppose $A$ is a set of parameters and $\varphi(x;b)$ is a formula which $q$-divides over $A$ for some global $A$-invariant type $q \supseteq \text{tp}(b/A)$.  If there is some global $A$-invariant $r \supseteq \text{tp}(b/A)$ such that $\varphi(x;y)$ does not $r$-divide, then $T$ has SOP$_{1}$.  
\end{prop}

\begin{proof}
As $\varphi(x;y)$ $q$-divides over $A$, there is $k$ so that instances of $\varphi(x;y)$ instantiated on a Morley sequence of $q$ are $k$-inconsistent.  

Let $(c_{i,1},c_{i,0})_{i \in \mathbb{Z}} \models (q \otimes r)^{\otimes \mathbb{Z}}|_{M}$.  We have to check that the sequence satisfies the following properties:
\begin{enumerate}
\item $\{\varphi(x;c_{i,0}) : i \in \mathbb{Z}\}$ is consistent
\item $\{\varphi(x;c_{i,1}) : i \in \mathbb{Z}\}$ is $k$-inconsistent
\item $c_{i,0} \equiv_{\overline{c}_{>i}} c_{i,1}$ for all $i \in \mathbb{Z}$.
\end{enumerate}
Note that $(c_{i,0})_{i \in \mathbb{Z}} \models r^{\otimes \mathbb{Z}}|_{M}$ so (1) follows from our assumption that $\varphi(x;y)$ does not $r$-divide.  Likewise, $(c_{i,1})_{i \in \mathbb{Z}} \models q^{\otimes \mathbb{Z}}|_{M}$ so (2) follows from the fact that $\varphi(x,y)$ $q$-divides.  Finally, for any $i \in \mathbb{Z}$, we have $\overline{c}_{>i}$ realizes a global $M$-invariant type over $Mc_{i,0}c_{i,1}$.  Hence (3) follows from the fact that $c_{i,0} \equiv_{M} c_{i,1}$.  
\end{proof}

\begin{thm}\label{kimslemmaforindk}
The following are equivalent for the complete theory $T$:
\begin{enumerate}
\item $T$ is NSOP$_{1}$
\item Ultrafilter independence of higher formulas:  for every model $M \models T$, and ultrafilters $\mathcal{D}$ and $\mathcal{E}$ on $M$ with $\text{Av}(\mathcal{D},M) = \text{Av}(\mathcal{E},M)$, $(\varphi,M, \mathcal{D})$ is higher if and only if $(\varphi, M,\mathcal{E})$ is higher. 
\item Kim's lemma for Kim-dividing:  For every model $M \models T$ and $\varphi(x;b)$, if $\varphi(x;y)$ $q$-divides for some global $M$-invariant $q \supseteq \text{tp}(b/M)$, then $\varphi(x;y)$ $q$-divides for \emph{every} global $M$-invariant $q \supseteq \text{tp}(b/M)$.  
\end{enumerate}
\end{thm}

\begin{proof}
(1)$\implies$(3) is the contrapositive of Proposition \ref{qdiv}.  

(2)$\implies$(1) is the contrapositive of Proposition \ref{higher}.

(3)$\implies$(2):  Immediate, since every type finitely satisfiable in $M$ is $M$-invariant.  
\end{proof}

\begin{rem}
Note that the proof gives a bit more:  if $T$ is NSOP$_{1}$, (2) is true over arbitrary sets and (3) is true over an arbitrary set $A$ as well, though this may be vacuous if $\text{tp}(b/A)$ does not extend to a global $A$-invariant type.  
\end{rem}

\subsection{The basic properties of Kim-independence}

Theorem \ref{kimslemmaforindk}, a kind of Kim's lemma for Kim-dividing, already gives a powerful tool for proving that in NSOP$_{1}$ theories Kim-independence enjoys many of the properties known to hold for non-forking independence in simple theories.  

We will frequently use the following easy observation.  The proof is exactly as in the case of dividing.  See, e.g., \cite[Lemma 1.5]{grossberg2002primer} or \cite[Lemma 1.4]{shelah1980simple}.

\begin{lem} \label{basiccharacterization}{(Basic Characterization of Kim-dividing)}
Suppose $T$ is an arbitrary complete theory.  The following are equivalent:
\begin{enumerate}
\item $\text{tp}(a/Ab)$ does not Kim-divide over $A$.  
\item For any global $A$-invariant $q \supseteq \text{tp}(b/A)$ and $I = \langle b_{i} : i < \omega \rangle \models q^{\otimes \omega}|_{A}$ with $b_{0} = b$, there is $a' \equiv_{Ab} a$ such that $I$ is $Aa'$-indiscernible.  
\item For any global $A$-invariant $q \supseteq \text{tp}(b/A)$ and $I = \langle b_{i} : i < \omega \rangle \models q^{\otimes \omega}|_{A}$ with $b_{0} = b$, there is $I' \equiv_{Ab} I$ such that $I'$ is $Aa$-indiscernible.  
\end{enumerate}
\end{lem}

Note that in an NSOP$_{1}$ theory, by Kim's Lemma for Kim-dividing, we could have replaced (2) by:  \emph{there is} a global $A$-invariant $q \supseteq \text{tp}(b/A)$ and $I = \langle b_{i} : i < \omega \rangle \models q^{\otimes \omega}|_{A}$ with $b_{0} = b$, so that for some $a' \equiv_{Ab} a$ $I$ is $Aa'$-indiscernible (and similarly for (3)), provided $\text{tp}(b/A)$ extends to a global $A$-invariant type.    

The following proposition is proved by the same argument one uses to prove forking = dividing  via Kim's lemma, as in \cite[Theorem 2.5]{grossberg2002primer} or \cite[Corollary 3.16]{chernikov2012forking}.

\begin{prop}\label{kforkingequalskdividing}{(Kim-forking = Kim-dividing)}
Suppose $T$ is NSOP$_{1}$.  If $M \models T$, if $\varphi(x;b)$ Kim-forks over $M$ then $\varphi(x;b)$ Kim-divides over $M$.
\end{prop}

\begin{proof}
Suppose $\varphi(x;b) \vdash \bigvee_{j< k} \psi_{j}(x;c^{j})$ where each $\psi_{i}(x;c^{i})$ Kim-divides over $M$.  Fix an ultrafilter $\mathcal{D}$ on $M$ so that $(b,c^{0}, \ldots, c^{k-1}) \models \text{Av}(\mathcal{D}, M)$.  Let $(b_{i},c_{i}^{0}, \ldots, c^{k-1}_{i})_{i < \omega}$ be a Morley sequence in $\text{Av}(\mathcal{D},\mathbb{M})$.  Then $(b_{i})_{i < \omega}$ is an $M$-invariant Morley sequence.  We must show $\{\varphi(x;b_{i}) : i < \omega \}$ is inconsistent.  Suppose not -- let $a \models \{\varphi(x;b_{i}) : i < \omega\}$.  We have $\varphi(x;b_{i}) \vdash \bigvee_{j < k} \psi_{j}(x;c^{j}_{i})$ so for each $i < \omega$, there is $j(i) < k$ so that $\models \psi_{j(i)}(a;c^{j(i)}_{i})$.  By the pigeonhole principle, there is $j_{*} < k$ so that $X = \{i < \omega : j(i) = j_{*}\}$ is infinite.  Then $(c^{j_{*}}_{i})_{i \in X}$ is an $M$-invariant Morley sequence in $\text{tp}(c^{j_{*}}/M)$.  As $T$ is NSOP$_{1}$, Kim-dividing over $M$ is witnessed by any $M$-invariant Morley sequence so $\{\psi_{j_{*}}(x;c^{j_{*}}_{i}) : i \in X\}$ is inconsistent.  But $a \models \{\psi_{j_{*}}(x;c^{j_{*}}_{i}) : i \in X\}$, a contradiction.  
\end{proof}

\begin{prop}\label{extension}{(Extension over Models)}
Suppose $M$ is a model, and $a \ind^{K}_{M} b$.  Then for any $c$, there is $a' \equiv_{Mb} a$ so that $a' \ind^{K}_{M} bc$.
\end{prop}

\begin{proof}
This is exactly as in the usual proof that forking satisfies extension.  Let $p(x;b) = \text{tp}(a/Mb)$.  We claim that the following set of formulas is consistent:
$$
p(x;b) \cup \{ \neg \psi(x;b,c) : \psi(x;b,c) \in L(Mbc) \text{ and } \psi(x;b,c) \text{ Kim-divides over } M\}.
$$
If this set of formulas is not consistent, then by compactness, 
$$
p(x;b) \vdash \bigvee_{i < k} \psi_{i}(x;b,c_{i}),
$$
where each $\psi_{i}(x;b,c_{i})$ Kim-divides over $M$.  It follows that $\text{tp}(a/Mb)$ Kim-forks over $M$, a contradiction.  So this set is consistent and we may choose a realization $a'$.  Then $a' \ind^{K}_{M} bc$ and $a' \equiv_{Mb} a$.  
\end{proof}

\begin{prop}\label{chaincondition}{(Chain Condition for Invariant Morley Sequences)}
Suppose $T$ is NSOP$_{1}$ and $M \models T$.  If $a \ind^{K}_{M} b$ and $q \supseteq \text{tp}(b/M)$ is a global $M$-invariant type, then for any $I = \langle b_{i} : i < \omega \rangle \models q^{\otimes \omega}|_{M}$ with $b = b_{0}$, there is $a' \equiv_{Mb} a$ so that $a' \ind_{M} I$ and $I$ is $Ma'$-indiscernible.  
\end{prop}

\begin{proof}
By the basic characterization of Kim-dividing, Lemma \ref{basiccharacterization}, given $a \ind^{K}_{M} b$, $q \supseteq \text{tp}(b/M)$ a global $M$-invariant type, and $I = \langle b_{i} : i < \omega \rangle \models q^{\otimes \omega}|_{M}$ with $b = b_{0}$, there is $a' \equiv_{Mb} a$ so that $I$ is $Ma'$-indiscernible.  To prove the proposition it suffices to show $a' \ind^{K}_{M} b_{<n}$ for all $n$.    Given $n < \omega$, let $r(x;y_{0}, \ldots, y_{n-1})  = \text{tp}(a'; b_{0}, \ldots, b_{n-1}/M)$.  Then $\langle (b_{kn + n-1}, b_{kn + n-2}, \ldots, b_{kn}) : k < \omega \rangle \models (q^{\otimes n})^{\otimes \omega}|_{M}$ and, by indiscernibility, 
$$
a' \models \bigcup_{k < \omega} r(x;b_{kn + n-1}, b_{kn + n-2}, \ldots, b_{kn}).
$$
As $T$ is NSOP$_{1}$, this shows $a' \ind^{K}_{M} b_{<n}$.  
\end{proof}

Section \ref{symmetrysection} will be dedicated to the proof that $\ind^{K}$ is symmetric in NSOP$_{1}$ theories.  The argument will require more tools, but at this stage we can already observe the converse: even a weak form of symmetry for $\ind^{K}$ will imply that a theory is NSOP$_{1}$.  

\begin{prop}\label{weaksymmetrylemma}
The following are equivalent for a complete theory $T$:
\begin{enumerate}
\item $T$ is NSOP$_{1}$.
\item Weak symmetry:  if $M \models T$, then  $b \ind^{i}_{M} a \implies a \ind^{K}_{M} b$.   
\end{enumerate}
\end{prop}

\begin{proof}
(1)$\implies$(2).  Suppose $T$ is NSOP$_{1}$.  As $b \ind^{i}_{M} a$, there is a global $M$-invariant type $r \supseteq \text{tp}(b/Ma)$.  We can find a Morley sequence $I = \langle b_{i} : i < \omega \rangle$ in $q|_{Ma}$ with $b_{0} = b$.  Then $I$ is $Ma$-indiscernible, so no formula in $\text{tp}(a/Mb)$ divides with respect to the sequence $I$.  But by Kim's lemma for Kim-dividing, this implies $a \ind^{K}_{M} b$.  

(2)$\implies$(1).  Suppose $T$ has SOP$_{1}$.  Then by \cite[Theorem 5.1]{ArtemNick}, there is a model $M \models T$, $a_{0}b_{0} \equiv_{M} a_{1}b_{1}$ with $b_{i} \ind^{i}_{M} a_{i}$ and $b_{1} \ind^{i}_{M} b_{0}$, but, setting $p(x;b_{0}) = \text{tp}(a_{0}/Mb_{0})$, we have $p(x;b_{0}) \cup p(x;b_{1})$ is inconsistent.  As $b_{0} \equiv_{M} b_{1}$ and $b_{1} \ind^{i}_{M} b_{0}$, $(b_{0},b_{1})$ starts a Morley sequence in some $M$-invariant type, $\langle b_{i} : i < \omega \rangle$.  As $\bigcup_{i < \omega} p(x;b_{i})$ is inconsistent, we have $a_{0} \nind^{K}_{M} b_{0}$.  Since $b_{0} \ind^{i}_{M} a_{0}$, weak symmetry fails.  
\end{proof}

\section{Local character} \label{localcharsection}

In this section, we prove local character for $\ind^{K}$ in NSOP$_{1}$ theories:  for every NSOP$_{1}$ theory, there is a cardinal $\kappa$ so that, given a model $M \models T$ and a type $p \in S(M)$, there is an elementary submodel $M' \preceq M$ of size $<\kappa$ such that $p$ does not Kim-fork over $M'$.  We give a simple and soft argument showing first that $\kappa$ can be taken to be the first measurable cardinal above $|T|$.  Then, in a more difficult argument, we show that $\kappa$ can be taken to be $(2^{|T|})^{+}$.  The argument involving large cardinals is, of course, implied by the stronger result, but we thought that the conceptual simplicity of the first argument might be helpful in understanding the second.  Lastly, we show that for any regular $\kappa$, we can construct a model which satisfies local character\textemdash this clarifies the situation for cardinals between $|T|$ and $2^{|T|}$.    

In order to prove our first theorem, we will use the following facts about measurable cardinals:

\begin{fact}
\label{fact:indiscernibles exist} \cite[Theorem 7.17]{MR1994835}
Suppose that $\mu>\left|T\right|$ is a measurable cardinal and that
$\mathcal{U}$ is a normal (non-principal) ultrafilter on $\mu$. Suppose
that $(a_{i})_{i<\mu}$ is a sequence of finite tuples
in $\mathbb{M}$, then for some set $X\in \mathcal{U}$, $(a_{i})_{i\in X}$
is an indiscernible sequence.
\end{fact}

\begin{fact}
\label{fac:full-indiscernibles}\cite[Fact 2.9]{18-Kaplan2014a} If
$A=\bigcup_{i<\mu}A_{i}\subseteq \mathbb{M}$ is a continuous increasing union
of sets where $\left|A_{i}\right|<\mu$, $B\subseteq \mathbb{M}$ is some set
of cardinality $<\mu$, and $(a_{i})_{i<\mu}$, $\mathcal{U}$
are as in Fact \ref{fact:indiscernibles exist} with $a_{i}$
tuples from $A$, then for some set $X\in \mathcal{U}$, $(a_{i})_{i\in X}$
is \textup{fully indiscernible over $B$} (with respect to $A$ and
$(A_{i})_{i<\mu}$), which means that for every $i\in X$
and $j<i$ in $X$, we have $a_{j}\subseteq A_{i}$ , and $(a_{j})_{i\leq j\in X}$
is indiscernible over $A_{i}\cup B$. 
\end{fact}

\begin{thm}
Suppose that $T$ is NSOP$_1$ and that $\left|T\right|<\mu$ is measurable.
Suppose that $M\models T$.  Then for every $p\in S\left(M\right)$
there is a model $N\prec M$ with $\left|N\right|<\mu$ such that
$p$ does not Kim-fork over $N$. 
\end{thm}

\begin{proof}
Suppose not.  Construct by induction on $i<\mu$ a sequence $(\left(M_{i},b_{i},b_{i}',\varphi_{i}\right))_{i<\mu}$
such that:
\begin{itemize}
\item $(M_{i})_{i<\mu}$ is an increasing continuous sequence of
models.
\item For $i<\mu$, $\varphi_{i}\left(x,y_{i},y_{i}'\right)$ is a formula
in $L$. 
\item $b_{i}'\in M_{i}$ and $b_{i}\in M\setminus M_{i}$. 
\item $\varphi\left(x,b_{i},b_{i}'\right)\in p$ witnesses Kim-dividing
of $p$ over $M_{i}$.
\item For $i<\mu$, $M_{i}$ is some model containing $\{b_{j}: j<i\}$
of size $\left|T\right|+\left|i\right|$. 
\end{itemize}
We can construct such a sequence by our assumption. 

Note that all clubs $E\subseteq\mu$ are in $\mathcal{U}$ (see the proof
of Fact \ref{fac:full-indiscernibles} in \cite[Fact 2.9]{18-Kaplan2014a}). 

By Fodor's lemma for normal ultrafilters \cite[Exercise 5.10]{MR1994835},
applied to the function $F:\lim\left(\mu\right)\to\mu$ such that
$F\left(\delta\right)=\min\{i<\delta : b_{\delta}'\in M_{i}\}$,
there is some $X_{0}\in \mathcal{U}$ (consisting of limit ordinals) such
that for all $i\in X_{0}$, $b_{i}'=b'$ is constant. For convenience
assume that $b'\in M_{0}$.

By Fact \ref{fac:full-indiscernibles}, there is some $X_{1}\subseteq X_{0}$
in $\mathcal{U}$ such that $(b_{i})_{i\in X_{1}}$ is fully indiscernible
with respect to $(M_{i})_{i<\mu}$. 

Let $X_{2}\subseteq X_{1}$ in $\mathcal{U}$ be such that for all $i\in X_{2}$
if $j<i$ then there is $j<\alpha<i$ such that $\alpha\in X_{1}$
(as $X_{1}$ is unbounded, the set of all $i<\mu$ such that for all
$j<i$ there is such an $\alpha$ is a club, so in $\mathcal{U}$). 

Let $\alpha_{0}\in X_{2}$. Then $(b_{i})_{i\geq\alpha_{0},i\in X_{2}}$
is an indiscernible sequence such that $b_{i}\ind_{M_{\alpha_{0}}}^{u}b_{>i}$
for all $\alpha_{0}\leq i\in X_{2}$. To see this, suppose that $\psi\left(b_{i},b_{>i},m\right)$
holds where $m\in M_{\alpha_{0}}$. Then $m\in M_{\beta}$ for some
$\beta<\alpha_{0}$ (as $\alpha_{0}$ is limit). So by definition
of $X_{2}$, there is some $\beta<j<\alpha_{0}$ in $X_{1}$. But
then $(b_{\varepsilon})_{\varepsilon\geq j}$ is indiscernible
over $M_{\beta}$ by choice of $X_{1}$, so $\psi\left(b_{j},b_{>i},m\right)$
holds as well. But $b_{j}\in M_{\alpha_{0}}$ by construction. 

We get a contradiction, since $\{\varphi\left(x,b_{i},b'\right) : i\in X_{2},i\geq\alpha_{0}\}$
is in $p$ but also inconsistent since $\varphi\left(x,b_{\alpha_{0}},b'\right)$
Kim-divides over $M_{\alpha_{0}}$. 
\end{proof}

\begin{lem} \label{lem:countable subset for dividing}
Suppose that $N$ is some model and that $p\in S\left(\mathbb{M}\right)$ is a global type finitely satisfiable in $N$ which extends $\text{tp}\left(c/N\right)$. Given any set $A \subseteq N$, there is some $B\prec N$ of size $\leq\left|T\right|+\left|A\right|$ such that $A\subseteq B$ and $p^{\otimes \omega}|_{B}$ is a type of a Morley sequence generated by some global type finitely satisfiable in $B$. 

In particular, if $\varphi\left(x,c\right)$ Kim-divides over $N$ then $\varphi\left(x,c\right)$ Kim-divides over $B$.
\end{lem}
\begin{proof}
Let $p\in S\left(\mathbb{M}\right)$ be a global type extending $\text{tp}\left(c/N\right)$, finitely satisfiable in $N$. 

Let $B_{0}$ be any model containing $A$ of size $\leq\left|A\right|+\left|T\right|$, and let $\bar{c}\models p^{\otimes \omega}|_{N}$. Let $N\supseteq B_{1}'\supseteq B_{0}$
be such that for every $n<\omega$ and every formula $\psi\left(y,c_{<n}\right)$
over $B_{0}$, if $\mathbb{M} \models\psi\left(c_{n},c_{<n}\right)$ then $\psi\left(y,c_{<n}\right)$
is satisfiable in $B_{1}'$ and let $B_{1}$ be any model containing
$B_{1}'$ of size $\left|T\right|+\left|A\right|$. Continue like this, and finally, let $B=\bigcup_{i<\omega}B_{i}$. Then $\bar{c}$ is still a Morley sequence sequence over $B$ in a $B$-finitely satisfiable type (note that it is indiscernible). 
\end{proof}
\begin{thm} \label{localcharacter}
Suppose $T$ is NSOP$_{1}$.  Then for any $M \models T$ and $p \in S(M)$, there is $M' \preceq M$ so that $p$ does not Kim-fork over $M'$ and $|M'| \leq 2^{|T|}$.
\end{thm}

\begin{proof}
Let $\kappa = (2^{|T|})^{+}$\textemdash  $\kappa$ is a regular cardinal, greater than $2^{|T|}$, and $\mu < \kappa$ implies $\mu^{|T|} < \kappa$ (these are the only properties of $\kappa$ we will use). 

Suppose not.  Then there is some $p\in S\left(M\right)$ witnessing
this.  Clearly $|M| \geq \kappa$.  For every $i<\kappa$ we can find $c_{i}$, $d_{i}$, $N_{i}$,
and $\varphi_{i}\left(x,y_{i},z_{i}\right)$ such that:
\begin{itemize}
\item $c_{i}\in N_{i+1} \setminus N_{i}$, $d_{i}\in N_{i}$, $\langle N_{i}: i<\kappa \rangle$ is
increasing continuous, $\left|N_{i}\right|\leq\left|T\right|+\left|i\right|<\kappa$,
$\varphi_{i}\left(x,y_{i},z_{i}\right)$ is a formula such that $\varphi_{i}\left(x,c_{i},d_{i}\right)$
Kim-divides over $N_{i}$ and is in $p$.
\end{itemize}
Let $S$ be $\{\delta<\kappa : \text{cof}\left(\delta\right)=\left|T\right|^{+}\}$.
Then $S$ is a stationary set. 

For every $\delta\in S$, fix some global coheir $q_{\delta}\in S\left(\mathbb{M}\right)$ over $N_{\delta}$ extending $\text{tp}\left(c_{\delta}/N_{\delta}\right)$.  Given a partition of a stationary subset of $\kappa$ into $<\kappa$ parts, one of these has to be a stationary set. Hence, we may assume that for every $\delta\in S$, $\varphi_{\delta}=\varphi$ and $\varphi\left(x,c_{i},d_{i}\right)$
is $k$-Kim-dividing for some fixed $k$, witnessed by any Morley sequence in $q_{\delta}$. Define the regressive function $f:S\to\kappa$ by $f(\delta) = \min\{i < \delta : d_{\delta} \in N_{i}\}$ (this set is non-empty by continuity of the sequence).  By Fodor's lemma, we may assume that $f$ is constant on $S$, and further restricting it, we may even assume that $d_{\delta}=d$
is fixed for every $\delta\in S$. This allows us to assume for simplicity that $d=\emptyset$.  

By Lemma \ref{lem:countable subset for dividing}, for every $\delta\in S$
there is some $M_{\delta}\prec N_{\delta}$ of size $\left|T\right|$
such that $\varphi\left(x,c_{\delta}\right)$ Kim-divides over $M_{\delta}$,
and moreover, such that $q_{\delta}^{\otimes \omega}|_{M_{\delta}}$
is a type of a Morley sequence of some global coheir $r_{\delta}$
over $M_{\delta}$. 

As $\text{cof}\left(\delta\right)=\left|T\right|^{+}$ for every $\delta\in S$,
for each such $\delta$ there is some $i<\delta$ such that $M_{\delta}\prec N_{i}$.
Hence by Fodor's lemma, there is some $i<\kappa$ and a stationary
$S'\subseteq S$ such that for every $\delta\in S'$, $M_{\delta}\prec N_{i}$.
Then we can find some model $M_{0}^{*}$,
a global coheir $r_{0}^{*}$ over $M_{0}^{*}$ and a stationary $S_{0}\subseteq S'$
such that for every $\delta\in S_{0}$, $M_{\delta}=M_{0}^{*}$ (note
that the number of possible $M_{\delta}$'s is $\leq\left|N_{i}\right|^{\left|T\right|}<\kappa$)
and $q_{\delta}^{\otimes \omega}|_{M_{0}^{*}}=r_{\delta}^{\otimes \omega}|_{M_{0}^{*}}=r_{0}^{*\otimes \omega }|_{M_{0}^{*}}$
(the number of $\omega$-types over $M_{0}^{*}$ is $\leq2^{\left|M_{0}^{*}\right|}<\kappa$
as $\left|M_{0}^{*}\right|=\left|T\right|$).

Let $\delta_{0}=\min S_{0}$ and $e_{0}=c_{\delta_{0}}$. 

By Lemma \ref{lem:countable subset for dividing}, for every $\delta\in S_{0}\backslash\left\{ \delta_{0}\right\} $
there is some $M_{0}^{*}c_{\delta_{0}}\subseteq M_{\delta}\prec N_{\delta}$
of size $\left|T\right|$ such that $\varphi\left(x,c_{\delta}\right)$
Kim-divides over $M_{\delta}$, and moreover, such that $q_{\delta}^{\otimes \omega}|_{M_{\delta}}$
is a type of a Morley sequence of some global coheir over $M_{\delta}$.
Thus, as above, we can find some stationary $S_{1}\subseteq S_{0}\backslash\left\{ \delta_{0}\right\} $,
$M_{1}^{*}$ and $r_{1}^{*}$ such that for every $\delta\in S_{1}$,
$M_{\delta}=M_{1}^{*}$ and $q_{\delta}^{\otimes \omega}|_{M_{1}^{*}}=r_{1}^{*\otimes \omega}|_{M_{1}^{*}}$.
Let $\delta_{1}=\min S_{1}$ and $e_{1}=c_{\delta_{1}}$. 

Continuing like this we find and increasing sequence $\langle \delta_{i}: i<\omega \rangle$
of ordinals in $\kappa$, an increasing sequence of models $\langle M_{i}^{*} :i<\omega \rangle$,
$e_{i}\in M$ for $i<\omega$ and global coheirs (over $M_{i}^{*}$)
$r_{i}^{*}$ such that:
\begin{itemize}
\item $M_{i}^{*}$ contains $e_{<i}$, $\varphi\left(x,e_{j}\right)$ is
$k$-Kim-dividing over $M_{i}^{*}$ for every $i<j$, $r_{i}^{*}$
is a global coheir over $M_{i}^{*}$ such that for all $i\leq j$,
$r_{i}^{*}$ extends $\text{tp}\left(e_{j}/M_{i}^{*}\right)$ (in particular,
$e_{j}\equiv_{M_{i}^{*}}e_{i}$ for all $j\geq i$) and $r_{j}^{*\otimes \omega}|_{M_{i}^{*}}=r_{i}^{*\otimes \omega}|_{M_{i}^{*}}$. 
\end{itemize}
Denote $\overline{e} = \langle e_{i} : i < \omega \rangle$.  Note that $\{\varphi(x;e_{i}) : i < \omega\}$ is a subset of $p$, hence consistent.

\textbf{Claim}:  Suppose $i_{0} < \ldots < i_{n-1} < \omega$ and for each $j < n$, $f_{j} \models r^{*}_{i_{j}}|_{M^{*}_{i_{j}}\overline{e}f_{>j}}$.  Then 
\begin{enumerate}
\item $e_{i_{j}} \equiv_{e_{i_{<j}}f_{<j}} f_{j}$ for all $j < n$
\item $\{\varphi(x;f_{j}) : j < n\}$ is $k$-inconsistent.  
\end{enumerate}

\emph{Proof of claim}:  
By induction on $n$, we prove that if $i_{0} < \ldots < i_{n-1} < \omega$, then $e_{i_{j}} \equiv_{M^{*}_{i_{0}}e_{<i_{j}} f_{<j}} f_{j}$. For $n=0$ there is nothing to prove. Suppose the claim is true for $n$ and we are given $i_{0} < \ldots < i_{n}$ and $(f_{j})_{j < n+1}$ with $f_{j} \models r^{*}_{i_{j}}|_{M^{*}_{i_{j}}\overline{e}f_{>j}}$ for all $j < n+1$.  Then clearly $f_{0}\equiv_{M_{i_{0}}^{*}}e_{i_{0}}$. For $0<j<n+1$, by induction $f_{j}\equiv_{M_{i_{1}}^{*}e_{<i_{j}}f_{\in\left[1,j\right)}}e_{i_{j}}$,
hence $f_{j}\equiv_{M_{i_{0}}^{*}e_{<i_{j}}f_{\in\left[1,j\right)}}e_{i_{j}}$.
As $f_{0}\ind_{M_{i_{0}}^{*}}^{u}\overline{e}f_{>0}$, we get that $f_{j}\equiv_{M_{i_{0}}^{*}e_{<i_{j}}f_{<j}}e_{i_{j}}$.  This shows (1).  To see (2), note that $(f_{n-1},f_{n-2},\ldots, f_{0}) \models r_{i_{0}}^{\otimes n}|M_{i_{0}}^{*} = q_{\delta_{i_{0}}}^{\otimes n}|M^{*}_{0}$, hence $\{\varphi(x;f_{j}) : j < n\}$ is $k$-inconsistent, by our assumption that $\varphi(x;c_{\delta_{0}})$ $k$-Kim-divides with respect to Morley sequences in $q_{\delta_{0}}$.  \qed

By compactness, we can find an array $(c_{i,0}, c_{i,1})_{i < \omega}$ so that $\{\varphi(x;c_{i,0}) : i < \omega\}$ is consistent, $\{\varphi(x;c_{i,1}) : i < \omega\}$ is $k$-inconsistent, and $c_{i,0} \equiv_{\overline{c}_{<i}} c_{i,1}$ for all $i < \omega$.  By Lemma \ref{karyversion}, we obtain SOP$_{1}$, a contradiction.   
\end{proof}

\begin{cor} \label{boundedweight}
Suppose that $T$ is a complete theory. The following are equivalent.
\begin{enumerate}
\item For some uncountable cardinal $\kappa$, there is no sequence $\langle N_{i},\varphi_{i}\left(x,y_{i}\right),c_{i} :i<\kappa \rangle$
such that $\langle N_{i} :i<\kappa \rangle$ is an increasing continuous
sequence of models of $T$ of size $<\kappa$, $\varphi_{i}(x,y_{i})$ is
a formula over $N_{i}$, $c_{i}\in N_{i+1}$, such that $\varphi_{i}(x,c_{i})$
Kim-forks over $N_{i}$ and $\{\varphi(x,c_{i}):i<\kappa\}$
is consistent. 
\item $T$ is NSOP$_{1}$. 
\end{enumerate}
\end{cor}
\begin{proof}
(2) implies (1) by the proof of Theorem \ref{localcharacter}
(with $\kappa=(2^{\left|T\right|})^{+}$).

(1) implies (2). This is a variation on the proof of Proposition \ref{higher}.  Suppose $T$ has SOP$_{1}$ as witnessed by some formula $\varphi(x,y)$. Let $T^{sk}$ be a Skolemized expansion of $T$. Then $T^{sk}$ also has SOP$_{1}$ as witnessed by $\varphi(x,y)$. Thus by Proposition 2.4, we can find
a formula $\varphi(x,y)$ and an array $(c_{i,j})_{i<\omega,j<2}$ such that $c_{i,0}\equiv_{\overline{c}_{<i}}c_{i,1}$ for all $i<\omega$, $\{\varphi(x,c_{i,0}): i<\omega\}$ is consistent and $\{\varphi(x,c_{i,1}): i<\omega\}$ is 2-inconsistent (all in $\mathbb{M}^{sk}$).  By Ramsey and compactness we may assume that $\langle \overline{c}_{i}: i<\omega \rangle$ is indiscernible (with respect to $\mathbb{M}^{sk}$).
Extend this sequence to one of length $\kappa$. 

For $i<\kappa$, let $N_{i}=\text{dcl}(\overline{c}_{<i})$ (in $\mathbb{M}^{sk}$). Then for every limit ordinal $\delta<\kappa$, $\varphi(x,c_{\delta,1})$ Kim-divides over $N_{\delta}$ as the sequence $\langle c_{j,1} : \delta\leq j<\kappa\rangle$ is indiscernible and for all $\delta\leq j$, $\overline{c}_{j}\ind_{N_{\delta}}^{u}\overline{c}_{>j}$.  As $c_{\delta,1}\equiv_{\bar{c}_{<\delta}}c_{\delta,0}$, it follows
that $c_{\delta,1}\equiv_{N_{\delta}}c_{\delta,0}$, and hence $\varphi\left(x,c_{\delta,0}\right)$
also Kim-divides. As $\kappa$ is uncountable, $\text{otp}\left(\text{lim}\left(\kappa\right)\right)=\kappa$, so $\langle N_{\delta},\varphi(x,y),c_{\delta,0}: \delta<\kappa \rangle$ contradicts (1). 
\end{proof}

\begin{quest}
Suppose $T$ is NSOP$_{1}$.  Must it be the case that if $M \models T$, and $p \in S(M)$, there is $M' \preceq M$ so that $p$ does not Kim-fork over $M'$ and $|M'| \leq |T|$?
\end{quest}

\begin{thm}
Suppose that $T$ is NSOP$_{1}$. Then for every regular cardinal $\kappa>\left|T\right|$
there is a model $M$ of size $\kappa$ such that for all $p\in S\left(M\right)$
there is $N\prec M$ with $\left|N\right|<\kappa$ such that $p$
does not Kim-fork over $N$. 
\end{thm}
\begin{proof}
Let $I=(a_{i})_{i<\kappa}$ be an indiscernible sequence with respect to $T^{sk}$ \textemdash{} a Skolemized expansion of $T$.
Let $M=\text{dcl}\left(I\right)$. Let $p\in S\left(M\right)$. For $i<\kappa$
let 
\[
N_{i}=\text{dcl}\left((a_{j})_{j<i}\right).
\]
Suppose for contradiction that for every $i<\kappa$, $p$ Kim-forks
over $N_{i}$. 

This means that for every $i<\kappa$ there is a formula $\varphi_{i}(x,t_{i}(b_{i},b_{i}'))$
witnessing Kim-dividing over $N_{i}$, where $t_{i}$ is a Skolem
term, $b_{i}\subseteq \{a_{j} :j\geq i\}$, $b_{i}'\subseteq \{a_{j} :j<i\}$,
and both are increasing tuples. 

Let $E\subseteq\kappa$ be the set of limits $\alpha\in\kappa$ such
that for all $i<\alpha$, $b_{i}\subseteq(a_{j})_{j<\alpha}$.
Then $E$ is a club of $\kappa$. Define $F:E\to\kappa$ by $F\left(\alpha\right)=\max\{ j : a_{j}\in b_{\alpha}'\}$.
By Fodor's lemma there is a stationary set $S\subseteq E$ on which
$F$ is constant $\gamma$. Reducing to an unbounded subset of $S$,
we may assume that for every $\alpha\in S$, $\varphi_{\alpha}=\varphi$,
$t_{\alpha}=t$ and $b_{\alpha}'=b'$ (all the $b_{\alpha}'$
come from $\{a_{j} : j<\gamma\}$ which has size $\left|\gamma\right|<\left|S\right|=\kappa$).
By choice of $E$, for all $\alpha<$$\beta$ from $S$, $b_{\alpha}< b_{\beta}$
(i.e., every coordinate of $b_{\beta}$ is greater than every
coordinate of $b_{\alpha}$). Hence $(t(b_{\alpha},b'))_{\alpha_{0}\leq\alpha\in S}$
is an indiscernible sequence over $N_{\alpha_{0}}$ such that $t\left(b_{\alpha},b'\right)\ind_{N_{\alpha_{0}}}^{u}t\left(b_{>\alpha},b'\right)$
for every $\alpha_{0}\in S$ by the construction of $N_{\alpha_{0}}$,
and hence, as $\varphi(x,t(b_{\alpha_{0}},b'))$
Kim-divides over $N_{\alpha_{0}}$, $\{\varphi(x,t(b_{\alpha},b')) :\alpha_{0}\leq\alpha\in S\}$
in inconsistent, but it is contained in $p$. 
\end{proof}

\begin{rem}
Note that this theorem is most interesting for the case $|T| < \kappa \leq 2^{|T|}$, as this is not covered by Theorem \ref{localcharacter}.  
\end{rem}

\section{Symmetry}\label{symmetrysection}

\subsection{Generalized indiscernibles and a class of trees}

For an ordinal $\alpha$, let the language $L_{s,\alpha}$ be $\langle \unlhd, \wedge, <_{lex}, (P_{\beta})_{\beta < \alpha} \rangle$.  We may view a tree with $\alpha$ levels as an $L_{s,\alpha}$-structure by interpreting $\unlhd$ as the tree partial order, $\wedge$ as the binary meet function, $<_{lex}$ as the lexicographic order, and $P_{\beta}$ interpreted to define level $\beta$.  For the rest of the paper, a tree will be understood to be an $L_{s,\alpha}$-structure for some appropriate $\alpha$.  We will sometimes suppress the $\alpha$ and refer instead to $L_{s}$, where the number of predicates is understood from context.  We define a class of trees $\mathcal{T}_{\alpha}$ as follows.

\begin{defn}
Suppose $\alpha$ is an ordinal.  We define $\mathcal{T}_{\alpha}$ to be the set of functions $f$ so that 
\begin{itemize}
\item $\text{dom}(f)$ is an end-segment of $\alpha$ of the form $[\beta,\alpha)$ for $\beta$ equal to $0$ or a successor ordinal.  If $\alpha$ is a successor, we allow $\beta = \alpha$, i.e. $\text{dom}(f) = \emptyset$.
\item $\text{ran}(f) \subseteq \omega$.
\item finite support:  the set $\{\gamma \in \text{dom}(f) : f(\gamma) \neq 0\}$ is finite.    
\end{itemize}
We interpret $\mathcal{T}_{\alpha}$ as an $L_{s,\alpha}$-structure by defining 
\begin{itemize}
\item $f \unlhd g$ if and only if $f \subseteq g$.  Write $f \perp g$ if $\neg(f \unlhd g)$ and $\neg(g \unlhd f)$.  
\item $f \wedge g = f|_{[\beta, \alpha)} = g|_{[\beta, \alpha)}$ where $\beta = \text{min}\{ \gamma : f|_{[\gamma, \alpha)} =g|_{[\gamma, \alpha)}\}$, if non-empty (note that $\beta$ will not be a limit, by finite support). Define $f \wedge g$ to be the empty function if this set is empty (note that this cannot occur if $\alpha$ is a limit).  
\item $f <_{lex} g$ if and only if $f \vartriangleleft g$ or, $f \perp g$ with $\text{dom}(f \wedge g) = [\gamma +1,\alpha)$ and $f(\gamma) < g(\gamma)$
\item For all $\beta \in \alpha \setminus \text{lim}(\alpha)$, $P_{\beta} = \{ f \in \mathcal{T}_{\alpha} : \text{dom}(f) = [\beta, \alpha)\}$.    
\end{itemize}
\end{defn}

\vspace{.25in}
\begin{center}
\includegraphics[scale=.18]{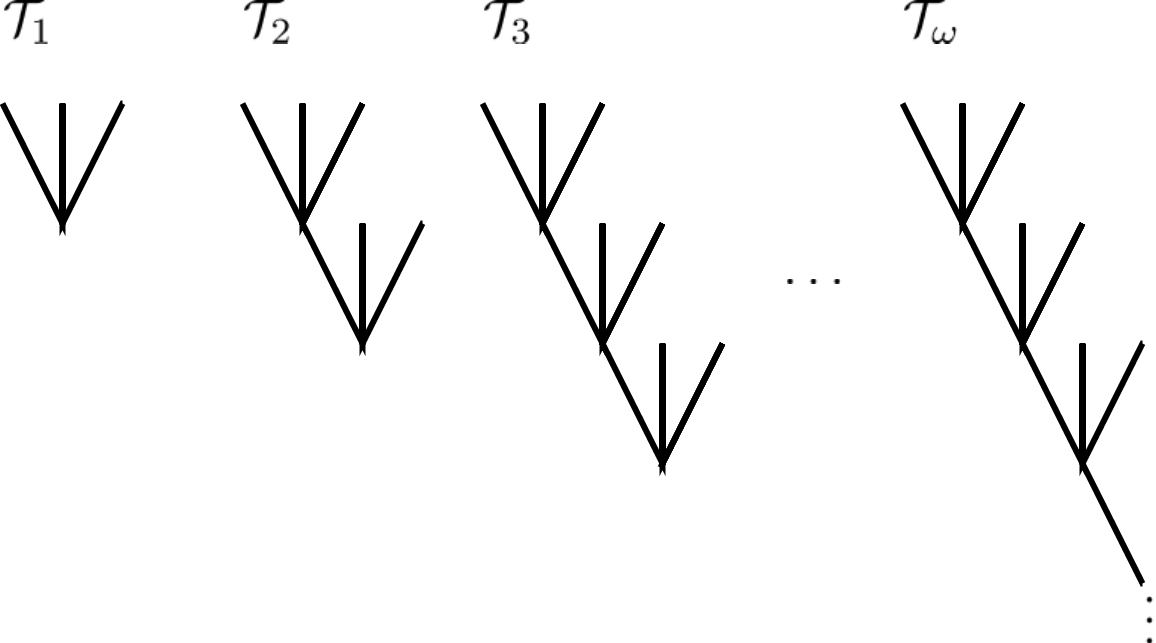} 
\captionof{figure}{Illustration of the trees $\mathcal{T}_{\alpha}$}
\end{center}

It is easy to check that for all $n < \omega$, $\mathcal{T}_{n} \cong \omega^{\leq n}$.  For $\alpha$ infinite, however, $\mathcal{T}_{\alpha}$ will be ill-founded (as a partial order).  In particular, $P_{0}$ names the level at the \emph{top} of the tree, $P_{\beta+1}$ names the level immediately \emph{below} $P_{\beta}$, and so on.  We remark that condition (1) in the definition of $\mathcal{T}_{\alpha}$ was stated incorrectly in the first version of this paper via the weaker requirement that $\text{dom}(f)$ is an end-segment, non-empty if $\alpha$ is limit.  The inductive constructions involving the $\mathcal{T}_{\alpha}$ typically assume that $\mathcal{T}_{\alpha+1}$ consists of the empty function (the root) and countably many copies of $\mathcal{T}_{\alpha}$ given by $\{\langle i \rangle \frown \eta : i< \omega,\eta \in \mathcal{T}_{\alpha}\}$.  But if $\alpha$ is a limit, this becomes false if we allow functions with domain $\{\alpha\}$ since the empty function is not an element of $\mathcal{T}_{\alpha}$ and therefore the function $\alpha \mapsto i$ is not of the form $\langle i \rangle \frown \eta$ for some $\eta \in \mathcal{T}_{\alpha}$.  This is rectified by omitting functions whose domain is an end-segment of the form $[\beta,\alpha)$ for $\beta$ limit.	

As many arguments in this paper will involve inductive constructions of trees of tuples indexed by $\mathcal{T}_{\alpha}$, it will be useful to fix notation as follows:
\begin{defn}
Suppose $\alpha$ is an ordinal.  
\begin{enumerate}
\item (Restriction) If $w \subseteq \alpha \setminus \text{lim}(\alpha)$, the \emph{restriction of} $\mathcal{T}_{\alpha}$ \emph{to the set of levels  }$w$ is given by 
$$
\mathcal{T}_{\alpha} \upharpoonright w = \{\eta \in \mathcal{T}_{\alpha} : \min (\text{dom}(\eta)) \in w \text{ and }\beta \in \text{dom}(\eta) \setminus w \implies \eta(\beta) = 0\}.
$$
\item (Concatenation)  If $\eta \in \mathcal{T}_{\alpha}$, $\text{dom}(\eta) = [\beta+1,\alpha)$, and $i < \omega$, let $\eta \frown \langle i \rangle$ denote the function $\eta \cup \{(\beta,i)\}$.  We define $\langle i \rangle \frown \eta \in \mathcal{T}_{\alpha+1}$ to be $\eta \cup \{(\alpha,i)\}$.  We write $\langle i \rangle$ for $\emptyset \frown \langle i \rangle$.
\item (Canonical inclusions) If $\alpha < \beta$, we define the map $\iota_{\alpha \beta} : \mathcal{T}_{\alpha} \to \mathcal{T}_{\beta}$ by $\iota_{\alpha \beta}(f) = f \cup \{(\gamma, 0) : \gamma \in \beta \setminus \alpha\}$.
\item (The all $0$'s path) If $\beta < \alpha$, then $\zeta_{\beta}$ denotes the function with $\text{dom}(\zeta_{\beta}) = [\beta, \alpha)$ and $\zeta_{\beta}(\gamma) = 0$ for all $\gamma \in [\beta,\alpha)$.  
\end{enumerate}
\end{defn}

The function $i_{\alpha \beta}$ includes $\mathcal{T}_{\alpha}$ into $\mathcal{T}_{\beta}$ by adding zeros to the bottom of every node in $\mathcal{T}_{\alpha}$.  Clearly if $\alpha < \beta < \gamma$, then $\iota_{\alpha \gamma} = \iota_{\beta \gamma} \circ \iota_{\alpha \beta}$.  If $\beta$ is a limit, then $\mathcal{T}_{\beta}$ is the direct limit of the $\mathcal{T}_{\alpha}$ for $\alpha < \beta$ along these maps.  Visually, to get $\mathcal{T}_{\alpha+1}$ from $\mathcal{T}_{\alpha}$, one takes countably many copies of $\mathcal{T}_{\alpha}$ and adds a single root at the bottom.  Lastly, note that the function $\zeta_{\beta}$ will only be an element of $\mathcal{T}_{\alpha}$ if $\beta \in \alpha \setminus \text{lim}(\alpha)$.  

\begin{defn}  Suppose $I$ is an $L'$-structure, where $L'$ is some language. 
\begin{enumerate}
\item  We say $(a_{i} : i \in I)$ is a set of $I$\emph{-indexed indiscernibles} if whenever 

$(s_{0}, \ldots, s_{n-1})$, $(t_{0}, \ldots, t_{n-1})$ are tuples from $I$ with 
$$
\text{qftp}_{L'}(s_{0}, \ldots, s_{n-1}) = \text{qftp}_{L'}(t_{0}, \ldots, t_{n-1}),
$$
then we have
$$
\text{tp}(a_{s_{0}},\ldots, a_{s_{n-1}}) = \text{tp}(a_{t_{0}},\ldots, a_{t_{n-1}}).
$$
\item In the case that $L' = L_{s,\alpha}$ for some $\alpha$, we say that an $I$-indexed indiscernible is $\emph{s-indiscernible}$.  As the only $L_{s,\alpha}$-structures we will consider will be trees, we will often refer to $I$-indexed indiscernibles in this case as \emph{s-indiscernible trees}.  
\item We say that $I$-indexed indiscernibles have the \emph{modeling property} if, given any $(a_{i} : i \in I)$ from $\mathbb{M}$, there is an \(I\)-indexed indiscernible \((b_{i} : i \in I)\) in $\mathbb{M}$ \emph{locally based} on \((a_{i} : i \in I)$ -- i.e., given any finite set of formulas \(\Delta\) from \(L\) and a finite tuple \((t_{0}, \ldots, t_{n-1})\) from \(I\), there is a tuple \((s_{0}, \ldots, s_{n-1})\) from \(I\) so that 
\[
\text{qftp}_{L'} (t_{0}, \ldots, t_{n-1}) =\text{qftp}_{L'}(s_{0}, \ldots , s_{n-1})
\]
and also 
\[
\text{tp}_{\Delta}(b_{t_{0}}, \ldots, b_{t_{n-1}}) = \text{tp}_{\Delta}(a_{s_{0}}, \ldots, a_{s_{n-1}}).
\]
\end{enumerate}
\end{defn}

\begin{fact}\cite[Theorem 4.3]{KimKimScow}\label{modeling}
Let denote \(I_{s}\) be the \(L_{s,\omega}\)-structure \((\omega^{<\omega}, \unlhd, <_{lex}, \wedge, (P_{\alpha})_{\alpha < \omega})\) with all symbols being given their intended interpretations and each \(P_{\alpha}\) naming the elements of the tree at level \(\alpha\).  Then \(I_{s}\)-indexed indiscernibles have the modeling property.  
\end{fact}

\begin{rem}
Note that the tree $\omega^{<\omega}$ is \emph{not} the same tree as $\mathcal{T}_{\omega}$, which is ill-founded.  
\end{rem}

\begin{cor}
For any $\alpha$, $\mathcal{T}_{\alpha}$-indexed indiscernibles have the modeling property.
\end{cor}

\begin{proof}
By Fact \ref{modeling} and compactness.  
\end{proof}

\begin{defn}
Suppose $(a_{\eta})_{\eta \in \mathcal{T}_{\alpha}}$ is a tree of tuples, and $C$ is a set of parameters.  
\begin{enumerate}
\item We say $(a_{\eta})_{\eta \in \mathcal{T}_{\alpha}}$ is \emph{spread out over} C if for all $\eta \in \mathcal{T}_{\alpha}$ with $\text{dom}(\eta) =[\beta+1,\alpha)$ for some $\beta < \alpha$, there is a global $C$-invariant type $q_{\eta} \supseteq \text{tp}(a_{\unrhd \eta \frown \langle 0 \rangle}/C)$ so that $(a_{\unrhd \eta \frown \langle i \rangle})_{i < \omega}$ is a Morley sequence over $C$ in $q_{\eta}$.  
\item Suppose $(a_{\eta})_{\eta \in \mathcal{T}_{\alpha}}$ is a tree which is spread out and $s$-indiscernible over $C$ and for all $w,v \in [\alpha \setminus \text{lim}(\alpha)]^{<\omega}$ with $|w| = |v|$,
$$
(a_{\eta})_{\eta \in \mathcal{T}_{\alpha} \upharpoonright w} \equiv_{C} (a_{\eta})_{\eta \in \mathcal{T}_{\alpha} \upharpoonright v}
$$
then we say $(a_{\eta})_{\eta \in \mathcal{T}_{\alpha}}$ is a \emph{Morley tree} over $C$.  
\item A \emph{tree Morley sequence} over $C$ is a $C$-indiscernible sequence of the form $(a_{\zeta_{\beta}})_{\beta in \alpha \setminus \text{lim}(\alpha)}$ for some Morley tree $(a_{\eta})_{\eta \in \mathcal{T}_{\alpha}}$ over $C$.  
\end{enumerate}
\end{defn}

\begin{rem}
With regards to condition (1), if $\beta$ is a limit ordinal, then $\eta \frown \langle i \rangle$ describes a function which is \emph{not} an element of $\mathcal{T}_{\alpha}$ (as the least element of its domain is not $0$ or a successor).  In this case, we will abuse notation, writing $a_{\unrhd \eta \frown \langle i \rangle}$ for the tuple enumerating $\{a_{\nu} : \nu \in \mathcal{T}_{\alpha}, \eta \frown \langle i \rangle \subseteq \nu\}$.  Additionally,if $(a_{\eta})_{\eta \in \mathcal{T}_{\alpha}}$ is $s$-indiscernible over $C$, then, in order to be spread out over $C$, it suffices to have global $C$-invariant types as in (1) for all $\eta$ identically zero\textemdash i.e. those nodes in the tree of the form $\zeta_{\beta}$ for some $\beta \in \alpha \setminus \text{lim}(\alpha)$.  Note that the condition in (2) forces $(a_{\zeta_{\beta}})_{\beta \in \alpha \setminus \text{lim}(\alpha)}$ to be $C$-indiscernible\textemdash in fact, (1) and (2) together can be shown to be equivalent to demanding that the tree is indiscernible with respect to the language $L = \langle \unlhd, <_{lex},\wedge, \leq_{len}\rangle$, where $\leq_{len}$ is interpreted as the pre-order which compares the lengths of nodes in the tree.  Finally, in (3) we speak of $(a_{\zeta_{\beta}})_{\beta < \alpha}$, the sequence indexed by the all-zeroes path in the tree, simply because this is a convenient choice of a path.  In an $s$-indiscernible tree over $C$, any two paths will have the same type over $C$.  Hence, (3) may be stated more succinctly as:  a tree Morley sequence over $C$ is a path in some Morley tree over $C$.  
\end{rem}

\begin{lem}\label{concatenation}
Suppose $(a_{i})_{i < \omega}$ is a tree Morley sequence over $C$.
\begin{enumerate}
\item If $a_{i} = (b_{i},c_{i})$ for all $i < \omega$, where the $b_{i}$'s are all initial subtuples of $a_{i}$ of the same length, then $(b_{i})_{i < \omega}$ is a tree Morley sequence over $C$.
\item Given $1 \leq n < \omega$, suppose $d_{i} = (a_{n\cdot i}, a_{n \cdot i + 1}, \ldots, a_{n \cdot i + n-1})$.  Then $(d_{i})_{i < \omega}$ is a tree Morley sequence over $C$.  
\end{enumerate}
\end{lem}

\begin{proof}
(1) is immediate from the definition:  $s$-indiscernibility, spread-outness, and being a Morley tree over $C$ are all preserved under taking subtuples.  

(2)  Suppose $(a_{\eta})_{\eta \in \mathcal{T}_{\omega}}$ is a Morley tree over $C$ with $a_{\zeta_{i}} = a_{i}$.  Define a function $j: \mathcal{T}_{\omega} \to \mathcal{T}_{\omega}$ so that if $\eta \in \mathcal{T}_{\omega}$ with $\text{dom}(\eta) = [k,\omega)$, then $\text{dom}(j(\eta)) = [n(k+1),\omega)$ and 
$$
j(\eta)(l) = \left\{ \begin{matrix}
\eta\left( \frac{l}{n} - 1 \right) & \text{ if } n | l \\
0 & \text{ otherwise}
\end{matrix} \right.
$$
for all $l \in [n(k+1),\omega)$.  Define $(b_{\eta})_{\eta \in \mathcal{T}_{\omega}}$ by
$$
b_{\eta} = (a_{j(\eta)}, a_{j(\eta) \frown \langle 0 \rangle}, \ldots, a_{j(\eta) \frown 0^{n-1}}).
$$
It is easy to check that this is also an $s$-indiscernible tree over $M$ (more formally, this construction corresponds to the $n$-fold elongation of the tree $(a_{\eta})_{\eta \in \mathcal{T}_{\omega}}$ as defined in \cite{ArtemNick} so $(b_{\eta})_{\eta \in \mathcal{T}_{\omega}}$ is $s$-indiscernible over $M$ by \cite[Proposition 2.1(1)]{ArtemNick} there).  It is also easy to check that $(b_{\eta})_{\eta \in \mathcal{T}_{\omega}}$ is spread out over $M$.  Finally, the tree $(b_{\eta})_{\eta \in \mathcal{T}_{\omega}}$ is also a Morley tree over $M$:  given $w \in [\omega]^{<\omega}$, let $w' = \{n(k+1) - l : k \in w, l < n\}$.  Then if $w,v \in [\omega]^{<\omega}$ and $|w| = |v|$, then $|w'| = |v'|$ so $(a_{\eta})_{\eta \in \mathcal{T}_{\omega} \upharpoonright w'} \equiv_{C} (a_{\eta})_{\eta \in \mathcal{T}_{\omega} \upharpoonright v'}$ so $(b_{\eta})_{\eta \in \mathcal{T}_{\omega} \upharpoonright w} \equiv_{C} (b_{\eta})_{\eta \in \mathcal{T}_{\omega}\upharpoonright v}$.  It follows that $(b_{\zeta_{i}})_{i < \omega}$ is a tree Morley sequence over $C$.  We have
\begin{eqnarray*}
b_{\zeta_{i}} &=& (a_{\zeta_{n(i+1)}}, a_{\zeta_{n(i+1)}\frown \rangle 0 \rangle}, \ldots, a_{\zeta_{n(i+1)}\frown 0^{n-1}}) \\
&=& (a_{n(i+1)}, a_{n(i+1)-1}, \ldots, a_{n(i+1) - (n-1)}),
\end{eqnarray*}
so by reversing the order of the tuple, we deduce that $(d_{i})_{i < \omega}$ is a tree Morley sequence over $M$.  
\end{proof}

From the existence of a sufficiently large tree which is spread out and $s$-indiscernible over $M$, one can obtain a Morley tree which is based on it.  The proof is via a standard Erd\H{o}s-Rado argument.  We follow the argument of \cite[Theorem 1.13]{grossberg2002primer}.  

\begin{lem}\label{morleyextraction}
Suppose $(a_{\eta})_{\eta \in \mathcal{T}_{\kappa}}$ is a tree of tuples, spread out and $s$-indiscernible over $M$.  If $\kappa$ is sufficiently large, then there is a Morley tree $(b_{\eta})_{\eta \in \mathcal{T}_{\omega}}$ so that for all $w \in [\omega]^{<\omega}$, there is $v \in [\kappa \setminus \text{lim}(\kappa)]^{<\omega}$ so that 
$$
(a_{\eta})_{\eta \in \mathcal{T}_{\kappa} \upharpoonright v} \equiv_{M} (b_{\eta})_{\eta \in \mathcal{T}_{\omega} \upharpoonright w}.  
$$
\end{lem}

\begin{proof}
Let $\lambda = 2^{|M|+|T|}$ and set $\kappa = \beth_{\lambda^{+}}(\lambda)$.  Given a tree $(a_{\eta})_{\eta \in \mathcal{T}_{\kappa}}$ $s$-indiscernible and spread out over $M$, let 
$$
\Gamma_{n} = \{\text{tp}((a_{\eta})_{\eta \in \mathcal{T}_{\omega} \upharpoonright w}/M) : w \in [\kappa \setminus \text{lim}(\kappa)]^{n}\}.
$$
By induction on $n$, we will find a sequence of types $p_{n} \in \Gamma_{n}$ so that 
$$
\Delta(x_{\eta} :\eta \in \mathcal{T}_{\omega}) = \bigcup_{n < \omega} \bigcup_{w \in [\omega]^{n}} p_{n}(x_{\eta} : \eta \in \mathcal{T}_{\omega} \upharpoonright w)
$$
is consistent.  Construct by induction on $n$ cofinal subsets $F_{n} \subseteq \lambda^{+}$ and subsets $X_{\xi,n} \subseteq \kappa \setminus \text{lim}(\kappa)$ so that 
\begin{enumerate}
\item $F_{n+1} \subseteq F_{n}$.
\item $|X_{\xi,n}| > \beth_{\alpha}(\lambda)$ when $\xi$ is the $\alpha$th element of $F_{n}$.
\item If $w \in [X_{\xi,n}]^{n}$, then $(a_{\eta})_{\eta \in \mathcal{T}_{\kappa} \upharpoonright w} \models p_{n}$.
\item $|F_{n}| = \lambda^{+}$.  
\end{enumerate}
For $n = 0$, we let $F_{0} = \lambda^{+}$ and $X_{\xi,0} = \kappa \setminus \text{lim}(\kappa)$ for all $\xi < \lambda^{+}$.  Suppose $F_{n}$ and $(X_{\xi,n})_{\xi \in F_{n}}$ have been constructed.  Write $F_{n} = \{\xi_{\alpha} : \alpha < \lambda^{+}\}$ where the $\xi_{\alpha}$ enumerate $F_{n}$ in increasing order.  Then for all $\alpha < \lambda^{+}$, 
$$
|X_{\xi_{\alpha+n+1},n}| > \beth_{\alpha + n+1}(\lambda).
$$
For a moment, fix $\xi = \xi_{\alpha+n+1}$.  Define a coloring on $[X_{\xi,n}]^{n+1}$ by
$$
w \mapsto \text{tp}((a_{\eta})_{\eta \in \mathcal{T}_{\kappa} \upharpoonright w}/M).
$$
This is a coloring with at most $\lambda$ many colors so by Erd\H{o}s-Rado there is a homogeneous subset $X_{\xi,n+1} \subseteq X_{\xi,n}$ with $|X_{\xi,n+1}| > \beth_{\alpha}(\lambda)$.  Let $p_{n+1,\alpha+n+1}$ denote its constant value.  By the pigeonhole principle, as the set of possible values is $\lambda$ and $\{\alpha + n+1 : \alpha < \lambda^{+}\}$ has size $\lambda^{+}$, there must be some subset $Y \subseteq \{\alpha + n+1 : \alpha < \lambda^{+}\}$ of cardinality $\lambda^{+}$ so that $\beta,\beta' \in Y$ implies $p_{n+1,\beta} = p_{n+1,\beta'}$.  Let $p_{n+1} = p_{n+1,\beta}$ for some/all $\beta \in Y$.  Put $F_{n+1} = \{\xi_{\beta} : \beta \in Y\}$.  Then $p_{n+1}$, $F_{n+1}$, and $(X_{\xi,n+1})_{\xi \in F_{n+1}}$ clearly satisfy the requirements.  

By compactness, this shows that $\Delta(x_{\eta} : \eta \in \mathcal{T}_{\omega})$ is consistent.  Let $(b_{\eta})_{\eta \in \mathcal{T}_{\omega}}$ be a realization\textemdash now to show $(b_{\eta})_{\eta \in \mathcal{T}_{\omega}}$ is a Morley tree over $M$, we must show that $(b_{\eta})_{\eta \in \mathcal{T}_{\omega}}$ is $s$-indiscernible and spread out over $M$.  To see that it is spread out over $M$, fix any $\eta \in \mathcal{T}_{\omega}$ with $\text{dom}(\eta) = [n+1,\omega)$.  Setting $w = \{0, \ldots, n\}$, there is $v \in [\kappa \setminus \text{lim}(\kappa)]^{<\omega}$, $v = \{\alpha_{0} < \ldots < \alpha_{n}\}$ so that $(b_{\nu})_{\nu \in \mathcal{T}_{\omega} \upharpoonright w} \equiv_{M} (a_{\nu})_{\nu \in \mathcal{T}_{\kappa} \upharpoonright v}$.  If $\nu_{i} \in \mathcal{T}_{\kappa}$ has domain $[\alpha_{0}, \kappa)$, $\nu_{i}(\alpha_{0}) = i$ and $\nu_{i}$ is identically zero elsewhere, $\langle a_{\unrhd \nu_{i}} : i < \omega \rangle$ is a Morley sequence over $M$ in an $M$-invariant type.  It follows that $\langle b_{\unrhd \eta \frown \langle i \rangle} : i < \omega \rangle$ is also a Morley sequence in an $M$-invariant type, which establishes spread-outness of the tree.  Checking that the tree is $s$-indiscernible over $M$ is entirely similar.  
\end{proof}

\subsection{The symmetry characterization of NSOP$_{1}$}

In this subsection, we prove a version of Kim's lemma for tree Morley sequences and use it to prove that Kim-independence is symmetric over models in an NSOP$_{1}$ theory.  Lemma \ref{treeexistence} is the key step, showing that tree Morley sequences exist under certain assumptions.  The method of proof is an inductive construction of a spread out $s$-indiscernible tree, from which a Morley tree (and hence a tree Morley sequence) can then be extracted.  This basic proof-strategy will be repeated several times throughout the paper.  

\begin{lem}\label{treeexistence}
Suppose $T$ is NSOP$_{1}$, $M \models T$, and $a \ind^{K}_{M} b$.  For any ordinal $\alpha \geq 1$, there is a spread out $s$-indiscernible tree $(c_{\eta})_{\eta \in \mathcal{T}_{\alpha}}$ over $M$, so that if $\eta \vartriangleleft \nu$ and $\text{dom}(\nu) = \alpha$, then $c_{\eta}c_{\nu} \equiv_{M} ab$.  
\end{lem}

\begin{proof}
We will argue by induction on $\alpha$.  For the case $\alpha=1$, fix $q \supseteq \text{tp}(b/M)$, a global $M$-invariant type.  Let $\langle b_{i} : i < \omega \rangle \models q^{\otimes \omega}|_{M}$.  As $a \ind^{K}_{M} b$, we may assume this sequence is $Ma$-indiscernible.  Put $c^{1}_{\emptyset} = a$ and $c^{1}_{\langle i \rangle} = b_{i}$.  It is now easy to check that $(c^{1}_{\eta})_{\eta \in \mathcal{T}_{1}}$ is a spread out $s$-indiscernible tree satisfying the requirements.  

Suppose for $\alpha$ we have constructed $(c^{\beta}_{\eta})_{\eta \in \mathcal{T}_{\beta}}$ for $1 \leq \beta \leq \alpha$ such that, if $\gamma < \beta \leq \alpha$ and $\eta \in \mathcal{T}_{\gamma}$ then $c^{\gamma}_{\eta} = c^{\beta}_{\iota_{\gamma \beta}(\eta)}$.  First assume $\alpha$ is a successor.  By spread-outness, we know that $\langle c^{\alpha}_{\unrhd \langle i \rangle} : i < \omega \rangle$ is an $M$-invariant Morley sequence which is, by $s$-indiscernibility over $M$, $Mc^{\alpha}_{\emptyset}$-indiscernible.  Therefore, $c^{\alpha}_{\emptyset} \ind^{K}_{M} (c^{\alpha}_{\unrhd \langle i \rangle})_{i < \omega}$.  By extension (Proposition \ref{extension}), we may find some $c' \equiv_{M(c^{\alpha}_{\unrhd \langle i \rangle})_{i < \omega}} c^{\alpha}_{\emptyset}$ so that 
$$
c' \ind^{K}_{M} (c^{\alpha}_{\eta})_{\eta \in \mathcal{T}_{\alpha}}.  
$$
Choose a global $M$-invariant type $q \supseteq \text{tp}((c^{\alpha}_{\eta})_{\eta \in \mathcal{T}_{\alpha}}/M)$.  Let $\langle (c^{\alpha}_{\eta,i})_{\eta \in \mathcal{T}_{\alpha}} : i < \omega \rangle \models q^{\otimes \omega}|_{M}$ with $c^{\alpha}_{\eta,0} = c^{\alpha}_{\eta}$ for all $\eta \in \mathcal{T}_{\alpha}$.  By the chain condition (Lemma \ref{chaincondition}), we can find $c'' \equiv_{M(c^{\alpha}_{\eta})_{\eta \in \mathcal{T}_{\alpha}}} c'$ so that $c'' \ind^{K}_{M} (c^{\alpha}_{\eta,i})_{\eta \in \mathcal{T}_{\alpha}, i < \omega}$ and $\langle (c^{\alpha}_{\eta,i})_{\eta \in \mathcal{T}_{\alpha}} : i < \omega \rangle$ is $Mc''$-indiscernible.  Define a new tree $(d_{\eta})_{\eta \in \mathcal{T}_{\alpha+1}}$ by setting $d_{\emptyset} = c''$ and $d_{\iota_{\alpha \alpha+1}(\eta)} = c^{\alpha}_{\eta}$ for all $\eta \in \mathcal{T}_{\alpha}$.  Then let $(c^{\alpha+1}_{\eta})_{\eta \in \mathcal{T}_{\alpha+1}}$ be a tree $s$-indiscernible over $M$ locally based on $(d_{\eta})_{\eta \in \mathcal{T}_{\alpha}}$.  By an automorphism, we may assume $c^{\alpha+1}_{\iota_{\alpha\alpha+1}(\eta)} = c^{\alpha}_{\eta}$ for all $\eta \in \mathcal{T}_{\alpha}$.  This satisfies our requirements.  

Next, assume $\alpha$ is a limit.  Let $p(x;b) = \text{tp}(a/Mb)$.  By induction, if $\beta \in \alpha \setminus \text{lim}(\alpha)$, then $c^{\alpha}_{\zeta_{\beta}} \ind^{K}_{M} c^{\alpha}_{\vartriangleright \zeta_{\beta}}$ and $c^{\alpha}_{\zeta_{\beta}}$ realizes $p(x;c_{\nu})$ for all $\nu \in \mathcal{T}_{\alpha}$ with $\zeta_{\beta} \vartriangleleft \nu$ and $\text{dom}(\nu) = \kappa$.  As $\alpha$ is limit, every element $\nu \in \mathcal{T}_{\alpha}$ satisfies $\zeta_{\beta} \vartriangleleft \nu$ for sufficiently large $\beta$ so, by compactness, we can find $c_{*} \ind^{K}_{M} (c^{\alpha}_{\eta})_{\eta \in \mathcal{T}_{\alpha}}$ such that $c_{*}$ realizes $p(x;c_{\nu})$ for all $\nu \in \mathcal{T}_{\alpha}$.  Let $\langle (c^{\alpha}_{\eta,i})_{\eta \in \mathcal{T}_{\alpha}} : i < \omega \rangle$ be an $M$-invariant Morley sequence over $M$ with $c^{\alpha}_{\eta,0} = c^{\alpha}_{\eta}$ for all $\eta \in \mathcal{T}_{\alpha}$.  Once more applying the chain condition (Lemma \ref{chaincondition}), we may assume $\langle (c^{\alpha}_{\eta,i})_{\eta \in \mathcal{T}_{\alpha}} : i < \omega \rangle$ is $Mc_{*}$-indiscernible.  As before, we define a tree $(d_{\eta})_{\eta \in \mathcal{T}_{\alpha+1}}$ by setting $d_{\emptyset} = c_{*}$ and $d_{\langle i \rangle \frown \eta} = c^{\alpha}_{\eta,i}$ for all $\eta \in \mathcal{T}_{\alpha}$ and $i < \omega$.  Then we define $(c^{\alpha+1}_{\eta})_{\eta \in \mathcal{T}_{\alpha+1}}$ to be a tree $s$-indiscernible over $M$ locally based on $(d_{\eta})_{\eta \in \mathcal{T}_{\alpha}}$, and by an automorphism, we may assume $c^{\alpha+1}_{\iota_{\alpha\alpha+1}(\eta)} = c^{\alpha}_{\eta}$ for all $\eta \in \mathcal{T}_{\alpha}$. 

Finally, suppose for $\delta$ limit we have constructed $(c^{\beta}_{\eta})_{\eta \in \mathcal{T}_{\beta}}$ for $1 \leq \beta < \delta$ such that, if $\gamma < \beta <\delta $ and $\eta \in \mathcal{T}_{\gamma}$ then $c^{\gamma}_{\eta} = c^{\beta}_{\iota_{\gamma \beta}(\eta)}$.  If $\eta \in \mathcal{T}_{\delta}$, then for some $\beta < \delta$, there is $\nu \in \mathcal{T}_{\beta}$ so that $\iota_{\beta \delta}(\nu) = \eta$.  Then put $c^{\delta}_{\eta} = c^{\beta}_{\nu}$.  This defines for all $\beta \leq \delta$ an $s$-indiscernible tree $(c^{\beta}_{\eta})_{\eta \in \mathcal{T}_{\eta}}$ satisfying our requirements.  
\end{proof}

\vspace{.15in}
\begin{center}
\includegraphics[scale=.12]{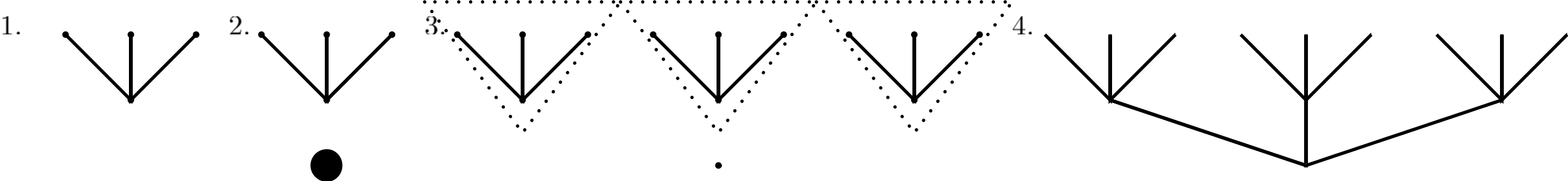} 
\captionof{figure}{The figure displays the construction of the tree indexed by $\mathcal{T}_{2}$ in stages:  1.  The tree indexed by $\mathcal{T}_{1}$. 2.  Using extension to obtain a new base point. 3.  Taking a Morley sequence in the given tree, indiscernible over the new base point. 4.  Extracting an s-indiscernible tree to obtain a spread out, s-indiscernible tree indexed by $\mathcal{T}_{2}$.}
\end{center}

\begin{lem}\label{movingtms}
Suppose $T$ is NSOP$_{1}$, $M \models T$, and $a \ind^{K}_{M} b$.  Then there is a tree Morley sequence $(a_{i})_{i < \omega}$ which is $Mb$-indiscernible with $a_{0} = a$.  
\end{lem}

\begin{proof}
By Lemma \ref{treeexistence}, for arbitrarily large cardinals $\kappa$, there is a tree $(c_{\eta})_{\eta \in \mathcal{T}_{\kappa}}$ which is spread out and $s$-indiscernible over $M$ so that if $\eta \vartriangleleft \nu$ and $\text{dom}(\nu) = \kappa$ then $c_{\eta}c_{\nu} \equiv_{M} ab$.  Note that $\mathcal{T}' = \mathcal{T}_{\kappa} \setminus \{\nu  \in \mathcal{T}_{\kappa} : \text{dom}(\nu) = \kappa\} = \{\eta \in \mathcal{T}_{\kappa} : \text{dom}(\eta) \subseteq [1,\kappa)\}$ is isomorphic to $\mathcal{T}_{\kappa}$.  So we may enumerate $(c_{\eta})_{\eta \in \mathcal{T}'}$ as $(d_{\eta})_{\eta \in \mathcal{T}_{\kappa}}$.  Note that for all $\eta \in \mathcal{T}_{\kappa}$, $d_{\eta} \equiv_{M} a$ and $d_{\zeta_{\alpha}} = c_{\zeta_{1+\alpha}}$ for all $\alpha < \kappa$.  By Lemma \ref{morleyextraction}, there is a Morley tree over $M$ $(d'_{\eta})_{\eta \in \mathcal{T}_{\omega}}$ so that for all $w \in [\omega]^{<\omega}$ there is $v \in [\kappa \setminus \text{lim}(\kappa)]^{<\omega}$ so that $(d_{\eta})_{\eta \in \mathcal{T}_{\kappa} \upharpoonright v} \equiv_{M} (d'_{\eta})_{\eta \in \mathcal{T}_{\omega} \upharpoonright w}$.  

Let $p(x;a) = \text{tp}(b/Ma)$.  We claim $\bigcup_{i < \omega} p(x;d'_{\zeta_{i}})$ is consistent.  Given $n$, let $w = \{0, \ldots, n-1\}$.  Find $v \in [\kappa \setminus \text{lim}(\kappa)]^{<\omega}$ so that $(d_{\eta})_{\eta \in \mathcal{T}_{\kappa} \upharpoonright v} \equiv_{M} (d'_{\eta})_{\eta \in \mathcal{T}_{\omega} \upharpoonright w}$.  If $v = \{\alpha_{0}, \ldots , \alpha_{n-1}\}$, then for $i<n$ we have $d_{\zeta_{\alpha_{i}}} = c_{1+\zeta_{\alpha_{i}}}$.   Then because $c_{\zeta_{1+\alpha_{i}}} c_{\zeta_{0}} \equiv_{M} ab$ for all $i<n$, we have $c_{\zeta_{0}} \models \bigcup_{i < n} p(x;d_{\zeta_{\alpha_{i}}})$.  This shows $\bigcup_{i < n} p(x;d_{\zeta_{\alpha_{i}}})$ is consistent and hence $\bigcup_{i < n} p(x;d'_{\zeta_{i}})$ is consistent.  The claim follows by compactness.  

Let $b' \models \bigcup_{i < \omega} p(x;d'_{\zeta_{i}})$.  Extract from $(d'_{\zeta_{i}})_{i < \omega}$ an $Mb'$-indiscernible sequence $(a_{i})_{i < \omega}$.  As $(a_{i})_{i < \omega} \equiv_{M} (d'_{\zeta_{i}})_{i < \omega}$, we know $(a_{i})_{i < \omega}$ is a tree Morley sequence.  By an automorphism, we may assume $b' = b$ and $a_{0} = a$.  
\end{proof}

\begin{prop}\label{kimslemmafortmsprequel}
Suppose $T$ is NSOP$_{1}$ and $M \models T$.  Suppose $(a_{i})_{i <\omega}$ is a tree Morley sequence over $M$.  Then $\{\varphi(x;a_{i}) : i < \omega\}$ is inconsistent if and only if $\varphi(x;a_{0})$ Kim-divides over $M$.  
\end{prop}

\begin{proof}
Suppose $(a_{i})_{i < \omega}$ is a tree Morley sequence over $M$.  Let $(a_{\eta})_{\eta \in \mathcal{T}_{\omega}}$ be a Morley tree over $M$ with $a_{\zeta_{i}} = a_{i}$.  Let $\eta_{i} \in \mathcal{T}_{\omega}$ be the function with $\text{dom}(\eta_{i}) = [i,\omega)$ and
$$
\eta_{i}(j) = \left\{ \begin{matrix}
1 & \text{ if } i = j \\
0 & \text{ otherwise}.  
\end{matrix} \right.
$$
Consider the sequence $I = (a_{\eta_{i}},a_{\zeta_{i}})_{i < \omega}$.  Because $(a_{\eta})_{\eta \in \mathcal{T}_{\omega}}$ is a Morley tree over $M$, $I$ is an $M$-indiscernible sequence.  Moreover, by $s$-indiscernibility, $a_{\eta_{0}} \equiv_{MI_{>0}} a_{\zeta_{0}}$.  By indiscernibility, for all $i$, we have $a_{\eta_{i}} \equiv_{MI_{>i}} a_{\zeta_{i}}$.  By NSOP$_{1}$, it follows that $\{\varphi(x;a_{\eta_{i}}) : i < \omega\}$ is consistent if and only if $\{\varphi(x;a_{\zeta_{i}}) : i < \omega\}$ is consistent:  if exactly one of them is consistent, then we have SOP$_{1}$ by Proposition \ref{arrayequivalent}.  

Because $(a_{\eta})_{\eta \in \mathcal{T}_{\omega}}$ is a spread out tree over $M$, $a_{\eta_{i}} \ind^{i}_{M} a_{\eta_{<i}}$ for all $i$.  Using the fact that $(a_{\eta_{i}})_{i < \omega}$ is an $M$-indiscernible sequence and the compactness of the space of $M$-invariant types, we have $(a_{\eta_{i}})_{i < \omega}$ is a Morley sequence in some global $M$-invariant type extending $\text{tp}(a/M)$, so $\varphi(x;a)$ Kim-divides over $M$ if and only if $\{\varphi(x;a_{\eta_{i}}) : i < \omega\}$ is inconsistent.  
\end{proof}

\begin{cor}{(Kim's lemma for tree Morley sequences)}\label{kimslemmafortms}
Suppose $T$ is NSOP$_{1}$ and $M \models T$.  The following are equivalent:
\begin{enumerate}
\item $\varphi(x;a)$ Kim-divides over $M$.
\item For some tree Morley sequence $(a_{i})_{i < \omega}$ over $M$ with $a_{0} = a$, $\{\varphi(x;a_{i}) : i < \omega\}$ is inconsistent.
\item For every tree Morley sequence $(a_{i})_{i < \omega}$ over $M$ with $a_{0} = a$, $\{\varphi(x;a_{i}) : i < \omega\}$ is inconsistent.  
\end{enumerate}
\end{cor}

\begin{cor}{(Chain condition for tree Morley sequences)} \label{tmschaincondition}
Suppose $T$ is NSOP$_{1}$ and $M \models T$.  If $a \ind^{K}_{M} b$ and $I =(b_{i})_{i < \omega}$ is a tree Morley sequence over $M$ with $b_{0} = b$, then there is $a' \equiv_{Mb} a$ so that $a' \ind^{K}_{M} I$ and $I$ is $Ma'$-indiscernible.  
\end{cor}

\begin{proof}
The proof is identical to Proposition \ref{chaincondition} above, since, by Lemma \ref{concatenation}, $\langle (b_{kn + n-1}, b_{kn + n-2}, \ldots, b_{kn}) : k < \omega \rangle$ is a tree Morley sequence over $M$.
\end{proof}

\begin{thm}{(Symmetry)} \label{symmetrycharthm}
Suppose $T$ is a complete theory.  The following are equivalent:
\begin{enumerate}
\item $T$ is NSOP$_{1}$.
\item $\ind^{K}$ is symmetric over models:  for any $M \models T$ and tuples $a,b$ from $\mathbb{M}$, $a\ind^{K}_{M} b \iff b \ind^{K}_{M} a$.
\item $\ind^{K}$ enjoys the following weak symmetry property:  for any $M \models T$ and tuples $a,b$ from $\mathbb{M}$, $a\ind^{i}_{M} b$ implies $b \ind^{K}_{M} a$.
\end{enumerate}
\end{thm}

\begin{proof}
(1)$\iff$(3) is Proposition \ref{weaksymmetrylemma} and (2)$\implies$(3) is immediate from the fact that $a \ind^{i}_{M} b$ implies $a \ind^{K}_{M} b$.  

(1)$\implies$(2).  Suppose $T$ is NSOP$_{1}$.  Assume towards contradiction that $a \ind^{K}_{M} b$ and $b \nind^{K}_{M} a$.  By Lemma \ref{movingtms}, there is a tree Morley sequence over $M$ with $a_{0} = a$ which is $Mb$-indiscernible.  Since $b \nind^{K}_{M} a$, there is some $\varphi(x;a) \in \text{tp}(b/Ma)$ which Kim-divides over $M$.  By Corollary \ref{kimslemmafortms}, $\{\varphi(x;a_{i}) : i < \omega\}$ is inconsistent.  But $\models \varphi(b;a_{i})$ for all $i < \omega$ by indiscernibility, a contradiction.  
\end{proof}

\begin{cor}
Assume the complete theory $T$ is NSOP$_{1}$ and $M \models T$.  Then 
$$
a \ind^{K}_{M} b \iff \text{acl}(a) \ind^{K}_{M} b \iff a \ind^{K}_{M} \text{acl}(b).
$$
\end{cor}

\begin{proof}
By symmetry, it is enough to prove $\text{acl}(a) \ind^{K}_{M} b$, assuming $a \ind^{K}_{M} b$.  If $a \ind^{K}_{M} b$, there is a Morley sequence in an $M$-invariant type $\langle b_{i} : i < \omega \rangle$ with $b_{0} = b$ which is $Ma$-indiscernible.  Then it is automatically $M\text{acl}(a)$-indiscernible so $\text{acl}(a) \ind^{K}_{M} b$.  
\end{proof}

\section{The Independence Theorem}\label{itsection}

The full independence theorem will be deduced from a weak independence theorem, which has an easy proof:  

\begin{prop}\label{reduction}
Assume $T$ is NSOP$_{1}$.  Then $\ind^{K}$ satisfies the following weak independence theorem over models:  if $M \models T$, $a \equiv_{M} a'$, $a \ind^{K}_{M} b$, $a' \ind^{K}_{M} c$ and $b \ind^{u}_{M} c$, then there is $a''$ with $a'' \equiv_{Mb} a$, $a'' \equiv_{Mc} a'$ and $a'' \ind^{K}_{M} bc$.  
\end{prop}

\begin{proof}
Suppose $T$ is NSOP$_{1}$ and fix $M \models T$ and tuples $a,a',b,c$ so that $a \equiv_{M} a'$, $a \ind^{K}_{M} b$, $a' \ind^{K}_{M} c$ and $b \ind^{u}_{M} c$.  

\textbf{Claim}:  There is $c'$ so that $ac' \equiv_{M} a'c$ and $a \ind^{K}_{M} bc'$.  

\emph{Proof of claim}:  By symmetry, it suffices to find $c'$ with $ac' \equiv_{M} a'c$ and $bc' \ind^{K}_{M} a$.  Let $p(x;a') = \text{tp}(c/Ma')$.  By invariance, we know $p(x;a)$ does not Kim-fork over $M$.  We have to show
$$
p(x;a) \cup \{\neg \varphi(x,b;a) : \varphi(x,y;a) \in L(Ma) \text{ Kim-divides over }M\}
$$
is consistent.  If not, then by compactness and Kim-forking = Kim-dividing, we must have 
$$
p(x;a) \vdash \varphi(x,b;a),
$$
for some $\varphi$ where $\varphi(x,y;a)$ Kim-divides over $M$.  By symmetry, $b \ind^{K}_{M} a$, so there is some $M$-invariant Morley sequence $(a_{i})_{i < \omega}$ with $a_{0} = a$ which is moreover $Mb$-indiscernible.  Then we have 
$$
\bigcup_{i < \omega} p(x;a_{i}) \vdash \{\varphi(x,b;a_{i}) : i < \omega\}.
$$
As $p(x;a)$ does not Kim-fork over $M$, we know $\bigcup_{i < \omega} p(x;a_{i})$ is consistent.  But, by Kim's lemma for Kim-dividing, we know $\{\varphi(x,y;a_{i}) : i < \omega\}$ is inconsistent and \emph{a fortiori} $\{\varphi(x,b;a_{i}) : i < \omega\}$ is inconsistent, a contradiction.  So the given partial type is consistent. Let $c'$ realize it.  Then $ac' \equiv_{M} a'c$ and $c'b \ind^{K}_{M} a$, which proves the claim. \qed

As $b \ind^{u}_{M} c$, by left extension, there is $c'' \equiv_{Mb} c$ with $bc' \ind^{u}_{M} c''$.  Then by right extension and automorphism, we can choose some $b''$ so that $bc' \equiv_{M} b''c''$ and $bc' \ind^{u}_{M} b''c''$.  As $bc' \ind^{u}_{M} b'' c''$ and $bc' \equiv_{M} b''c''$, it follows that $(b''c'', bc')$ starts a Morley sequence $I$ in some global $M$-finitely satisfiable (hence $M$-invariant) type.  As $a \ind^{K}_{M} bc'$, we may, by the chain condition (Proposition \ref{chaincondition}) find some $a_{*} \equiv_{Mbc'} a$ so that $I$ is $Ma_{*}$-indiscernible and $a_{*} \ind^{K}_{M} I$.  Then, we obtain $a_{*} \equiv_{Mb} a$, $a_{*}c'' \equiv_{M} a'c$, and $a_{*} \ind^{K}_{M} bc''$.  By construction, $c'' \equiv_{Mb} c$ so there is $\sigma \in \text{Aut}(\mathbb{M}/Mb)$ with $\sigma(c'') = c$.  Then $\sigma(a_{*}) \ind^{K}_{M} bc$, $\sigma(a_{*}) \equiv_{Mb} a$, and $\sigma(a_{*}) \equiv_{Mc} a'$, which shows that the weak independence theorem over models holds for $T$.  
\end{proof}

\begin{lem}\label{consistenttree}
Suppose $T$ is NSOP$_{1}$, $M \models T$, and $a \ind^{K}_{M} b$.  Fix an ordinal $\alpha$ and any $q \supseteq \text{tp}(b/M)$, a global $M$-invariant type.  If $(b_{\eta})_{\eta \in \mathcal{T}_{\alpha}}$ is a tree, spread out over $M$, so that, for all $\nu \in \mathcal{T}_{\alpha}$, $b_{\nu} \models q|_{M b_{\vartriangleright \nu}}$, then, writing $p(x;b)$ for $\text{tp}(a/Mb)$, we have 
$$
\bigcup_{\eta \in \mathcal{T}_{\alpha}} p(x;b_{\eta})
$$
is consistent and non-Kim-forking over $M$.  
\end{lem}

\begin{proof}
The proof is by induction on $\alpha$.  For $\alpha = 0$, there is nothing to show.  For $\alpha$ limit, it follows by induction, using that $\mathcal{T}_{\alpha}$ is the direct limit of the $\mathcal{T}_{\beta}$ for $\beta < \alpha$ along the maps $\iota_{\beta \alpha}$.  Now suppose given $(b_{\eta})_{\eta \in \mathcal{T}_{\alpha+1}}$ as in the statement.  We know that $b_{\unrhd \zeta_{\alpha}} = (b_{\iota_{\alpha \alpha+1}(\eta)})_{\eta \in \mathcal{T}_{\alpha}}$ is a tree spread out over $M$ so that, for all $\nu \in \mathcal{T}_{\alpha}$, $b_{\iota_{\alpha \alpha+1}(\nu)} \models q|_{M (b_{\iota_{\alpha \alpha+1}(\eta)})_{\eta \unrhd \nu}}$.  Note that $\emptyset \frown \langle 0 \rangle = \zeta_{\alpha}$.  By induction, then, 
$$
\bigcup_{\nu \unrhd \emptyset \frown \langle0\rangle } p(x;b_{\nu})
$$
is consistent and non-Kim-forking over $M$.  By spread outness over $M$, $\langle b_{\unrhd \emptyset \frown \langle i \rangle} : i < \omega \rangle$ is a Morley sequence in some global $M$-invariant type.  By the chain condition, 
$$
\bigcup_{i < \omega} \bigcup_{\nu \unrhd \emptyset \frown \langle i \rangle} p(x;b_{\nu})
$$
is consistent and non-Kim-forking over $M$.  As $b_{\emptyset} \models q|_{M(b_{\nu \vartriangleright \emptyset})}$, it follows by Proposition \ref{reduction} that 
$$
 p(x;b_{\emptyset}) \cup \bigcup_{i < \omega} \bigcup_{\nu \unrhd \emptyset \frown \langle i \rangle} p(x;b_{\nu}) 
$$
is consistent and non-Kim-forking over $M$.  Unwinding definitions, this says
$$
\bigcup_{\eta \in \mathcal{T}_{\alpha+1}} p(x;b_{\eta})
$$
is consistent and non-Kim-forking over $M$, completing the proof.  \end{proof}

\begin{rem}\label{strongconsistenttree}
In the above proof, the hypothesis that $b_{\nu} \models q|_{Mb_{\vartriangleright \nu}}$ is used to apply the weak independence theorem (Proposition \ref{reduction}).  Once one has proved the full independence theorem (Theorem \ref{itthmchar}), the same proof gives $\bigcup_{\eta \in \mathcal{T}_{\alpha}} p(x;b_{\eta})$ is consistent and non-Kim-forking over $M$, just under the hypothesis that $(b_{\eta})_{\eta \in \mathcal{T}_{\alpha}}$ is $s-$indiscernible and spread out over $M$, since $b_{\nu} \ind^{K}_{M} b_{\vartriangleright \nu}$ in any tree $s$-indiscernible and spread out over $M$.  
\end{rem}

\begin{lem}{(Zig-zag Lemma)}\label{zigzag}
Suppose the complete theory $T$ is NSOP$_{1}$, $M \models T$ and $b \ind^{K}_{M} b'$.  Then for any global $M$-invariant type $q \supseteq \text{tp}(b/M)$, there is a tree Morley sequence over $M$ $(b_{i}, b'_{i})_{i < \omega}$ starting with $(b,b')$ so that 
\begin{enumerate}
\item If $i \leq j$, then $b_{i}b'_{j} \equiv_{M} bb'$.
\item If $i > j$, then $b_{i} \models q|_{Mb'_{j}}$.
\end{enumerate}
\end{lem}

\begin{proof}
Fix $q \supseteq \text{tp}(b/M)$ and let $p(x;b) = \text{tp}(b'/Mb)$.  By recursion on $\alpha$, we will construct trees $(c_{\eta}^{\alpha},d^{\alpha}_{\eta})_{\eta \in \mathcal{T}_{\alpha}}$ so that, for all $\alpha$
\begin{enumerate}
\item If $\eta \in \mathcal{T}_{\alpha}$, then 
$$
c^{\alpha}_{\eta} \models q|_{Mc^{\alpha}_{\vartriangleright \eta}d^{\alpha}_{\vartriangleright \eta}}
$$
\item If $\eta \in \mathcal{T}_{\alpha}$, then 
$$
d^{\alpha}_{\eta} \models \bigcup_{\nu \unrhd \eta} p(x;c^{\alpha}_{\nu})
$$
\item $(c_{\eta}^{\alpha}, d_{\eta}^{\alpha})_{\eta \in \mathcal{T}_{\alpha}}$ is spread out and $s$-indiscernible over $M$
\item If $\beta < \alpha$ then $(c^{\alpha}_{\iota_{\beta \alpha}(\eta)}, d^{\alpha}_{\iota_{\beta \alpha}(\eta)}) = (c^{\beta}_{\eta}, d^{\beta}_{\eta})$ for all $\eta \in \mathcal{T}_{\beta}$.
\end{enumerate}
To start, define $(c^{0}_{\emptyset},d^{0}_{\emptyset}) = (b,b')$.  This defines $(c^{0}_{\eta},d^{0}_{\eta})_{\eta \in \mathcal{T}_{0}}$.  

Now suppose given $(c^{\alpha}_{\eta},d^{\alpha}_{\eta})_{\eta \in \mathcal{T}_{\alpha}}$.  Let $\langle (c^{\alpha}_{\eta,i},d^{\alpha}_{\eta,i}) : i < \omega \rangle$ be an $M$-invariant Morley sequence with $(c^{\alpha}_{\eta,0},d^{\alpha}_{\eta,0})_{\eta \in \mathcal{T}_{\alpha}} = (c^{\alpha}_{\eta},d^{\alpha}_{\eta})_{\eta \in \mathcal{T}_{\alpha}}$.  Pick $c_{*}$ so that 
$$
c_{*} \models q|_{M(c^{\alpha}_{\eta,i},d^{\alpha}_{\eta,i})_{\eta \in \mathcal{T}_{\alpha}, i < \omega}}.
$$
Then, by Lemma \ref{consistenttree}, we may choose $d_{*}$ so that 
$$
d_{*} \models \bigcup_{\substack{\eta \in \mathcal{T}_{\alpha} \\ i < \omega}} p(x;c^{\alpha}_{\eta, i }) \cup p(x;c_{*}).
$$
Define a tree $(e_{\eta}, f_{\eta})_{\eta \in \mathcal{T}_{\alpha+1}}$ by 
\begin{eqnarray*}
(e_{\emptyset},f_{\emptyset}) &=& (c_{*},d_{*}) \\
(e_{\langle i \rangle \frown \eta},f_{\langle i \rangle \frown \eta}) &=& (c^{\alpha}_{\eta, i},d^{\alpha}_{\eta, i}).
\end{eqnarray*}
Finally, let $(c^{\alpha+1}_{\eta}, d^{\alpha+1}_{\eta})_{\eta \in \mathcal{T}_{\alpha+1}}$ be a tree $s$-indiscernible over $M$ locally based on this tree.  By an automorphism, we may assume that $c^{\alpha+1}_{\iota_{\alpha\alpha+1}(\eta)} = c^{\alpha}_{\eta}$ for all $\eta \in \mathcal{T}_{\alpha}$.  This satisfies the requirements.  

Finally, arriving to stage $\delta$ for $\delta$ limit, we simply define $(c^{\delta}_{\eta}, d^{\delta}_{\eta})_{\eta \in \mathcal{T}_{\delta}}$ by stipulating $(c^{\delta}_{\iota_{\beta \delta}(\eta)},d^{\delta}_{\iota_{\beta \delta}(\eta)}) = (c^{\beta}_{\eta}, d^{\beta}_{\eta})$ for all $\beta < \delta$.  By the coherence condition (4), this is well-defined, and satisfies the requirements.  We conclude by extracting a Morley tree, by Lemma \ref{morleyextraction}.  
\end{proof}

\begin{thm} \label{itthmchar}
Suppose $T$ is a complete theory.  The following are equivalent:
\begin{enumerate}
\item $T$ is NSOP$_{1}$.
\item $\ind^{K}$ satisfies the independence theorem over models:  if $M \models T$, $a \equiv_{M} a'$, $a \ind^{K}_{M} b$, $a' \ind^{K}_{M} c$, and $b \ind^{K}_{M} c$, then there is $a''$ with $a''\equiv_{Mb} a$, $a'' \equiv_{Mc} a'$ and $a'' \ind^{K}_{M} bc$.  
\end{enumerate}
\end{thm}

\begin{proof}
(2)$\implies$(1) follows from \cite[Theorem 5.1]{ArtemNick}, using that $\ind^{i}$ implies $\ind^{K}$.  

(1)$\implies$(2):  Assume $T$ is NSOP$_{1}$.  Suppose $M \models T$, $a \equiv_{M} a'$, and $a \ind^{K}_{M} b$, $a' \ind^{K}_{M} c$ and $b \ind^{K}_{M} c$.  We must show there is $a''$ with $a'' \equiv_{Mb} a$, $a'' \equiv_{Mc} a'$ and $a'' \ind^{K}_{M} bc$.  Let $p_{0}(x;b) = \text{tp}(a/Mb)$ and $p_{1}(x;c) = \text{tp}(a'/Mc)$.  Suppose towards contradiction that $p_{0}(x;b) \cup p_{1}(x;c)$ Kim-forks over $M$.  Let $q \supseteq \text{tp}(b/M)$ be a global type finitely satisfiable in $M$.  In particular, $q$ is $M$-invariant so, by Lemma \ref{zigzag}, there is a tree Morley sequence over $M$, $(b_{i}, c_{i})_{i \in \mathbb{Z}}$ so that 
\begin{enumerate}[(a)]
\item If $i \leq j$, then $b_{i}c_{j} \equiv_{M} bc$.
\item If $i > j$, then $b_{i} \models q|_{Mc_{j}}$.
\end{enumerate}
Then both $(b_{2i}, c_{2i+1})_{i \in \mathbb{Z}}$ and $(b_{2i},c_{2i-1})_{i \in \mathbb{Z}}$ are tree Morley sequences over $M$ by Lemma \ref{concatenation}.  By (a), we know $p_{0}(x;b_{0}) \cup p_{1}(x;c_{1})$ Kim-forks over $M$ so 
$$
\bigcup_{i \in \mathbb{Z}} p_{0}(x;b_{2i}) \cup p_{1}(x;c_{2i+1})
$$
is inconsistent.  However, because $b_{0} \ind^{u}_{M} c_{-1}$ by (2), Proposition \ref{reduction}  gives that $p_{0}(x;b_{0}) \cup p_{1}(x;c_{-1})$ does not Kim-fork over $M$.  Therefore
$$
\bigcup_{i \in \mathbb{Z}} p_{0}(x;b_{2i}) \cup p_{1}(x;c_{2i-1})
$$
is consistent.  And this is a contradiction, as these two partial types are the same.  This completes the proof.  
\end{proof}

\begin{cor}
Suppose $T$ is NSOP$_{1}$, $M \models T$, $b \equiv_{M} b'$ and $b \ind^{K}_{M} b'$.  Then there is a tree Morley sequence $(b_{i})_{i < \omega}$ over $M$, with $b_{0} = b$ and $b_{1} = b'$.    
\end{cor}

\begin{proof}
Let $p(x;b) = \text{tp}(b'/Mb)$.  By induction on ordinals $\alpha \geq 1$, we will build trees $(b^{\alpha}_{\eta})_{\eta \in \mathcal{T}_{\alpha}}$ spread out and $s$-indiscernible over $M$ so that
\begin{enumerate}
\item $\nu \vartriangleleft \eta$ then $b^{\alpha}_{\nu}b^{\alpha}_{\eta} \equiv_{M} b'b$.   
\item If $1 \leq \beta < \alpha$, then $b^{\alpha}_{\iota_{\beta \alpha}(\eta)} = b^{\beta}_{\eta}$.  
\end{enumerate} 
To start, let $\overline{b} = (b_{i})_{i < \omega}$ be an $M$-invariant Morley sequence\textemdash as $b \ind^{K}_{M} b'$, we may assume this sequence is $Mb'$-indiscernible.  Define $(b^{1}_{\eta})_{\eta \in \mathcal{T}_{1}}$ by $b^{1}_{\emptyset} = b'$ and $b^{1}_{\langle i \rangle} = b_{i}$.  Then $(b^{1}_{\eta})_{\eta \in \mathcal{T}_{1}}$ is spread out and $s$-indiscernible over $M$ and clearly satisfies (1).  

Now suppose given $(b_{\eta}^{\alpha})_{\eta \in \mathcal{T}_{\alpha}}$.  Let $\langle (b^{\alpha}_{\eta,i})_{\eta \in \mathcal{T}_{\alpha}} : i < \omega \rangle$ be an $M$-invariant Morley sequence with $(b^{\alpha}_{\eta,0})_{\eta \in \mathcal{T}_{\alpha}} = (b^{\alpha}_{\eta})_{\eta \in \mathcal{T}_{\alpha}}$.  Choose $b'' \ind^{K}_{M} (b^{\alpha}_{\eta})_{\eta \in \mathcal{T}_{\alpha}}$ with 
$$
b'' \models \bigcup_{\eta \in \mathcal{T}_{\alpha}} p(x;b^{\alpha}_{\eta}),
$$
(this is possible by Remark \ref{strongconsistenttree}).  By the chain condition, we may assume the sequence $\langle (b^{\alpha}_{\eta,i})_{\eta \in \mathcal{T}_{\alpha}} : i < \omega \rangle$ is $Mb''$-indiscernible and that $b'' \ind^{K}_{M} (b^{\alpha}_{\eta,i})_{\eta \in \mathcal{T}_{\alpha}, i < \omega}$.  Define a tree $(c_{\eta})_{\eta \in \mathcal{T}_{\alpha+1}}$ by $c_{\emptyset} = b''$ and $c_{\langle i \rangle \frown \eta} = b^{\alpha}_{\eta, i}$.  Then let $(b^{\alpha+1}_{\eta})_{\eta \in \mathcal{T}_{\alpha+1}}$ be a tree which is $s$-indiscernible over $M$ and locally based on $(c_{\eta})_{\eta \in \mathcal{T}_{\alpha+1}}$.  By an automorphism, we may assume that $b^{\alpha+1}_{\iota_{\alpha \alpha+1}(\eta)} = b^{\alpha}_{\eta}$ for all $\eta \in \mathcal{T}_{\alpha}$.  This satisfies the requirements.

Finally, if $\delta$ is a limit and we are given $(b^{\alpha}_{\eta})_{\eta \in \mathcal{T}_{\alpha}}$ for all $\alpha < \delta$, define $(b_{\eta}^{\delta})_{\eta \in \mathcal{T}_{\delta}}$ as follows:  if $\eta \in \mathcal{T}_{\delta}$, choose any $\alpha < \delta$ and $\nu \in \mathcal{T}_{\alpha}$ so that $\eta = \iota_{\alpha \delta}(\nu)$.  Then define $b^{\delta}_{\eta} = b^{\alpha}_{\nu}$.  By the coherence condition, this is well-defined and clearly satisfies the requirements.  

To conclude, let $\kappa$ be big enough for Erd\H{o}s-Rado and consider $(b_{\eta}^{\kappa})_{\eta \in \mathcal{T}_{\kappa}}$ given by the above construction.  Apply Lemma \ref{morleyextraction} to find $(c_{\eta})_{\eta \in \mathcal{T}_{\omega}}$, a Morley tree over $M$, based on this tree.  By an automorphism, we may assume $c_{\zeta_{0}} = b$ and $c_{\zeta_{1}} = b'$.  The sequence $(c_{\zeta_{i}})_{i < \omega}$ is the desired tree Morley sequence.  
\end{proof}

\section{Forking and Witnesses} \label{forkingsection}

\subsection{Basic properties of forking}

\begin{defn} \label{forkingandidivdingdef}
\begin{enumerate}
\item The formula $\varphi(x;b)$ \emph{divides} over $A$ if there is an $A$-indiscernible sequence $\langle b_{i} : i < \omega \rangle$ with $b_{0} = b$ so that $\{\varphi(x;b_{i}) : i < \omega\}$ is inconsistent.  A type $p(x)$ divides over $A$ if it implies some formula that divides over $A$.  Write $a \ind^{d}_{A} B$ to mean that $\text{tp}(a/AB)$ does not divide over $A$.
\item The formula $\varphi(x;b)$ \emph{forks} over $A$ if $\varphi(x;b)$ implies a finite disjunction $\bigvee_{i} \psi_{i}(x;c_{i})$ where each $\psi_{i}(x;c_{i})$ divides over $A$.  A type $p(x)$ forks over $A$ if it implies a formula which forks over $A$.  We write $a \ind^{f}_{A} B$ to mean that $\text{tp}(a/AB)$ does not fork over $A$.  
\end{enumerate}
\end{defn}

The following facts about forking and dividing are easy and well-known -- see, e.g., \cite{grossberg2002primer} \cite{adler2005explanation}.  

\begin{fact} \label{forkingfacts}
The following are true with respect to an arbitrary theory:
\begin{enumerate}
\item $a \ind^{d}_{A} b$ if and only if, given any $A$-indiscernible sequence $I = \langle b_{i} : i < \omega \rangle$ with $b = b_{0}$, there is $a' \equiv_{Ab} a$ so that $I$ is $Aa'$-indiscernible.
\item $\ind^{f}$ is an invariant ternary relation on small subsets satisfying:
\begin{enumerate}
\item (Extension)  If $a \ind^{f}_{A} b$, then, for all $c$, there is $a' \equiv_{Ab} a$ so that $a' \ind^{f}_{A} bc$.  
\item (Base Monotonicity)  If $a \ind^{f}_{A} bc$ then $a \ind^{f}_{Ab} c$.  
\item (Left Transitivity)  If $a \ind^{f}_{Ab} c$ and $b \ind^{f}_{A} c$ then $ab \ind^{f}_{A} c$.  
\end{enumerate}
\item For any model $M$, 
$$
a \ind^{i}_{M} b \implies a \ind^{f}_{M} b \implies a \ind^{K}_{M} b.
$$
\end{enumerate}
\end{fact}

\begin{rem}
$\ind^{d}$ may fail to satisfy (2)(a) in an arbitrary theory, but always satisfies (2)(b) and (2)(c).  
\end{rem}

As a warm-up to the theorem in the next subsection, we note that these properties easily give a weak form of transitivity for $\ind^{K}$:

\begin{lem}
Suppose $a \ind^{d}_{M} bc$ and $b \ind^{K}_{M} c$.  Then $ab \ind^{K}_{M} c$.  
\end{lem}

\begin{proof}
Assume $a \ind^{d}_{M} bc$ and $b \ind^{K}_{M} c$.  As $b\ind^{K}_{M} c$, for any $M$-invariant Morley sequence $I = (c_{i})_{i < \omega}$ with $c_{0} = c$, there is $I' = (c'_{i})_{i < \omega}$ with $I' \equiv_{Mc_{0}} I$ which is, moreover, $Mb$-indiscernible.  By base monotonicity of $\ind^{d}$, $a \ind^{d}_{Mb} c$ so there is an $Mab$-indiscernible sequence $I'' = (c''_{i})_{i < \omega}$ with $I'' \equiv_{Mbc} I'$.  Thus $I''$ is an $M$-invariant Morley sequence with $c''_{0} = c$ which is $Mab$-indiscernible.  By an automorphism, we obtain $a'b' \equiv_{Mc} ab$ so that $I$ is $Ma'b'$-indiscernible.  As $I$ was an arbitrary $M$-invariant Morley sequence over $M$, it follows that $ab \ind^{K}_{M} c$.  
\end{proof}

\subsection{Morley Sequences}

\begin{defn}
Suppose $M \models T$.  An $\ind^{K}$-Morley sequence over $M$ is an $M$-indiscernible sequence $\langle b_{i} : i < \omega \rangle$ satisfying $b_{i} \ind^{K}_{M} b_{<i}$.  Likewise, an $\ind^{f}$-Morley sequence over $M$ is an $M$-indiscernible sequence $\langle b_{i} : i < \omega \rangle$ satisfying $b_{i} \ind^{f}_{M} b_{<i}$.  
\end{defn}

\begin{lem}\label{consistentindk}
Suppose the complete theory $T$ is NSOP$_{1}$, $M \models T$, and $\varphi(x;b)$ does not Kim-divide over $M$.  Then for any $\ind^{K}$-Morley sequence $\langle b_{i} : i < \omega\rangle$ over $M$ with $b_{0} = b$, $\{\varphi(x;b_{i}) : i < \omega \}$ is non-Kim-forking over $M$.  In particular, this set of formulas is consistent.
\end{lem}

\begin{proof}
By induction on $n$, we will show that $\{\varphi(x;b_{i}) : i \leq n\}$ is non-Kim-forking over $M$.  The case of $n = 0$ follows by hypothesis.  Now suppose $\{\varphi(x;b_{i}) : i \leq n\}$ is non-Kim-forking over $M$.  Fix $\sigma \in \text{Aut}(\mathbb{M}/M)$ with $\sigma(b_{0}) = b_{n+1}$.  Let $a \models \{\varphi(x;b_{i}) : i \leq n\}$ with $a \ind^{K}_{M} b_{\leq n}$.  Then $\sigma(a) \equiv_{M} a$ and $\models \varphi(\sigma(a);b_{n+1})$.  We know $b_{n+1} \ind^{K}_{M} b_{\leq n}$ so by the independence theorem, there is $a'$ with $a' \equiv_{Mb_{\leq n}} a$ and $a' \equiv_{Mb_{n+1}} \sigma(a)$ so that $a' \ind^{K}_{M} b_{\leq n+1}$.  As $a' \models \{\varphi(x;b_{i}) : i \leq n+1\}$, this completes the induction.  The lemma, then, follows by compactness.  
\end{proof}

\begin{thm}\label{kimslemmaforforking}
Suppose the complete theory $T$ is NSOP$_{1}$ and $M \models T$.  The following are equivalent:
\begin{enumerate}
\item $\varphi(x;b)$ Kim-divides over $M$.
\item For some $\ind^{f}$-Morley sequence $(b_{i})_{i  < \omega}$ over $M$ with $b_{0} = b$, $\{\varphi(x;b_{i}) : i < \omega\}$ is inconsistent.
\item For every $\ind^{f}$-Morley sequence $(b_{i})_{i < \omega}$ over $M$ with $b_{0} = b$, $\{\varphi(x;b_{i}) : i < \omega\}$ is inconsistent. 
\end{enumerate} 
\end{thm}

\begin{proof}
(3)$\implies$(2) is immediate, as a Morley sequence in a global $M$-invariant type is, in particular, an $\ind^{f}$-Morley sequence and such sequences always exist.  

(2)$\implies$(1) follows from Lemma \ref{consistentindk}, as an $\ind^{f}$-Morley sequence is an $\ind^{K}$-Morley sequence.

Now we show (1)$\implies$(3).  Suppose not\textemdash assume that $\varphi(x;b)$ is a formula which Kim-divides over $M$, but there is some $\ind^{f}$-Morley sequence over $M$ with $b_{0} = b$ so that $\{\varphi(x;b_{i}) : i < \omega\}$ is consistent.  By induction on $n$, we will construct a sequence $(b'_{i})_{i \leq n}$ and an elementary chain $(N_{i})_{i \leq n}$ so that 
\begin{enumerate}
\item For all $n < \omega$, $b_{0}\ldots b_{n} \equiv_{M} b'_{0} \ldots b'_{n}$.
\item For all $n < \omega$, $M \prec N_{n} \prec N_{n+1} \prec \mathbb{M}$.
\item For all $n < \omega$, $b'_{n} \ind^{f}_{M} N_{n}$.
\item For all $n < \omega$, $b'_{n} \in N_{n+1}$.  
\end{enumerate}
For the $n = 0$ case, set $b'_{0} = b_{0}$ and $N_{0} = M$.  Now suppose we are given $(N_{i})_{i \leq n}$ and $(b'_{i})_{i \leq n}$.  Let $N_{n+1}$ be an arbitrary (small) elementary extension of $N_{n}$ which contains $b'_{n}$.  By invariance and extension of $\ind^{f}$, we may choose some $b'_{n+1}$ so that $b'_{0}\ldots b'_{n+1} \equiv_{M} b_{0}\ldots b_{n+1}$ and $b'_{n+1} \ind^{f}_{M} N_{n+1}$.  This completes the recursion.  

Set $N = \bigcup_{i < \omega} N_{i}$.  

\textbf{Claim 1}:  For all $n < \omega$, $(b'_{i})_{i \geq n} \ind^{f}_{M} N_{n}$.  

\emph{Proof of claim}:  Fix $n$.  We will argue by induction on $k$ that $b'_{n}\ldots b'_{n+k} \ind^{f}_{M} N_{n}$.  For $k = 0$, this is by construction.  Assume it has been proven for $k$.  Note that $b'_{n+k+1} \ind^{f}_{M} N_{n+k+1}$.  Now $N_{n}$ and $(b'_{i})_{i \leq n+k}$ are contained in $N_{n+k+1}$ so, in particular, we have $b'_{n+k+1} \ind^{f}_{M} N_{n}b'_{0}\ldots b'_{n+k}$.  By base monotonicity, we have 
$$
b'_{n+k+1} \ind^{f}_{Mb'_{0}\ldots b'_{n+k}} N_{n}.
$$
This, together with the induction hypothesis, implies 
$$
b'_{0}\ldots b'_{n+k+1} \ind^{f}_{M} N_{n}
$$
by left-transitivity.  The claim follows by finite character.  \qed

Let $\mathcal{D}$ be any non-principal ultrafilter on $\{b'_{i} : i < \omega\}$ and $(c_{i})_{i < \omega}$ be a sequence chosen so that $c_{i} \models \text{Av}(\mathcal{D},Nc_{<i})$, i.e. a Morley sequence over $N$ in the global $(b'_{i})_{i < \omega}$-invariant type $\text{Av}(\mathcal{D},\mathbb{M})$.

\textbf{Claim 2}:  $(c_{i})_{i < \omega} \ind^{f}_{M} N$.  

\emph{Proof of claim}:  Suppose not.  Then by finite character, there is $l$ so that $(c_{i})_{i < l} \nind^{f}_{M} N$ so we choose some $\varphi(x_{0}, \ldots, x_{l-1};d) \in \text{tp}(c_{0},\ldots, c_{l-1}/N)$ which forks over $M$.  Choose $n$ so that $d \in N_{n}$.  By definition of average type, we may find $i_{0} > \ldots > i_{l-1} > n$ so that $\mathbb{M} \models \varphi(b'_{i_{0}}, \ldots, b'_{i_{l-1}};d)$.  Then $(b'_{i})_{i \geq n} \nind^{f}_{M} N_{n}$, contradicting Claim 1.  \qed

Let $q \supseteq \text{tp}((c_{i})_{i < \omega}/M)$ be a global $M$-invariant and indiscernible type, as in Definition \ref{indiscernibletype}.  Let $\langle (c_{k,i})_{i < \omega } : k < \omega \rangle$ be a Morley sequence over $M$ in $q$ with $c_{0,i} = c_{i}$ for all $i < \omega$.  By Lemma \ref{pathtype}, $(\overline{c}_{k})_{k < \omega}$ is a mutually-indiscernible array over $M$.  By Claim 2, we know $\overline{c}_{0} \ind^{f}_{M} N$ hence $\overline{c}_{0} \ind^{K}_{M} N$, so we may assume the sequence $(\overline{c}_{k})_{k < \omega}$ is $N$-indiscernible by symmetry.  We know that $\{\varphi(x;b_{i}) : i < \omega\}$ is consistent so $\{\varphi(x;b'_{i}) : i < \omega\}$ is consistent, and therefore $\{\varphi(x;c_{0,i}) : i < \omega\}$ is consistent.  The sequence $(c_{0,i})_{i < \omega}$ is also an $N$-invariant Morley sequence so $\varphi(x;c_{0,0})$ does not Kim-divide over $N$.  But as $c_{0,0} \equiv_{M} b$, $(c_{i,0})_{i < \omega}$ is an $M$-invariant Morley sequence over $M$, and $\varphi(x;b)$ Kim-divides over $M$, we know that $\{\varphi(x;c_{i,0}) : i < \omega\}$ is inconsistent.  

Let $(\overline{d}_{i})_{i < \omega}$ be a mutually indiscernible array over $N$, locally based on $(\overline{c}_{i})_{i < \omega}$ (exists by \cite[Lemma 1.2]{ChernikovNTP2}), with $(\overline{d}_{i})_{i < \omega}$ an $N$-indiscernible sequence.  By Lemma \ref{pathtype}, we have $(d_{i,0})_{i < \omega} \equiv_{M} (c_{i,0})_{i < \omega}$.  Also, because $(\overline{c}_{i})_{i < \omega}$ was taken to be $N$-indiscernible and $\overline{c}_{0}$ was an $N$-invariant Morley sequence, we know each $\overline{c}_{i}$ is an $N$-invariant Morley sequence, and therefore each $\overline{d}_{i}$ is an $N$-invariant Morley sequence.  By choice of the array, $\{\varphi(x;d_{i,j}) : j < \omega\}$ is consistent for all $i$, so $\varphi(x;d_{i,0})$ does not Kim-divide over $N$.  Also, we have $\{\varphi(x;d_{i,0}) : i < \omega\}$ is inconsistent.  Thus, to derive a contradiction, it suffices by Lemma \ref{consistentindk} to establish the following:

\textbf{Claim 3}:  $(d_{i,0})_{i < \omega}$ is an $\ind^{K}$-Morley sequence over $N$.  

\emph{Proof of claim}:  As the $(d_{i,j})_{i,j < \omega}$ forms a mutually indiscernible array over $N$, we know that for each $i < \omega$, $\overline{d}_{i}$ is an $N\overline{d}_{<i}$-indiscernible sequence.  But it is also an $N$-invariant Morley sequence so $\overline{d}_{<i} \ind^{K}_{N} d_{i,0}$.  By symmetry, this yields in particular that $d_{i,0} \ind^{K}_{N} d_{0,0} \ldots d_{i-1,0}$.  This proves the claim and completes the proof.  \end{proof}

\subsection{Witnesses}  

\begin{defn}
Suppose $M$ is a model and $(a_{i})_{i < \omega}$ is an $M$-indiscernible sequence.
\begin{enumerate}
\item Say $(a_{i})_{i < \omega}$ is a \emph{witness} for Kim-dividing over $M$ if, whenever $\varphi(x;a_{0})$ Kim-divides over $M$, $\{\varphi(x;a_{i}) : i <\omega\}$ is inconsistent.
\item Say $(a_{i})_{i < \omega}$ is a \emph{strong witness} to Kim-dividing over $M$ if, for all $n$, the sequence $\langle (a_{n \cdot i}, a_{n \cdot i + 1}, \ldots, a_{n \cdot i + n-1}) : i < \omega \rangle$ is a witness to Kim-dividing over $M$.  
\end{enumerate}
\end{defn}

Corollary \ref{kimslemmafortms} and Lemma \ref{concatenation} show that tree Morley sequences are strong witnesses for Kim-dividing.  The following proposition shows the converse, giving a characterization of strong witnesses as exactly the tree Morley sequences.  

\begin{prop} \label{witnesschar}
Suppose $T$ is NSOP$_{1}$ and $M \models T$.  Then $(a_{i})_{i < \omega}$ is a strong witness for Kim-dividing over $M$ if and only if $(a_{i})_{i < \omega}$ is a tree Morley sequence over $M$.  
\end{prop}

\begin{proof}
If $(a_{i})_{i < \omega}$ is a tree Morley sequence, then $(a_{n \cdot i}, a_{n \cdot i + 1}, \ldots, a_{n \cdot i + (n-1)})_{i < \omega}$ is also a tree Morley sequence over $M$ by Lemma \ref{concatenation}.  It follows that $(a_{i})_{i < \omega}$ is a strong witness to Kim-dividing by Corollary \ref{kimslemmafortms}.

For the other direction, suppose $(a_{i})_{i < \omega}$ is a strong witness to Kim-dividing over $M$.  Given an arbitrary cardinal $\kappa$, we may, by compactness, stretch the sequence to $I = (a_{i})_{i \in \kappa \setminus \text{lim}(\kappa)}$ which is still a strong witness to Kim-dividing over $M$.  By recursion on $\alpha < \kappa$, we will construct trees $(a^{\alpha}_{\eta})_{\eta \in \mathcal{T}_{\alpha}}$ so that 
\begin{enumerate}
\item For all $i \in \alpha \setminus \text{lim}(\alpha)$, $a^{\alpha}_{\zeta_{i}} = a_{i}$ and also $a^{\alpha}_{\emptyset} = a_{\alpha}$ for $\alpha$ successor.
\item $I_{>\alpha}$ is $M(a^{\alpha}_{\eta})_{\eta \in \mathcal{T}_{\alpha}}$-indiscernible.
\item $(a^{\alpha}_{\eta})_{\eta \in \mathcal{T}_{\alpha}}$ is spread out over $M$ and $s$-indiscernible over $MI_{>\alpha}$.
\item If $\alpha < \beta$, then $a^{\alpha}_{\eta} = a^{\beta}_{\iota_{\alpha \beta}(\eta)}$ for all $\eta \in \mathcal{T}_{\alpha}$.  
\end{enumerate}
For the case $\alpha = 0$, put $a^{0}_{\emptyset} = a_{0}$.  This satisfies the demands.  Suppose $(a^{\beta}_{\eta})_{\eta \in \mathcal{T}_{\beta}}$ has been defined for all $\beta \leq \alpha$.  As $I_{> \alpha} = (a_{i})_{i > \alpha}$ is also a strong witness to Kim-dividing over $M$ and is $M(a^{\alpha}_{\eta})_{\eta \in \mathcal{T}_{\alpha}}$-indiscernible, we have 
$$
(a^{\alpha}_{\eta})_{\eta \in \mathcal{T}_{\alpha}} \ind^{K}_{M} I_{> \alpha}.
$$
Let $J = \langle (a^{\alpha}_{\eta,i})_{\eta \in \mathcal{T}_{\alpha}} : i < \omega \rangle$ be a Morley sequence in an $M$-invariant type with $a^{\alpha}_{\eta,0} = a^{\alpha}_{\eta}$ for all $\eta \in \mathcal{T}_{\alpha}$.  By symmetry, $I_{> \alpha} \ind^{K}_{M} (a^{\alpha}_{\eta})_{\eta \in \mathcal{T}_{\alpha}}$ so we may assume $J$ is $MI_{>\alpha}$-indiscernible.  Let $I'_{>\alpha} = \langle a'_{i} : i \in \kappa \setminus (\lim(\kappa) \cup \alpha) \rangle$ be an $MJ$-indiscernible sequence locally based on $I_{>\alpha}$.  Note that $J$ is $MI'_{>\alpha}$-indiscernible as well and $I'_{>\alpha}(a^{\alpha}_{\eta,i})_{\eta \in \mathcal{T}_{\alpha+1}} \equiv_{M} I_{>\alpha}(a^{\alpha}_{\eta})_{\eta \in \mathcal{T}_{\alpha}}$ for all $i < \omega$.  

Define the tree $(c_{\eta})_{\eta \in \mathcal{T}_{\alpha+1}}$ by $c_{\emptyset} = a'_{\alpha+1}$ and $c_{\langle i \rangle \frown \eta} = a^{\alpha}_{\eta,i}$ for all $i < \omega$ and $\eta \in \mathcal{T}_{\alpha}$.  Let $(c'_{\eta})_{\eta \in \mathcal{T}_{\alpha+1}}$ be an a tree $s$-indiscernible over $MI'_{>\alpha}$ locally based on $(c_{\eta})_{\eta \in \mathcal{T}_{\alpha}}$.  The sequence $I'_{>\alpha}$ is $M(c'_{\unrhd \langle i \rangle})_{i < \omega}$-indiscernible and, by the construction of $(c_{\eta})_{\eta \in \mathcal{T}_{\alpha}}$, we have also $I'_{>\alpha}(c'_{\langle 0 \rangle \frown \eta})_{\eta \in \mathcal{T}_{\alpha}} \equiv_{M} I_{>\alpha}(a^{\alpha}_{\eta})_{\eta \in \mathcal{T}_{\alpha}}$.  Let $\sigma \in \text{Aut}(\mathbb{M}/M)$ be an automorphism with $\sigma(I'_{>\alpha}(c'_{\langle 0 \rangle \frown \eta})_{\eta \in \mathcal{T}_{\alpha}})= I_{>\alpha}(a^{\alpha}_{\eta})_{\eta \in \mathcal{T}_{\alpha}}$ and define the tree $(a^{\alpha+1}_{\eta})_{\eta \in \mathcal{T}_{\alpha+1}}$ by setting $a^{\alpha+1}_{\eta} = \sigma(c'_{\eta})$ for all $\eta \in \mathcal{T}_{\alpha+1}$.  Note in particular, this definition gives $a^{\alpha+1}_{\iota_{\alpha \alpha+1}(\eta)} = a^{\alpha+1}_{0 \frown \eta} = a^{\alpha}_{\eta}$ for all $\eta \in \mathcal{T}_{\alpha}$.  The tree we just constructed satisfies the demands, completing the successor step.  

Now suppose given $(a^{\beta}_{\eta})_{\eta \in \mathcal{T}_{\beta}}$ for all $\beta < \delta$, where $\delta$ is a limit.  Define $(a^{\delta}_{\eta})_{\eta \in \mathcal{T}_{\delta}}$ by setting $a^{\delta}_{\iota_{\alpha \delta}(\eta)} = a^{\alpha}_{\eta}$ for all $\alpha < \delta$ and $\eta \in \mathcal{T}_{\alpha}$.  Condition (3) guarantees that this is well-defined.  

Taking $\kappa$ to be sufficiently large, we may extract a Morley tree from the tree we just constructed by Lemma \ref{morleyextraction} -- in particular, we may obtain a Morley tree $(b_{\eta})_{\eta \in \mathcal{T}_{\omega}}$ so that $(b_{\zeta_{i}})_{i < \omega} \equiv_{M} (a_{i})_{i < \omega}$.  This shows that $(a_{i})_{i < \omega}$ is a tree Morley sequence over $M$.  
\end{proof}

\begin{cor} \label{forkingimpliestree}
Suppose $T$ is NSOP$_{1}$ and $M \models T$.  An $\ind^{f}$-Morley sequence over $M$ is a tree Morley sequence.  
\end{cor}

\begin{proof}
Suppose $(a_{i})_{i < \omega}$ is an $\ind^{f}$-Morley sequence over $M$.  Arguing as in Claim 1 of the proof of Theorem \ref{kimslemmaforforking}, for all $n < \omega$, $a_{>n} \ind^{f}_{M} a_{\leq n}$.  Therefore, 

\noindent $\langle (a_{n \cdot i}, a_{n \cdot i + 1}, \ldots, a_{n \cdot i + n-1}) : i < \omega \rangle$ is an $\ind^{f}$-Morley sequence over $M$, hence a witness to Kim-dividing over $M$ by Theorem \ref{kimslemmaforforking}.  This shows $(a_{i})_{i < \omega}$ is a strong witness to Kim-dividing over $M$.  By Proposition \ref{witnesschar}, $(a_{i})_{i < \omega}$ is a tree Morley sequence over $M$.  
\end{proof}

In any theory, if $(a_{i})_{i < \omega}$ is an $\ind^{f}$-Morley sequence over $A$, then, as the proof of Corollary \ref{forkingimpliestree} shows, that $a_{>n} \ind^{f}_{A} a_{\leq n}$ for all $n < \omega$.  As base monotonicity and left-transitivity do not necessarily hold for $\ind^{K}$, we give a Morley sequence with this stronger behavior a name:  

\begin{defn}
Say the $M$-indiscernible sequence $(a_{i})_{i < \omega}$ is a \emph{total} $\ind^{K}$\emph{-Morley sequence} if $a_{> n} \ind^{K}_{M} a_{\leq n}$ for all $n < \omega$.  
\end{defn}

\begin{quest}
Suppose $T$ is NSOP$_{1}$, $M \models T$, and $I = (a_{i})_{i < \omega}$ is a total $\ind^{K}$-Morley sequence over $M$.  Is $I$ a tree Morley sequence over $M$?
\end{quest}

\section{Characterizing NSOP$_{1}$ and Simple Theories}\label{mtsection}

\subsection{The Main Theorem}

Before continuing with the rest of the paper, we pause to take stock of what has been shown:

\begin{thm}\label{mainthm}
The following are equivalent for the complete theory $T$:
\begin{enumerate}
\item $T$ is NSOP$_{1}$
\item Ultrafilter independence of higher formulas:  for every model $M \models T$, and ultrafilters $\mathcal{D}$ and $\mathcal{E}$ on $M$ with $\text{Av}(\mathcal{D},M) = \text{Av}(\mathcal{E},M)$, $(\varphi,M, \mathcal{D})$ is higher if and only if $(\varphi, M,\mathcal{E})$ is higher 
\item Kim's lemma for Kim-dividing:  For every model $M \models T$ and $\varphi(x;b)$, if $\varphi(x;y)$ $q$-divides for some global $M$-invariant $q \supseteq \text{tp}(b/M)$, then $\varphi(x;y)$ $q$-divides for \emph{every} global $M$-invariant $q \supseteq \text{tp}(b/M)$.  
\item Local character:  for some infinite cardinal $\kappa$, there cannot be a sequence $\langle N_{i},\varphi_{i}\left(x,y_{i}\right),c_{i} :i<\kappa \rangle$
such that $\langle N_{i} :i<\kappa \rangle$ is an increasing continuous
sequence of models of $T$, $\varphi_{i}(x,y_{i})$ is
a formula over $N_{i}$, $c_{i}\in N_{i+1}$, such that $\varphi_{i}(x,c_{i})$
Kim-forks over $N_{i}$ and $\{\varphi(x,c_{i}):i<\kappa\}$
is consistent. 
\item Symmetry over models:  for every $M \models T$, then $a \ind^{K}_{M} b$ if and only if $b \ind^{K}_{M} a$.
\item Independence theorem over models:  if $M \models T$, $a \equiv_{M} a'$, $a \ind^{K}_{M} b$, $a' \ind^{K}_{M} c$, and $b \ind^{K}_{M} c$, then there is $a''$ with $a''\equiv_{Mb} a$, $a'' \equiv_{Mc} a'$ and $a'' \ind^{K}_{M} bc$. 
\end{enumerate}
\end{thm}

\begin{proof}
(1)$\iff$(2)$\iff$(3) is Theorem \ref{kimslemmaforindk}.  

(1)$\iff$(4) is Corollary \ref{boundedweight}.

(1)$\iff$(5) is Theorem \ref{symmetrycharthm}.  

(1)$\iff$(6) is Theorem \ref{itthmchar}.  
\end{proof}

\subsection{Simplicity within the class of NSOP$_{1}$ theories}

\begin{defn} \cite[Section 6]{ChernikovNTP2} \label{simplecosimpledef} 
Suppose $p(x)$ is a partial type over the set $A$.
\begin{enumerate}
\item We say $p$ is a \emph{simple type} if there is no $\varphi(x;y)$, $(a_{\eta})_{\eta \in \omega^{<\omega}}$ and $k < \omega$ so that $\{\varphi(x;a_{\eta \frown \langle i \rangle}) : i < \omega\}$ is $k$-inconsistent for all $\eta \in \omega^{<\omega}$ and $p(x) \cup \{\varphi(x;a_{\eta | i}) : i < \omega\}$ is consistent for all $\eta \in \omega^{\omega}$.  Equivalently, $p(x)$ is simple if, whenever $B \supseteq A$, $q \in S(B)$, and $p \subseteq q$, then $q$ does not divide over $AB'$ for some $B' \subseteq B$, $|B'| \leq |T|$ (for the definition of dividing, see Definition \ref{forkingandidivdingdef} above).  
\item We say $p(x)$ is a \emph{co-simple type} if there is no formula $\varphi(x;y) \in L(A)$ for which there exists $(a_{\eta})_{\eta \in \omega^{<\omega}}$ and $k < \omega$ so that $\{\varphi(x;a_{\eta \frown \langle i \rangle}) : i < \omega\}$ is $k$-inconsistent for all $\eta \in \omega^{<\omega}$ and $\{\varphi(x;a_{\eta | i}) : i < \omega\}$ is consistent for all $\eta \in \omega^{\omega}$ and moreover $a_{\eta} \models p$ for all $\eta \in \omega^{<\omega}$.  
\end{enumerate}
\end{defn}

\begin{prop}\label{simpletype}
Assume $T$ is NSOP$_{1}$ and let $\pi(x)$ be a partial type over $A$.
\begin{enumerate}
\item  Assume that for any $\varphi(x;a)$ and any model $M \supseteq A$, $\pi(x) \cup \{\varphi(x;a)\}$ divides over $M$ if and only if $\pi \cup \{\varphi(x;a)\}$ Kim-divides over $M$.  Then $\pi(x)$ is a simple type.  
\item Assume that if $M \supseteq A$, then for any $a$ and for any $b \models \pi(x)$, $a \ind^{f}_{M} b$ if and only if $a \ind^{K}_{M} b$.  Then $\pi$ is a co-simple type.
\end{enumerate}
\end{prop}

\begin{proof}
Fix a Skolemization $T^{Sk}$ of $T$.  Throughout the proof, indiscernibility will be with respect to the language $L^{Sk}$ of the Skolemization.
(1)  Suppose $\pi$ is not simple.  Then by compactness, there is a formula $\varphi(x;y)$ over $A$ and a tree $(a_{\eta})_{\eta \in \omega^{<\omega + 1}}$ $s$-indiscernible over $A$ so that for some $k < \omega$
\begin{itemize}
\item For all $\eta \in \omega^{\omega + 1}$, $\pi (x) \cup \{\varphi(x;a_{\eta | \alpha}) : \alpha < \omega + 1\}$ is consistent 
\item For all $\eta \in \omega^{<\omega + 1}$, $\{\varphi(x;a_{\eta \frown \alpha}): \alpha < \omega\}$ is $k$-inconsistent.
\end{itemize}
Moreover we may assume $(a_{0^{\alpha}} : \alpha < \omega + 1)$ is an $A$-indiscernible sequence.  Let $b \models \pi(x) \cup \{\varphi(x;a_{0^{\alpha}}) : \alpha < \omega + 1\}$.  By Ramsey, compactness, and automorphism, we may assume $(a_{0^{\alpha}} : \alpha < \omega + 1)$ is $Ab$-indiscernible.  Let $C = \{a_{0^{\alpha}} : \alpha < \omega\}$.  Then $s$-indiscernibility implies $(a_{0^{\omega} \frown \beta} : \beta < \omega)$ is indiscernible over $A \cup C$ and $\{\varphi(x;a_{0^{\omega} \frown \beta}) : \beta < \omega\}$ is $k$-inconsistent by our assumption.  As $b \models \varphi(x;a_{0^{\omega} \frown \langle0\rangle})$, we have $b \nind^{d}_{AC} a_{0^{\omega} \frown \langle 0 \rangle}$.  But by indiscernibility, $a_{0^{\omega} \frown \langle 0 \rangle} \ind^{u}_{AC} b$ so in particular $a_{0^{\omega} \frown \langle 0 \rangle} \ind^{K}_{M} b$ and $b\ind^{K}_{M} a_{0^{\omega} \frown \langle 0 \rangle}$, where $M = Sk(AC)$, by symmetry.

(2)  We argue similarly.  Suppose $(a_{\eta})_{\eta \in \omega^{<\omega + 1}}$ is a collection of realizations of $\pi$, forming a tree $s$-indiscernible over $A$, with respect to which $\varphi(x;y)$ witnesses that $\pi$ is not co-simple.  Let $a \models \{\varphi(x;b_{0^{\alpha}}) : \alpha < \omega + 1\}$.  By Ramsey, compactness, and automorphism, we may assume $(b_{0^{\alpha}} : \alpha < \omega + 1)$ is a $Ba$-indiscernible sequence.  Setting $M = \text{Sk}(A(b_{0^{\alpha}})_{\alpha < \omega})$, we have $a \nind^{d}_{M} b_{0^{\omega} \frown \langle 0 \rangle}$ but $b_{0^{\omega} \frown \langle 0 \rangle} \ind^{u}_{M} a$ so $a \ind^{K}_{M} b_{0^{\omega} \frown \langle 0 \rangle}$.  
\end{proof}

In a similar vein, we have:

\begin{prop}\label{forkequalskimfork}
The complete theory $T$ is simple if and only if $T$ is NSOP$_{1}$ and $\ind^{f} = \ind^{K}$ over models.  
\end{prop}

\begin{proof}
If $T$ is simple, then $\ind^{f} = \ind^{K}$ over models by Kim's lemma for simple theories \cite[Proposition 2.1]{kim1998forking}, as a Morley sequence in a global invariant type is, in particular, a Morley sequence in the sense of non-forking.  On the other hand, by \cite[Theorem 2.4]{kim2001simplicity} forking is symmetric if and only if $T$ is simple and, by \cite[Lemma 6.16]{ChernikovNTP2}, we even have that if forking is symmetric over models then $T$ is simple.  If $T$ is NSOP$_{1}$, then $\ind^{K}$ is symmetric so $\ind^{K} = \ind^{f}$ implies $T$ is simple.  
\end{proof}

We also can give an interesting new proof of the following well-known fact:

\begin{cor}
The complete theory $T$ is simple if and only if $T$ is NSOP$_{1}$ and NTP$_{2}$.  
\end{cor}

\begin{proof}
In an NTP$_{2}$ theory, if $\varphi(x;b)$ divides over a model $M$, there is a Morley sequence sequence over $M$ in some global $M$-finitely satisfiable type witnessing this \cite[Lemma 3.14]{chernikov2012forking}.  So $\ind^{d} = \ind^{K}$, which implies $T$ is simple.  
\end{proof}

\begin{defn} \cite[Definition 2.5]{yaacov2014independence}
We say $(a_{i})_{i \in \kappa}$ is a \emph{universal Morley sequence} in $p \in S(A)$ if 
\begin{itemize}
\item $(a_{i})_{i \in \kappa}$ is indiscernible with $a_{i} \models p$
\item If $\varphi(x;y) \in L(A)$ and $\varphi(x;a_{0})$ divides over $A$ then $\{\varphi(x;a_{i}) : i \in \kappa\}$ is inconsistent.  
\end{itemize}
\end{defn}

\begin{prop}\label{nomorley}
Suppose $T$ is NSOP$_{1}$.  Then $T$ is simple if and only if, for any $M \models T$ and $p(x) \in S(M)$, there is a universal Morley sequence in $p$.    
\end{prop}

\begin{proof}
If $T$ is simple, then in any type $p(y) \in S(M)$, there is a $\ind^{f}$-Morley sequence in $p(y)$.  By Kim's lemma for simple theories \cite[Proposition 2.1]{kim1998forking}, this is a universal Morley sequence in $p$.  

If $T$ is not simple, then there is some formula $\varphi(x;b) \in L(Mb)$ which divides over $M$ but does not Kim-divide over $M$, by Proposition \ref{forkequalskimfork}.  Suppose there is a universal Morley sequence in $\text{tp}(b/M)$\textemdash by compactness we can take it to be $(b_{i})_{i \in \mathbb{Z}}$ indexed by $\mathbb{Z}$.  Then given $i \in \mathbb{Z}$, we have $b_{<i}$ is $Mb_{i}$-indiscernible so $b_{<i} \ind^{d}_{M} b_{i}$ so $b_{i} \ind^{K}_{M} b_{<i}$ by symmetry.  So $(b_{i})_{i \in \mathbb{Z}}$ is an $\ind^{K}$-Morley sequence.  By Lemma \ref{consistentindk}, $\{\varphi(x;b_{i}) : i \in \mathbb{Z}\}$ is consistent.  But $\varphi(x;b)$ divides over $M$ and $(b_{i})_{i \in \mathbb{Z}}$ is a universal Morley sequence so $\{\varphi(x;b_{i}) : i \in \mathbb{Z}\}$ is inconsistent.  This is a contradiction.  
\end{proof}

If $a \ind^{K}_{M} bb'$, it does not always make sense to ask if $a \ind^{K}_{Mb} b'$, since it is not always the case that $\text{tp}(b'/Mb)$ extends to a global $Mb'$-invariant type.  This can occur, however, whenever $Mb'$ is a model, for instance.  Say $\ind^{K}$ satisfies base monotonicity \emph{over models} if, whenever $a \ind^{K}_{M} Nb$ where $M,N \models T$, then $a \ind^{K}_{N} b$.  

\begin{prop}
The NSOP$_{1}$ theory $T$ is simple if and only if $\ind^{K}$ satisfies base monotonicity over models.
\end{prop}

\begin{proof}
It $T$ is simple, this follows from Proposition \ref{forkequalskimfork}, using Fact \ref{forkingfacts}(2b).  On the other hand, suppose $\ind^{K}$ satisfies base monotonicity over models.  We will show that $\ind^{K} = \ind^{d}$ over models.  It follows then that $T$ is simple, by Proposition \ref{forkequalskimfork}.  So suppose towards contradiction that $a \ind^{K}_{M} b$ but $a \nind^{d}_{M} b$, witnessed by $\varphi(x;b) \in \text{tp}(a/Mb)$ and $I =(b_{i})_{i < \omega + 1}$ an $M$-indiscernible sequence with $b_{\omega} = b$ and $\{\varphi(x;b_{i}) : i < \omega+1\}$ inconsistent.  Fix a Skolemization $T^{Sk}$ of $T$.  By Ramsey and automorphism, we may assume $(b_{i} : i < \omega+1)$ is $L^{Sk}$-indiscernible over $M$.  As $a \ind^{K}_{M} b$, we may, by extension, assume $a \ind^{K}_{M} \text{Sk}(MI)$.  Let $N = Sk(MI_{<\omega})$.  By base monotonicity over models, we have $a \ind^{K}_{N} b$.  But stretching $I$ to $(b_{i})_{i < \omega + \omega}$, we have that $(b_{\omega + i})_{i < \omega}$ is a $N$-invariant Morley sequence (in the reverse order) in $\text{tp}(b/N)$ and $\{\varphi(x;b_{\omega + i}) : i < \omega\}$ is inconsistent.  So $a \nind^{K}_{N}b$, a contradiction.  
\end{proof}

\section{Examples}\label{examplesection}

\subsection{A Kim-Pillay-style characterization of $\ind^{K}$}

We are interested in explicitly describing $\ind^{K}$ in concrete examples.  As in simple theories, this is most easily acheived by establishing the existence of an independence relation with certain properties and then deducing that, therefore, the relation coincides with $\ind^{K}$.  The following theorem explains how this works.  The content of the theorem is essentially the same as \cite[Proposition 5.8]{ArtemNick}, where a Kim-Pillay style criterion for NSOP$_{1}$ theories was observed, but we point out how this gives information about Kim-independence.  

\begin{thm} \label{criterion}
Assume there is an \(\text{Aut}(\mathbb{M})\)-invariant ternary relation \(\ind\) on small subsets of the monster \(\mathbb{M} \models T\) which satisfies the following properties, for an arbitrary \(M \models T\) and arbitrary tuples from $\mathbb{M}$.
\begin{enumerate}
\item Strong finite character:  if \(a \nind_{M} b\), then there is a formula \(\varphi(x,b,m) \in \text{tp}(a/bM)\) such that for any \(a' \models \varphi(x,b,m)\), \(a' \nind_{M} b\).
\item Existence over models:  \(M \models T\) implies \(a \ind_{M} M\) for any \(a\).
\item Monotonicity: \(aa' \ind_{M} bb'\) \(\implies\) \(a \ind_{M} b\).
\item Symmetry: \(a \ind_{M} b \iff b \ind_{M} a\).
\item The independence theorem: \(a \ind_{M} b\), \(a' \ind_{M} c\), \(b \ind_{M} c\) and $a \equiv_{M} a'$ implies there is $a''$ with $a'' \equiv_{Mb} a$, $a'' \equiv_{Mc} a'$ and $a'' \ind_{M} bc$.
\end{enumerate}
Then \(T\) is NSOP\(_{1}\) and $\ind$ \emph{strengthens} $\ind^{K}$\textemdash i.e. if $M \models T$, $a \ind_{M} b$ then $a \ind^{K}_{M} b$.  If, moreover, $\ind$ satisfies
\begin{enumerate}
\addtocounter{enumi}{5}
\item Witnessing:  if $a \nind_{M} b$ witnessed by $\varphi(x;b)$ and $(b_{i})_{i < \omega}$ is a Morley sequence over $M$ in a global $M$-invariant type extending $\text{tp}(b/M)$, then $\{\varphi(x;b_{i}) : i < \omega\}$ is inconsistent.  
\end{enumerate}
then $\ind = \ind^{K}$ over models, i.e. if $M \models T$, $a \ind_{M} b$ if and only if $a \ind^{K}_{M} b$.
\end{thm}

\begin{proof}
It was shown in \cite[Proposition 5.8]{ArtemNick} that if there is such a relation $\ind$, then $T$ is NSOP$_{1}$. The proof there shows that if $\ind$ satisfies axioms (1)-(4), then $a \ind^{u}_{M} b$ implies $a \ind_{M} b$.  Now suppose $a \ind_{M} b$.  Let $p(x;b) = \text{tp}(a/Mb)$ and let $q$ be a global coheir of $\text{tp}(b/M)$.  By the independence theorem for $\ind$, if $(b_{i})_{i < \omega}$ is a Morley sequence over $M$ in $q$ with $b_{0} = b$, then $\bigcup_{i < \omega} p(x;b_{i})$ is consistent.  But then $a \ind^{K}_{M} b$.  The ``moreover" clause follows by definition of $\ind^{K}$.  
\end{proof}

\begin{rem}
The condition (6) can be weakened to quantifying only over global \emph{coheirs} of $\text{tp}(b/M)$, or asserting the existence of \emph{one} such coheir -- this is sometimes slightly easier in practice.  
\end{rem}

\begin{rem}
Axioms (1)-(5) do not, by themselves, suffice to characterize $\ind^{K}$.  See Remark \ref{notequal} below.  
\end{rem}

\subsection{Combinatorial examples}

In this section, we study some combinatorial examples of NSOP$_{1}$ theories which are not simple.  They are structures which encode a generic family of selector functions for an equivalence relation.   The theories defined below provide a different presentation of a theory defined by D\v{z}amonja and Shelah in \cite{dvzamonja2004maximality} (where it was called $T^{*}_{feq}$ -- though this name is now typically reserved for a different theory) and later studied by Malliaris in  \cite{malliaris2012hypergraph} (where it was called $T^{s}$).  We give a family of theories $T^{*}_{n}$ as $n$ ranges over positive integers, but we will only be interested in the case of $n = 1,2$.  Among non-simple NSOP$_{1}$ theories, the theory $T^{*}_{1}$ is probably the easiest to understand, and we show that already $T^{*}_{1}$ witnesses many of the new phenomena in our context:  with respect to this theory, we give explicit examples of formulas which divide but do not Kim-divide, formulas which fork and do not divide over models, and types which contain no universal Morley sequences.  

We use $T^{*}_{1}$ to answer a question of Chernikov from \cite{ChernikovNTP2} concerning simple and co-simple types and a question of Conant from \cite{conant2014forking} concerning forking and dividing.  A type is simple if no instance of the tree property is consistent with the type and a type is co-simple if the tree property cannot witnessed using parameters which realize the type (see Definition \ref{simplecosimpledef} above for the precise definition).  For stability, no such distinction arises, but Chernikov was able to show that, in general, there are co-simple types which are not simple.  In fact, examples can be found in the triangle-free random graph.  It was asked if there can exist simple types which are not co-simple and he showed that there can be no such types in an NTP$_{2}$ theory.  In \cite{conant2014forking}, Conant gave a detailed analysis of forking and dividing in the Henson graphs and showed that forking does not equal dividing for formulas, though every complete type has a global non-forking extension.  As the Henson graphs all have the property SOP$_{3}$, Conant asked if there could be an NSOP$_{3}$ example of this behavior.  
We show the answer to both questions is yes already within the class of NSOP$_{1}$ theories.

Lastly, we use $T^{*}_{2}$ to give a counter-example to transitivity for $\ind^{K}$.  Because Kim-dividing does not behave well with respect to changing the base, the normal formulation of transitivity does not necessarily make sense.   Nonetheless, there is a natural way to formulate a version which does make sense.  Suppose $T$ is NSOP$_{1}$, $M \models T$ and both $a \ind^{K}_{M} bc$ and $b \ind^{K}_{M} c$.  Must it also be the case, then, that $ab \ind^{K}_{M} c$?  We show the answer is no.  

For the remainder of this subsection, if $A$ is a structure in some language and $X \subseteq A$, write $\langle X \rangle^{A}$ for the substructure of $A$ generated by $X$.  We write just $\langle X \rangle$ when $A $ is the monster model.    

For a natural number $n \geq 1$, let $L_{n} = \langle O,F, E, \text{eval} \rangle$ where $O,F$ are sorts, $E$ is a binary relation symbol, and $\text{eval}$ is an $n+1$-ary function.  The theory $T_{n}$ will say
\begin{itemize}
\item $O$ and $F$ are sorts\textemdash $O$ and $F$ disjoint and the universe is their union.  
\item $E \subseteq O^{2}$ is an equivalence relation on $O$.
\item $\text{eval}: F^{n} \times O \to O$ is a function so that for all $f \in F^{n}$, $\text{eval}(f,-)$ is a function from $O$ to $O$ which is a selector function for $E$ -- more formally, for all $b \in O$, we have $E(\text{eval}(f,b),b)$ and if $b,b' \in O$ and $E(b,b')$ then we have 
$$
\text{eval}(f,b) = \text{eval}(f,b').
$$
\end{itemize}
The letter $F$ is for `function' and $O$ is for `object'\textemdash we think of a tuple $f \in F^{n}$ as naming the function $\text{eval}(f,-)$.  Let $\mathbb{K}_{n}$ be the class of finite models of $T_{n}$.  

Recall that a Fra\"iss\'e class $\mathbb{K}$ is said to have the \emph{strong amalgamation property (SAP)} if, whenever $A,B,C \in \mathbb{K}$, and $e:A \to B$ and $f: A \to C$ are embeddings, then there is a structure $D \in \mathbb{K}$ and embeddings $g: B \to D$, $h: C \to D$ so that $ge = hf$ and, moreover, $(\text{im}g) \cap (\text{im}h) = (\text{im}ge)$ (and hence also $=(\text{im}hf)$).  

\begin{lem}
The class $\mathbb{K}_{n}$ is a Fra\"iss\'e class with SAP.  Moreover, it is uniformly locally finite.  
\end{lem}

\begin{proof}
HP is clear as the axioms of $T_{n}$ are universal.  The argument for JEP is identical to that for SAP, so we show SAP.  Suppose $A,B,C \in \mathbb{K}_{n}$ where $A \subseteq B,C$ and $B \cap C = A$.  It suffices to define a $L_{n}$-structure with domain $D = B \cup C$, extending both $B$ and $C$.  Interpret $O^{D}$ and $F^{D}$ by $O^{D} = O^{B} \cup O^{C}$ and $F^{D} = F^{B} \cup F^{C}$.  Let $E^{D}$ be the equivalence relation generated by $E^{B} \cup E^{C}$.  It follows that if $b \in B$, $c \in C$ and $(b,c) \in E^{D}$, then there is some $a \in A$ so that $(a,b) \in E^{B}$ and $(a,c) \in E^{C}$ and, moreover, $(O^{D}, E^{D})$ extends both $(O^{B}, E^{B})$ and $(O^{C}, E^{C})$ as equivalence relations.  

We are left with interpreting $\text{eval}^{D}$.  Let $\{a_{i} : i < k_{0}\}$ enumerate a collection of representatives for the $E^{A}$-classes in $A$.  Then let $\{b_{i} : i < k_{1}\}$ and $\{c_{i} : i < k_{2}\}$ enumerate representatives for the $E^{B}$- and $E^{C}$-classes of elements not represented by an element of $A$, respectively.  Then every element of $O^{D}$ is equivalent to a unique element of 
$$
X = \{a_{i} : i < k_{0}\} \cup \{b_{i} : i < k_{1}\} \cup \{c_{i} : i < k_{2}\}.  
$$
Suppose $d \in X$.  If $f \in (F^{B})^{n}$, define $\text{eval}^{D}(f,d) = \text{eval}^{B}(f,d)$ if $d \in B$ and $\text{eval}^{D}(f,d) = d$ otherwise.  Likewise, if $f \in (F^{C})^{n}$ and $d \in C$, put $\text{eval}^{D}(f,d) = \text{eval}^{C}(f,c)$ if $c \in C$ and $\text{eval}^{C}(f,c) = c$ otherwise.  If $f \in (F^{D})^{n} \setminus ((F^{B})^{n} \cup (F^{C})^{n})$, put $\text{eval}^{D}(f,d) = d$.  This defines $\text{eval}$ on $(F^{D})^{n} \times X$.  More generally, if $f \in (F^{D})^{n}$ and $e \in O^{D}$, define $\text{eval}^{D}(f,e) = \text{eval}^{D}(f,d)$ for the unique $d \in X$ equivalent to $e$.  This is well-defined as $B$ and $C$ agree on $A$ and the $D$ defined in this way is clearly in $\mathbb{K}_{n}$.  

Finally, note that a structure in $\mathbb{K}_{n}$ generated by $k$ elements is obtained by applying $\leq k^{n}$ functions of the form $\text{eval}(f,-)$ to $\leq k$ elements in $O$, so has cardinality $\leq k^{n+1}+k$.  This shows $\mathbb{K}_{n}$ is uniformly locally finite.   
\end{proof}

It follows that there is a complete $\aleph_{0}$-categorical theory $T_{n}^{*}$ extending $T_{n}$ whose models have age $\mathbb{K}_{n}$ \cite[Chapter 7]{hodges1993model}.  By the uniform local finiteness of $\mathbb{K}_{n}$, $T_{n}^{*}$ has quantifier-elimination so $T_{n}^{*}$ is the model completion of $T_{n}$.  Let $\mathbb{M}_{n} \models T_{n}^{*}$ be a monster model.  

\begin{defn}
Define a ternary relation $\ind^{*}$ on small subsets of $\mathbb{M}_{n}$ by:  $a \ind^{*}_{C} b$ if and only if
\begin{enumerate}
\item $\text{dcl}(aC)/E \cap \text{dcl}(bC) /E \subseteq \text{dcl}(C)/E$.
\item $\text{dcl}(aC) \cap \text{dcl}(bC) \subseteq \text{dcl}(C)$.  
\end{enumerate}
where $X/E = \{[x]_{E} :x \in X\}$ denotes the collection of $E$-classes represented by an element of $X$.  
\end{defn}

\begin{lem}\label{independencetheoremfor2functions}
The relation $\ind^{*}$ satisfies the independence theorem over structures:  if $M \models T_{n}$ (not necessarily $T^{*}_{n}$), $a \equiv_{M} a'$, $a \ind^{*}_{M} B$, $a' \ind^{*}_{M} C$ and $B \ind^{*}_{M} C$ then there is $a''$ with $a'' \equiv_{MB} a$, $a'' \equiv_{MC} a'$, and $a'' \ind^{*}_{M} BC$.  
\end{lem}

\begin{proof}
We may assume $M$ is a substructure of $\mathbb{M}_{n}$, $M \subseteq B,C$ and that $B$ and $C$ are definably closed.  Write $a = (d_{0},\ldots, d_{k-1},e_{0},\ldots, e_{l-1})$ with $d_{i} \in F$ and $e_{j} \in O$ and likewise 

\noindent $a' = (d'_{0}, \ldots, d'_{k-1}, e'_{0}, \ldots, e'_{l-1})$.  Fix an automorphism $\sigma \in \text{Aut}(\mathbb{M}_{n}/M)$ with $\sigma(a) = a'$.  Let $U = \{u_{f} : f \in \text{dcl}(aB) \setminus B\}$ and $V = \{v_{f} : f \in \text{dcl}(a'C) \setminus C\}$ denote a collection of new formal elements with $u_{h} = v_{\sigma(h)}$ for all $h \in \langle aM \rangle \setminus B$.  Let, then, $a_{*}$ be defined by 
$$
a_{*} = (u_{d_{0}}, \ldots, u_{d_{k-1}}, u_{e_{0}}, \ldots, u_{e_{l-1}}) = (v_{d'_{0}}, \ldots, v_{d'_{k-1}}, v_{e'_{0}}, \ldots, v_{e'_{l-1}}).
$$
We will construct by hand an $L$-structure $D$ extending $\langle BC \rangle$ with domain $UV\langle BC \rangle$ in which $a^{*} \equiv_{B} a$, $a^{*} \equiv_{C} a'$ and $a^{*} \ind^{*}_{M} BC$. 

There is a bijection $\iota_{0}: \text{dcl}(aB) \to BU$ given by $\iota_{0}(b) = b$ for all $b \in B$ and $\iota_{0}(f) = u_{f}$ for all $f \in \text{dcl}(aB) \setminus B$.  Likewise, we have a bijection $\iota_{1}: \text{dcl}(a'C) \to CV$ given by $\iota_{1}(c) = c$ for all $c \in C$ and $\iota_{1}(f) = v_{f}$ for all $f \in \text{dcl}(a'C) \setminus C$.  The union of the images of these functions is the domain of the structure $D$ to be constructed and their intersection is $\iota_{0}(\langle aM \rangle) = \iota_{1}(\langle a'M \rangle)$.  Consider $BU$ and $CV$ as $L_{n}$-structures by pushing forward the structure on $\text{dcl}(aB)$ and $\text{dcl}(a'C)$ along $\iota_{0}$ and $\iota_{1}$, respectively.  Note that $\iota_{0}|_{\langle aM \rangle} = (\iota_{1} \circ \sigma)|_{\langle aM \rangle}$.    

We are left to show that we can define an $L_{n}$-structure on $UV\langle BC \rangle$ extending that of $BU$, $CV$, and $\langle BC \rangle$ in such a way as to obtain a model of $T^{*}_{n}$.  To begin, interpret the predicates by $O^{D} = O^{BU} \cup O^{CV} \cup O^{\langle BC \rangle}$ and $F^{D} = F^{BU} \cup F^{CV} \cup F^{\langle BC \rangle}$.  Let $E^{D}$ be defined to be the equivalence relation generated by $E^{BU}$, $E^{CV}$, and $E^{\langle BC \rangle}$.  The interpretation of the predicates is well-defined since if $f$ is an element of $\iota_{0}(\langle aM \rangle) = \iota_{1}(\langle a'M \rangle)$ then $\iota_{0}^{-1}(f)$ is in the predicate $O$ if and only if $\iota_{1}^{-1}(f)$ is as well, and, moreover, it is easy to check that our assumptions on $a,a',B,C$ entail that no pair of inequivalent elements in $BU$, $CV$, or $\langle BC \rangle$ become equivalent in $D$.  

All that is left is to define the function $\text{eval}^{D}$ extending $\text{eval}^{BU} \cup \text{eval}^{CV} \cup \text{eval}^{\langle BC \rangle}$. We first claim that $\text{eval}^{BU} \cup \text{eval}^{CV} \cup \text{eval}^{\langle BC \rangle}$ is a function.  The intersection of the domains of the first two functions is (in a Cartesian power of)  $\iota_{0}(\langle aM \rangle) = \iota_{1}(\langle aM \rangle)$.  If $b,b'$ are in this intersection, we must show
$$
\text{eval}^{BU}(b,b') = c \iff \text{eval}^{CV}(b,b') = c.
$$
Choose $b_{0},b'_{0}, c_{0} \in \langle a M \rangle$ and $b_{1},b'_{1},c_{1} \in \langle a'M \rangle$ with $\iota_{i}(b_{i},b'_{i},c_{i}) = (b,b',c)$ for $i = 0,1$.  Then since $\iota_{0} = \iota_{1} \circ \sigma$ on $\langle aM \rangle$, we have 
\begin{eqnarray*}
\mathbb{M}_{n} \models \text{eval}(b_{0},b'_{0}) = c_{0} &\iff& \mathbb{M}_{n} \models \text{eval}(\sigma(b_{0}),\sigma(b'_{0})) = \sigma(c_{0}) \\
&\iff& \mathbb{M}_{n} \models \text{eval}(b_{1},b'_{1}) = c_{1}.
\end{eqnarray*}
Since $\text{eval}^{BU}$ and $\text{eval}^{CV}$ are defined by pushing forward the structure on $\langle aB\rangle$ and $\langle a'C \rangle$ along $\iota_{0}$ and $\iota_{1}$, respectively, this shows that $\text{eval}^{BU} \cup \text{eval}^{CV}$ defines a function.  Now the intersection of $\langle BC \rangle$ with $BU \cup CV$ is $BC$ and, by construction, all 3 functions agree on this set.  So the union defines a function.  

Choose a complete set of $E^{D}$-class representatives $\{d_{i}: i < \alpha\}$ so that if $d_{i}$ represents an $E^{D}$-class that meets $M$ then $d_{i} \in M$.  If $e \in O^{D}$ is $E^{D}$-equivalent to some $e'$ and $(f,e')$ is in the domain of $\text{eval}^{BU} \cup \text{eval}^{CV} \cup \text{eval}^{\langle BC \rangle}$, define $\text{eval}^{D}(f,e)$ to be the value that this function takes on $(f,e')$.  On the other hand, if $f \in (F^{D})^{n} \setminus ((F^{BU})^{n} \cup (F^{CV})^{n} \cup (F^{\langle BC \rangle})^{n})$ or $e$ is not $E^{D}$-equivalent to any element on which $\text{eval}^{D}(f,-)$ has already been defined, put $\text{eval}^{D}(f,e) = d_{i}$ for the unique $d_{i}$ which is $E^{D}$-equivalent to $e$.  This now defines $\text{eval}^{D}$ on all of $(F^{D})^{n} \times O^{D}$ and, by construction, $\text{eval}^{D}(f,-)$ is a selector function for $E^{D}$ for all $f \in (F^{D})^{n}$.  This completes the construction of $D$ and we have shown $D$ is a model of $T_{n}$.  By model-completeness and saturation, $D$ embeds into $\mathbb{M}_{n}$ over $BC$.  If we can show $a_{*} \ind^{*}_{M} BC$ in $D$, then this will be true for the image of $a_{*}$ in $D$.

We have already argued that $BU$ and $CV$ are substructures of $D$\textemdash it follows that every $E^{D}$-class represented by an element of $a_{*}$ can only be equivalent to an element of $B$ or $C$ if it is equivalent to an element of $M$.  Moreover, our construction has guaranteed that $\langle a_{*}M \rangle^{D} \cap \langle BC \rangle \subseteq BU \cap \langle BC \rangle^{D} \subseteq B$ and, by similar reasoning, $\langle a_{*}M \rangle \subseteq C$.  This implies $\langle a_{*} M \rangle^{D} \cap \langle BC \rangle B \cap C \subseteq M$, so $a_{*} \ind^{*}_{M} BC$.  
\end{proof}

\begin{prop}\label{kimindependencefor2functions}
The theory $T_{n}^{*}$ is NSOP$_{1}$ and, moreover, if $M \models T_{n}^{*}$, then $a \ind^{*}_{M} b$ if and only if $a \ind^{K}_{M} b$.  
\end{prop}

\begin{proof}
In Lemma \ref{independencetheoremfor2functions}, we showed $\ind^{*}$ satisfies the independence theorem over a model, and the other conditions (1)-(4) in Theorem \ref{criterion} are clear for $\ind^{*}$.  To show (6), notice that if $A \nind^{*}_{M} B$ with $A,B$ definably closed and containing $M$, then either there is some $a \in A$ and $b \in B$ so that $\models E(a,b)$ and the $E$-class of $b$ does not meet $M$ or $a = b$ for some $b \not\in M$.  Suppose $(b_{i})_{i < \omega}$ is a Morley sequence in some global $M$-invariant $q \supseteq \text{tp}(b/M)$.  If the class of $b$ does not meet $M$, then $\neg E(b_{i},b_{j})$ for $i \neq j$ by $M$-invariance so $\{E(x;b_{i}) : i < \omega\}$ is $2$-inconsistent.  Likewise, if $b$ is not in $M$, then $b_{i} \neq b_{j}$ for $i \neq j$ so $\{x = b_{i} : i < \omega\}$ are $2$-inconsistent.  It follows that $\ind^{*} = \ind^{K}$ over models.  
\end{proof}

We note that $\ind^{K}$ satisfies (a form of) local character in $T^{*}_{1}$:

\begin{prop} \label{localcharexample}
For any model $N \models T^{*}_{1}$ and $p \in S(N)$, there is a countable $M \prec N$ so that $p(x)$ does not Kim-fork over $M$.
\end{prop}

\begin{proof}
We use the characterization of $\ind^{K}$ from Proposition \ref{kimindependencefor2functions}.  Let $a \models p$ and choose $M_{0} \prec N$ to be an arbitrary elementary submodel.  By induction, construct an elementary chain $(M_{i})_{i < \omega}$ of countable elementary submodels of $N$ so that $\text{dcl}(aM_{i}) \cap N \subseteq M_{i+1}$ and every equivalence class in $\text{dcl}(aM_{i})/E \cap N/E$ is represented by an element of $M_{i+1}$.  Since $M_{i}$ is countable, $\text{dcl}(aM_{i})$ is countable, there is no problem in choosing $M_{i+1}$, by downward L\"owenheim-Skolem.  Let $M = \bigcup_{i < \omega} M_{i}$.  We claim $a \ind^{K}_{M} N$.  Given $c \in \text{dcl}(aM) \cap N$, there is $n$ so that $c \in \text{dcl}(aM_{n}) \cap N$ hence $c \in M_{n} \subseteq M$.  This shows $\text{dcl}(aM) \cap N \subseteq M$.  Arguing similiarly, we have $\text{dcl}(aM) / E \cap N /E \subseteq M/E$.  This shows $a \ind^{K}_{M} N$.  
\end{proof}

\begin{lem}\label{notcosimple}
Modulo $T^{*}_{1}$, the formula $O(x)$ axiomatizes a complete type over $\emptyset$ which is not co-simple.  
\end{lem}

\begin{proof}
That $O(x)$ implies a complete type is clear from quantifier-elimination.  In $O(\mathbb{M}_{1})$, choose an array $(a_{\alpha,\beta})_{\alpha, \beta < \omega}$ of distinct elements so that, for all $\alpha< \alpha' < \omega$, given $\beta, \beta'$, $\mathbb{M}_{1} \models E(a_{\alpha,\beta},a_{\alpha,\beta'})$ and $\mathbb{M}_{1} \models \neg E(a_{\alpha,\beta},a_{\alpha',\beta'})$.  Let $\varphi(x;y)$ be the formula $\text{eval}(x,y) = y$.  It is now easy to check 
\begin{itemize}
\item For all functions $f: \omega \to \omega$, $\{\varphi(x;a_{\alpha, f(\alpha)}) : \alpha < \omega\}$ is consistent
\item For all $\alpha < \omega$, $\{\varphi(x;a_{\alpha, \beta}) : \beta < \omega\}$ is $2$-inconsistent,
\end{itemize}
so $\varphi(x;y)$ witnesses TP$_2$ with respect to parameters realizing $O(x)$.  This shows $O(x)$ is not co-simple.  
\end{proof}

\begin{lem}
Suppose $A \subseteq \mathbb{M}_{1}$.  Then $\text{acl}(A) = \langle A \rangle = A \cup \text{eval}(F(A) \times O(A))$.
\end{lem}

\begin{proof}
The equality of $\text{acl}(A)$ and $\langle A \rangle$ follows from SAP for $\mathbb{K}_{1}$ \cite[Theorem 7.1.8]{hodges1993model}.  The axioms of $T^{*}_{1}$ imply that every term of $L_{1}$ is equivalent to one of the form $x$ or $\text{eval}(x,y)$, so $\langle A \rangle = A \cup \text{eval}(F(A) \times O(A))$.   
\end{proof}

We will see that $\ind^{*}$ characterizes dividing when elements on the left-hand side come from $O$.  The following lemma is the key ingredient in proving this:

\begin{lem}\label{dividing}
Suppose $A = \text{dcl}(A) \subseteq \mathbb{M}_{1}$ and $A = \langle a,B \rangle$ for some $a \in O(A)$ and $B = \text{dcl}(B) \subseteq \mathbb{M}_{1}$, where $l(a) = 1$.  Given a sequence $(B_{i})_{i < N}$ of substructures of $\mathbb{M}_{1}$ isomorphic to $B$ over $C = \text{dcl}(C)$ where for $i \neq j$, $B_{i} \cap B_{j} = C$.  Then if $a \ind^{*}_{C} B$, then there is a structure $D \models T_{1}$ and some $a' \in D$ so that
\begin{enumerate}
\item $\langle (B_{i})_{i < N} \rangle \subseteq D$.
\item $\langle a',B_{i} \rangle^{D} \cong_{C} \langle a, B \rangle$ for all $i < N$.  
\end{enumerate}
\end{lem}

\begin{proof}
Suppose $A = \langle a, B \rangle$, $(B_{i})_{i < N}$ and $C$ are given as in statement, satisfying (1).  If $a \in C$, the lemma is clear so assume it is not, and therefore $a \not\in B$ by our assumption that $A \models a \neq b$ for all $b \in B \setminus C$.  Moreover, we may assume $B_{0} = B$.  Note that the underlying set of $A$ is $B \cup \{a\} \cup \text{eval}(F(B),a)$.  Let $X = \langle (B_{i})_{i < N} \rangle$.

\textbf{Case 1}:  $A \models E(a,c)$ for some $c \in C$.  In this case, the underlying set of $A$ is $B \cup \{a\} \cup \text{eval}(F(B),c) = B \cup \{a\}$. Let $D$ be the extension of $X$ with underlying set $X \cup \{a\}$ with relations interpreted so that $D \models a \in O \wedge E(a,c)$ and the function $\text{eval}$ defined to extend $\text{eval}^{X}$ and so that $\text{eval}^{D}(d,a) = \text{eval}^{X}(d,c)$ for all $d \in O^{D}$.  It is easy to check that this satisfies (2).  

\textbf{Case 2}:  $A \models \neg E(a,c)$ for all $c \in C$.  By our assumption that $A$ satisfies (1), it follows that $A \models \neg E(a,b)$ for all $b \in B$ and hence the underlying set of $A$ is the disjoint union of $B$ and $\{a\} \cup \text{eval}^{A}(F(B),a)$.  Let $Y = \{a\} \cup \text{eval}^{A}(F(B),a)$.  We will define an $L_{1}$-structure extending $X$ with underlying set $X \cup Y$.  Interpret the sorts $F^{D} = F^{X}$ and $O^{D} = O^{X} \cup Y$.  Define the equivalence relation so that $E^{X} \subset E^{D}$ and $Y$ forms one $E^{D}$-class.  

Fix for all $i < N$ a $C$-isomorphism $\sigma_{i} : B_{i} \to B_{0}$ (assume $\sigma_{0} = \text{id}_{B_{0}}$).  Note that $F^{X} = \bigcup_{i < N} F^{B_{i}}$.  Interpret $\text{eval}^{D}$ to extend $\text{eval}^{X}$ and so that, if $b \in F^{B_{i}}$ and $e \in Y$, 
$$
\text{eval}^{D}(b,e) = \text{eval}^{A}(\sigma_{i}(b),a).
$$
This defines $D \models T_{1}$ and, by construction, the map extending $\sigma_{i}$ and sending $a \mapsto a$ induces an isomorphism $\langle a,B_{i} \rangle^{D} \to \langle a, B_{0} \rangle^{D} = A$ for all $i < N$.  This completes the proof.   
\end{proof}

\begin{cor} \label{dividingmore}
Suppose $F$ is a substructure of $\mathbb{M}_{1}$.  If $a \in O(\mathbb{M}_{1})$ and $l(a) = 1$, then $a \ind^{d}_{F} B$ if and only if $a \ind^{*}_{F} B$.  
\end{cor}

\begin{proof}
If $a \nind^{*}_{F} B$ then clearly $a \nind^{d}_{F} B$, so we prove the other direction.  Suppose $a \ind^{*}_{F} B$ and $a \nind^{d}_{F} B$ and we will get a contradiction. Suppose $\varphi(x;c,b)$ witnesses dividing, so $\varphi(x;c,b) \in \text{tp}(a/FB)$ with $c \in F$ and $b \in B$, and there is an $F$-indiscernible sequence $\langle b_{i} : i < \omega \rangle$ with $b_{0} = b$ so that $\{\varphi(x;c,b_{i}) : i < \omega\}$ is $k$-inconsistent for some $k$.  As $a \ind^{*}_{F} B$ we may, by growing the tuple $c$, assume that every equivalence class represented both by $b$ and $a$ is represented by an element of $c$.  Let $B_{i} = \langle c,b_{i} \rangle$, $C = B_{0} \cap B_{i}$ for some/all $i \neq 0$ (by $F$-indiscernibility, this is well-defined and contains $c$) and $A = \langle a,b,c \rangle$.  As $a \ind^{*}_{C} B$, the structures $A$, $C$, and $(B_{i})_{i < k+1}$ satisfy (1) of Lemma \ref{dividing}, and therefore there is $D \in \mathbb{K}_{1}$ and some $a' \in D$ so that $\langle (B_{i})_{i < k+1}\rangle \subseteq D$ and $\langle a',B_{i} \rangle^{D} \cong \langle a, B \rangle^{A}$ for all $i < k+1$.  By embedding $D$ into $\mathbb{M}_{1}$ over $\langle (B_{i})_{i < k+1} \rangle^{D}$ we see that, in $\mathbb{M}_{1}$, $\{\varphi(x;c,b_{i}) : i < k+1 \}$ is consistent by quantifier-elimination.  This is a contradiction.  
\end{proof}

\begin{cor}\label{artemQ}
The theory $T^{*}_{1}$ is NSOP$_1$ and the formula $O(x)$ axiomatizes a complete type which is simple and not cosimple.
\end{cor}

\begin{proof}
Lemma \ref{notcosimple} shows that $O(x)$ axiomatizes a complete type which is not cosimple.  To show $O(x)$ is simple, we have to show that $\ind^{d}$ satisfies local character on $O(x)$.  So fix any $a \in \mathbb{M}_{1}$ with $\mathbb{M}_{1} \models O(a)$ and any small set $B \subseteq \mathbb{M}_{1}$.  We may suppose $B = \text{dcl}(B)$.  Notice that $\text{dcl}(a) = a$.  If $a \in B$ then $a \ind^{*}_{a} B$.  If $a \not\in B$ but $\mathbb{M} \models E(a,b)$ for some $b \in B$ then $a \ind^{*}_{b} B$.  Finally, if $a$ not $E$-equivalent to any element of $B$ then $a \ind^{*}_{\emptyset} B$.  Corollary \ref{dividingmore} showed $a \ind^{*}_{C} B$ if and only if $a \ind^{d}_{C} B$ for any $a$ with $\mathbb{M} \models O(a)$, so $\ind^{d}$ satisfies local character on $O$.  Therefore $O$ is simple.  
\end{proof}

\begin{rem}
This answers Problem 6.10 of \cite{ChernikovNTP2}. 
\end{rem}

\begin{rem}
Given a model $M \models T^{*}_{1}$, one can consider the complete type $p(x)$ over $M$ axiomatized by saying 
\begin{itemize}
\item $O(x)$
\item $\neg E(x,m)$ for all $m \in O(M)$
\item $\text{eval}(m,x) \neq x$ for all $m \in O(M)$ 
\end{itemize}
In a similar fashion, one can check that this is simple, non-co-simple so, in particular, nothing is gained by working over a model.  In fact, in this situation, we get another proof of the corollary, using Proposition \ref{simpletype}, as we have shown that if $a \models p$, then $a \ind^{d}_{M} b$ if and only if $a \ind^{K}_{M} b$ so $p$ is simple.  
\end{rem}

Proposition \ref{nomorley} above shows that in any non-simple NSOP$_{1}$ theory, there are types over models with no universal Morley sequences in them.  The following gives an explicit example:

\begin{prop}\label{nouniversal}
Given $M \models T_{1}^{*}$, there is a type $p \in S(M)$ with no universal Morley sequence.
\end{prop}

\begin{proof}
Pick $b \in O(\mathbb{M})$ not in $M$ and let $p(x) = \text{tp}(b/M)$.  Towards contradiction, suppose $(b_{i})_{i < \omega}$ is a universal Morley sequence in $p$.  

\textbf{Case 1}:  $\mathbb{M} \models E(b_{i},b_{j})$ for all $i,j<\omega$.  

The formula $E(x;b)$ divides over $M$:  choose any $M$-indiscernible sequence $\langle c_{i} : i < \omega\rangle$ with $c_{0} = b$ and $\neg E(c_{i},c_{i+1})$ -- then $\{E(x;c_{i}) : i < \omega\}$ is inconsistent.  But $\{E(x;b_{i}) : i < \omega\}$ is consistent, a contradiction.

\textbf{Case 2}:  $\mathbb{M} \models \neg E(b_{i},b_{j})$ for $i \neq j$.  The formula $\text{eval}(x,b) = b$ divides over $M$ -- choose any $M$-indiscernible sequence $\langle c_{i} : i < \omega \rangle$ with $E(c_{i},c_{j})$ for all $i,j$ and $c_{0} = b$.  Then $\{\text{eval}(x,c_{i}) = c_{i} : i < \omega\}$ is inconsistent (as for any $a$, the function $\text{eval}(a,-)$ takes on only one value on elements of any equivalence class).  But $\{\text{eval}(x,b_{i}) = b_{i} : i < \omega\}$ is consistent, a contradiction.  
\end{proof}

\begin{prop} \label{forkingnotdividing}
In $T^{*}_{1}$, forking does not equal dividing, even over models.
\end{prop}

\begin{proof}
Fix $M \models T^{*}_{1}$.  Let $\varphi(x,y;z)$ be the formula $\text{eval}(x,z) = z \vee E(y,z)$.  Given any $b \in O(\mathbb{M}_{1})$ not in $M$, we claim the formula $\varphi(x,y;b)$ forks but does not divide over $M$.  The proof of Proposition \ref{nouniversal} shows that both $E(x,b)$ and $\text{eval}(x,b) = b$ divide over $M$ so $\varphi(x,y;b)$ forks over $M$.  Given any $M$-indiscernible sequence $\langle b_{i} :i < \omega \rangle$ starting with $b$, either all $b_{i}$'s lie in a single equivalence class, in which case $\{E(y,b_{i}) : i < \omega\}$ is consistent, or they all lie in different classes, in which case $\{\text{eval}(x,b_{i}) = b_{i} : i < \omega\}$ is consistent.  Either way, $\{\varphi(x,y;b_{i}) : i < \omega\}$ is consistent, so $\varphi(x,y;b)$ does not divide over $M$.  
\end{proof}

Recall that a formula $\psi(x;c)$ \emph{quasi-divides} over $M$ if there are finitely many $(c_{i})_{i < k}$ with $c_{i} \equiv_{M} c$ so that $\{\psi(x;c_{i}) : i < k\}$ is inconsistent.  Note that in the proof of Proposition \ref{forkingnotdividing}, the formula $\varphi(x,y;b)$ quasi-divides over $M$:  take $(b_{i})_{i < 4}$ to be 4 elements realizing $\text{tp}(b/M)$ so that $E(b_{0},b_{1})$, $E(b_{2},b_{3})$ and $\neg E(b_{1},b_{2})$.  Then $\{\varphi(x,y;b_{i}) : i < 4\}$ is inconsistent.  We ask if this must always be the case:

\begin{quest}
Suppose $T$ is NSOP$_{1}$, $M \models T$ and the formula $\varphi(x;b)$ forks over $M$.  Does $\varphi(x;b)$ necessarily quasi-divide over $M$?
\end{quest}

In a similar vein, the following problem was suggested by Artem Chernikov:

\begin{quest}
If $T$ is NSOP$_{1}$, $N \prec M \models T$, and $p(x) \in S(M)$ then must it be the case that $p(x)$ forks over $N$ if and only if $p(x)$ divides over $N$?  
\end{quest}

We note that graph-theoretic examples of theories for which forking and dividing are different, but coincide for complete types have been studied by Conant \cite{conant2014forking}.  

\begin{lem}
Any $a$ be a tuple in $F$, $b$ a tuple in $O$, and $C = \text{dcl}(C) \subseteq \mathbb{M}$.  Then $\text{tp}(a,b/C)$ extends to a global $C$-invariant type.
\end{lem}

\begin{proof}
Write $a = (a_{0}, \ldots, a_{n-1})$, $b = (b_{0}, \ldots, b_{k-1})$.  We may assume that no equalities occur between the elements of $a$ and of $b$, or between $a,b$ and $C$.  We define a $C$-invariant global type $p(x,y) \in S(\mathbb{M})$ as follows.  The type $p(x,y)$ contains all formulas of $\text{tp}(a,b/C)$ together with the following axiom scheme:
\begin{eqnarray*}
\text{eval}(x_{i},m) \neq m &\in& p(x,y) \text{ for all } i<n, m \in \mathbb{M} \setminus C. \\
\text{eval}(x_{i},m) \neq \text{eval}(x_{j},m) &\in& p(x,y) \text{ for all }i < j < n, m\in \mathbb{M} \text{ with } m/E \not\in C/E. \\
\text{eval}(x_{i},y_{j}) \neq m &\in& p(x,y) \text{ for all } i<n,j < k, m \in \mathbb{M} \setminus C. \\
\text{eval}(m,y_{j}) \neq y_{j} &\in&p(x,y) \text{ for all } j < k, m \in \mathbb{M} \setminus C. \\
\neg E(y_{j},m) &\in& p(x,y) \text{ for all } j <  k, m \in \mathbb{M} \text{ with } m/E \not\in C/E. 
\end{eqnarray*}
It is clear that this type is consistent and $C$-invariant.  We claim it implies a complete type over $\mathbb{M}$:  note that because $\text{eval}(x,\text{eval}(y,z)) = \text{eval}(x,z)$, every term is equivalent to $x$ or $\text{eval}(x,y)$.  Because $E(x,\text{eval}(y,z))$ is equivalent to $E(x,z)$, every atomic formula is equivalent to an equality of terms or of the form $E(x,y)$.  Equalities of the form $\text{eval}(x_{i},y_{j}) = \text{eval}(x_{i'},y_{j'})$ are implied or negated by $\text{tp}(a,b/C)$, so the truth value of every atomic formula in the variables $x,y$ with parameters in $\mathbb{M}$ is determined by the above.  
\end{proof}

\begin{cor}
The theory $T^{*}_{1}$ is an NSOP$_{1}$ theory for which forking does not equal dividing, yet every type has a global non-forking extension.
\end{cor}

\begin{rem}
This answers Question 7.1(1) of \cite{conant2014forking}, which asked if forking = dividing in every NSOP$_{3}$ theory in which every type has a global non-forking extension, as every NSOP$_{1}$ theory is NSOP$_{3}$ \cite[Claim 2.3]{dvzamonja2004maximality}.  
\end{rem}

Finally, the following proposition gives a counter-example to the form of transitivity mentioned at the beginning of the subsection.  

\begin{prop}
For any model $M \models T^{*}_{2}$, there are $f,g$, and $c$ so that $f \ind^{K}_{M} gc$, $g \ind^{K}_{M} c$, and $fg \nind^{K}_{M} c$.   
\end{prop}

\begin{proof}
Given $M \models T^{*}_{2}$, choose any $c \in \mathbb{M}_{2} \setminus M$ in an $E$-class represented by an element $m$ of $M$\textemdash let $\{m_{i} : i < \alpha\}$ enumerate a set of representatives for the remaining $E$-classes of $M$.  Then choose distinct elements $f,g \in F$ so that 
\begin{enumerate}
\item $\text{eval}(f,g,m) = \text{eval}(g,f,m) = c$.
\item $\text{eval}(f,h,m) = \text{eval}(h,f,m) =  m$ and
$$
\text{eval}(f,h,m_{i}) = \text{eval}(h,f,m_{i}) = m_{i} 
$$
for all $h \in F^{M} \cup \{f\}$.
\item $\text{eval}(g,h,m) = \text{eval}(h,g,m) = m$ and
$$
\text{eval}(g,h,m_{i}) = \text{eval}(h,g,m_{i}) = m_{i}
$$ 
for all $h \in F^{M} \cup \{g\}$.
\end{enumerate}
Then we have
\begin{eqnarray*}
\text{dcl}(fM) &=& M \cup \{f\} \\
\text{dcl}(gM) &=& M \cup \{g\} \\
\text{dcl}(cM) &=& M \cup \{c\} \\
\text{dcl}(fgM) &=& M \cup \{f,g,c\} \\
\text{dcl}(gcM) &=& M \cup \{g,c\} .
\end{eqnarray*}
It follows that $\text{dcl}(fM) \cap \text{dcl}(gcM)$ and $\text{dcl}(gM) \cap \text{dcl}(cM)$ are contained in $M$ so $f \ind^{*}_{M} gc$ and $g \ind^{*}_{M} c$.  However, $c \in (\text{dcl}(fgM) \cap \text{dcl}(cM)) \setminus M$, showing $fg \nind^{*}_{M} c$.  As Proposition \ref{kimindependencefor2functions} showed $\ind^{K} = \ind^{*}$, we are done.  
\end{proof}

\subsection{Frobenius Fields}

In this section, we study a class of NSOP$_{1}$ fields.  If $F$ is a field, we write $F^{\text{alg}}$ and $F^{s}$ for the algebraic and separable closures of $F$, respectively.

\begin{defn}
Suppose $F$ is a field.
\begin{enumerate}
\item We say $F$ is \emph{pseudo-algebraically closed} (PAC) if every absolutely irreducible variety over $F$ has an $F$-rational point.  
\item We say $F$ is a \emph{Frobenius field} if $F$ is PAC and its absolute Galois group $\mathcal{G}(F)$ has the \emph{embedding property} (also known as the \emph{Iwasawa property}), that is, if $\alpha: \mathcal{G}(F) \to A$ and $\beta : B \to A$ are continuous epimorphisms and $B$ is a finite quotient of $\mathcal{G}(F)$, then there is a continuous epimorphism $\gamma : \mathcal{G}(F) \to B$ so that $\beta \circ \gamma = \alpha$ as in the following diagram:
$$
\xymatrix{ & \mathcal{G}(F) \ar@{-->>}[dl]  \ar@{->>}[d] \\
B \ar@{->>}[r] & A }
$$
\end{enumerate}
\end{defn}

The free profinite group on countably many generators $\hat{F}_{\omega}$ has the embedding property so the $\omega$-free PAC fields are Frobenius fields.  However, there are many others\textemdash see, e.g., \cite[24.6]{fried2008field}.

\begin{defn}
Suppose $G$ is a profinite group.  Let $\mathcal{N}(G)$ be the collection of open normal subgroups of $G$.  We define 
$$
\mathcal{S}(G) = \coprod_{N \in \mathcal{N}(G)} G/N.  
$$
Let $L_{G}$ the language with a sort $X_{n}$ for each $n \in \mathbb{Z}^{+}$, two binary relation symbols $\leq$, $C$, and a ternary relation $P$.  We regard $\mathcal{S}(G)$ as an $L_{G}$-structure in the following way:
\begin{itemize}
\item The coset $gN$ is in sort $X_{n}$ if and only if $[G : N] \leq n$.
\item $gN \leq hM$ if and only if $N \subseteq M$
\item $C(gN,hM) \iff N \subseteq M$ and $gM = hM$.
\item $P(g_{1}N_{1}, g_{2}N_{2}, g_{3}N_{3}) \iff N_{1}=N_{2} = N_{3}$ and $g_{1}g_{2}N_{1} = g_{3}N_{1}$.  
\end{itemize}
Note that we do not require that the sorts be disjoint (see \cite[Section 1]{chatzidakis1998model} for a discussion on the syntax of this structure).  
\end{defn}

Interpretability of $\mathcal{S}(\mathcal{G}(F))$ in $(F^{\text{alg}},F)$ is proved in \cite[Proposition 5.5]{chatzidakis2002properties}.  The ``moreover'' clause is clear from the proof.  

\begin{fact} \label{interpretability}
Both $F$ and $\mathcal{S}(\mathcal{G}(F))$ are interpretable in $(K,F)$ where $K$ is any algebraically closed field containing $F$.  Call the interpretation $\pi$.  Moreover, if $L \subseteq F$ is a subfield so that $F$ is a regular extension of $L$, then the restriction of $\pi$ to $(K,L)$ produces an interpretation of $\mathcal{S}(\mathcal{G}(L))$, contained in $\mathcal{S}(\mathcal{G}(F))$ in a natural way.  
\end{fact}

\begin{lem}
Let $F$ be a large sufficiently saturated and homogeneous field (i.e. a monster model of its theory) and $M \prec F$ a small elementary substructure.  Suppose $A = \text{acl}(A)$, $B = \text{acl}(B)$ are subsets of $F$ with $M \subseteq A \cap B$.  
\begin{enumerate}
\item If $A \equiv_{M} B$ in $F$, then $\mathcal{S}(\mathcal{G}(A)) \equiv_{\mathcal{S}(\mathcal{G}(M))} \mathcal{S}(\mathcal{G}(B))$.
\item If $(A_{i})_{i < \omega}$ is an $M$-indiscernible sequence with $A_{0} = A$, then $(\mathcal{S}(\mathcal{G}(A_{i})))_{i < \omega}$ is $\mathcal{S}(\mathcal{G}(M))$-indiscernible.  
\item If $A \ind^{u}_{M} B$ in $F$, then $\mathcal{S}(\mathcal{G}(A)) \ind^{u}_{\mathcal{S}(\mathcal{G}(B))} \mathcal{S}(\mathcal{G}(B))$ in $\mathcal{S}(\mathcal{G}(M))$.  
\end{enumerate}
\end{lem}

\begin{proof}
(1)  If $A \equiv_{M} B$ in $F$, then there is an automorphism $\sigma \in \text{Aut}(F/M)$ with $\sigma(A) = B$.  The map $\sigma$ has an extension $\tilde{\sigma}$ to $F^{\text{alg}}$ which is, then, an automorphism of the pair $(F^{\text{alg}},F)$ taking $A$ to $B$ and fixing $M$ pointwise.  It follows $A \equiv_{M} B$ in the pair $(F^{\text{alg}},F)$.  Since $A = \text{acl}(A)$ and $B = \text{acl}(B)$, we know $F$ is a regular extension of $A$ and of $B$ (see, e.g., \cite[Section 1.17]{ZoePAC2}).  By Fact \ref{interpretability}, we have $\mathcal{S}(\mathcal{G}(A)) \equiv_{\mathcal{S}(\mathcal{G}(M))} \mathcal{S}(\mathcal{G}(B))$.  

(2)  If $(A_{i})_{i < \omega}$ is an $M$-indiscernible sequence with $A_{0} = A$, given $i_{0} < \ldots < i_{k-1}$ and $j_{0} < \ldots < j_{k-1}$, we know $A_{i_{0}}\ldots A_{i_{k-1}} \equiv_{M} A_{j_{0}}\ldots A_{j_{k-1}}$ so $\text{acl}(A_{i_{0}}\ldots A_{i_{k-1}}) \equiv_{M} \text{acl}(A_{j_{0}}\ldots A_{j_{k-1}})$.  Then by (1) $\mathcal{S}(\mathcal{G}(\text{acl}(A_{i_{0}}\ldots A_{i_{k-1}})))\equiv_{\mathcal{S}(\mathcal{G}(M))} \mathcal{S}(\mathcal{G}(\text{acl}(A_{j_{0}}\ldots A_{j_{k-1}})))$, which implies $(\mathcal{S}(\mathcal{G}(A_{i})))_{i < \omega}$ is $\mathcal{S}(\mathcal{G}(M))$-indiscernible.

(3)  In \emph{any} theory, if $\pi$ is an interpretation of the structure $X$ in the structure $Y$, and $A \ind^{u}_{C} B$ in $Y$, then $\pi(A) \ind^{u}_{\pi(C)} \pi(B)$.  It follows that if $A \ind^{u}_{M} B$ in $F$, then $\mathcal{S}(\mathcal{G}(A)) \ind^{u}_{\mathcal{S}(\mathcal{G}(M))} \mathcal{S}(\mathcal{G}(\mathcal{B}))$ by Fact \ref{interpretability}.  
\end{proof}

\begin{prop}\label{kimimpliesweak1}
Suppose $F$ is an arbitrary field and, in an elementary extension $F^{*}$ of $F$, $a \ind^{K}_{F} b$.  Then the fields $A = \text{acl}(Fa)$ and $B = \text{acl}(Fb)$ satisfy the following conditions:
\begin{enumerate}
\item $A$ and $B$ are linearly disjoint over $F$
\item $F^{*}$ is a separable extension of $AB$
\item $\text{acl}(AB) \cap A^{s}B^{s} = AB$.
\end{enumerate}
\end{prop}

\begin{proof}
In \cite[Theorem 3.5]{ZoePAC2}, Chatizdakis proves (1)-(3) for an \emph{arbitrary} theory of fields under the assumption that $a\ind^{f}_{F} b$.  She deduces from $a \ind^{f}_{F} b$ that there is an $F$-indiscernible coheir sequence $(B_{i})_{i < \omega}$, i.e. an $F$-indiscernible sequence with $B_{<i} \ind^{u}_{F} B_{i}$ for all $i$, so that $AB_{i} \equiv_{F} AB$ for all $i$ (rather, she proves this with a \emph{heir} sequence, but the argument is symmetric).  She then proves that (1)-(3) follow from the existence of such a sequence.  Note, however, that this follows merely from the assumption $a \ind^{K}_{F} b$. 
\end{proof}

\begin{rem}
Note (1) and (2) are equivalent to saying $A \ind^{SCF}_{F} B$ \cite[Remark 3.3]{ZoePAC2}, where SCF denotes the complete (stable) theory of which $F^{s}$ is a model.  
\end{rem}

\begin{lem}\label{kimimpliesweak2}
Suppose $F$ is a Frobenius field.  If $A = \text{acl}(A)$, $B = \text{acl}(B)$ contain $F$ and $A \ind^{K}_{F} B$ then $\mathcal{S}(\mathcal{G}(A)) \ind^{f}_{\mathcal{S}(\mathcal{G}(F))} \mathcal{S}(\mathcal{G}(B))$ in $\text{Th}(\mathcal{S}(\mathcal{G}(F)))$.  
\end{lem}

\begin{proof}
Chatzidakis \cite{chatzidakis1998model} shows that the Galois group $\mathcal{S}(\mathcal{G}(F))$ is $\omega$-stable.  Let $(B_{i})_{i < \omega}$ be a Morley sequence in a global type finitely satisfiable in $F$ extending $\text{tp}(B/F)$.  As $A \ind^{K}_{F} B$, we may assume $(B_{i})_{i < \omega}$ is $A$-indiscernible.  Then $(\mathcal{S}(\mathcal{G}(B_{i})))_{i < \omega}$ is a Morley sequence in a global type finitely satisfiable in $\mathcal{S}(\mathcal{G}(F))$ which is moreover $\mathcal{S}(\mathcal{G}(A))$-indiscernible.  This implies $\mathcal{S}(\mathcal{G}(A)) \ind^{K}_{\mathcal{S}(\mathcal{G}(F))} \mathcal{S}(\mathcal{G}(B))$.  As $\text{Th}(\mathcal{S}(\mathcal{G}(F))$ is simple, this implies $\mathcal{S}(\mathcal{G}(A)) \ind^{f}_{\mathcal{S}(\mathcal{G}(F))} \mathcal{S}(\mathcal{G}(B))$ by Kim's lemma \cite[Proposition 2.1]{kim1998forking}.  
\end{proof}

Fix a field $F$ and let SCF denote the complete theory of which $F^{s}$ is a model.  

\begin{defn}
Suppose $A = \text{acl}(A)$, $B = \text{acl}(B)$, and $C = \text{acl}(C)$ in the field $F$.  We say $A$ is \emph{weakly independent} from $B$ over $C$ if 
\begin{enumerate}
\item $A \ind^{\text{SCF}}_{C} B$
\item $\mathcal{S}(\mathcal{G}(A)) \ind^{f}_{\mathcal{S}(\mathcal{G}(F))} \mathcal{S}(\mathcal{G}(B))$, where $\ind^{f}$ denotes non-forking independence in $\text{Th}(\mathcal{S}(\mathcal{G}(\mathcal{F})))$
\end{enumerate}
Extend this to arbitrary tuples by stipulating $a$ is \emph{weakly independent} from $b$ over $c$ if and only if $\text{acl}(a,c)$ is weakly independent from $\text{acl}(b,c)$ over $\text{acl}(c)$.  
\end{defn}

\begin{thm}{\cite[Theorem 6.1]{chatzidakis2002properties}} \label{zoeIT}
Let $F$ be a Frobenius field, sufficiently saturated, and $E = \text{acl}(E)$ a subfield of $F$.  Assume, moreover, that $\text{acl}(\mathcal{S}(\mathcal{G}(E))) = \mathcal{S}(\mathcal{G}(E))$ and if the degree of imperfection of $F$ is finite, that $E$ contains a $p$-basis of $F$.  Assume that the tuples $a,b,c_{1},c_{2}$ of $F$ satisfy:
\begin{enumerate}
\item $a$ and $c_{1}$ are weakly independent over $E$, $b$ and $c_{2}$ are weakly independent over $E$, $c_{1} \equiv_{E} c_{2}$
\item $\text{acl}(Ea)$ and $\text{acl}(Eb)$ are SCF-independent over $E$.  
\end{enumerate}
Then there is $c$ realizing $\text{tp}(\text{acl}(Ea)) \cup \text{tp}(c_{2}/\text{acl}(Eb))$ such that $c$ and $\text{acl}(Eab)$ are weakly independent over $E$.  
\end{thm}

\begin{thm}
Suppose $F$ is a Frobenius field and $a,b$ are tuples from an elementary extension of $F$.  Then $a \ind^{K}_{F} b$ if and only if $a$ and $b$ are weakly independent over $F$.
\end{thm}

\begin{proof}
Given $a,b$, and $F$, set $A = \text{acl}(aF)$ and $B = \text{acl}(bF)$.  It suffices to show $A \ind^{K}_{F} B$ if and only if $A$ is weakly independent from $B$ over $F$.  If $A \ind^{K}_{F} B$, then $A \ind^{SCF}_{F} B$ by Proposition \ref{kimimpliesweak1} and $\mathcal{S}(\mathcal{G}(A)) \ind^{f}_{\mathcal{S}(\mathcal{G}(F))} \mathcal{S}(\mathcal{G}(B))$ by Proposition \ref{kimimpliesweak1}.  Hence $A$ and $B$ are weakly independent over $F$.  For the other direction, suppose $A$ and $B$ are weakly independent over $F$.  Let $(B_{i})_{i < \omega}$ be a Morley sequence in a global $F$-invariant type with $B_{0} = B$ and set $p(X;B) = \text{tp}(A/B)$.  We will show by induction that $\bigcup_{i \leq n} p(X;B_{i})$ has a realization weakly independent from $(B_{i})_{i \leq n}$ over $F$.  For $n = 0$, this is by the assumption that $A$ and $B$ are weakly independent over $F$.  If it has been shown for $n$, then note that, because, $B_{n+1} \ind^{i}_{F} B_{0}\ldots B_{n}$, we have, in particular, $B_{n+1}$ and $(B_{i})_{i \leq n}$ are weakly independent over $F$.  By Theorem \ref{zoeIT}, $p(X;B_{n+1}) \cup \bigcup_{i \leq n} p(X;B_{i})$ has a realization weakly independent from $(B_{i})_{i \leq n+1}$.  By compactness, we conclude $\bigcup_{i < \omega} p(X;B_{i})$ is consistent.  As $(B_{i})_{i < \omega}$ was arbitrary, this shows $A \ind^{K}_{F} B$.  
\end{proof}

\subsection{Vector spaces}

The theories of a vector space over a field equipped with a symmetric or alternating bilinear form have model companions\textemdash they are the theories of an infinite dimensional vector space over an algebraically closed field equipped with a generic nondegenerate alternating or symmetric bilinear form. We use $T_{\infty}$ to refer to both the model companion where the form is symmetric and where it is alternating, as this choice makes no difference for our analysis below.  The language is two-sorted:  there is a sort $V$ for the vector space, with the language of abelian groups on it, a sort $K$ for the field, equipped with the ring language, a function $K \times V \to V$ for the action of scalar multiplication, and a function $[,]:V \times V \to K$ for the bilinear form.  In this subsection, we write $\mathbb{M} \models T_{\infty}$ for a fixed monster model of $T_{\infty}$.  

\begin{fact}
Given a set $X \subseteq \mathbb{M}$, write $X_{K}$ for the field points of $X$ and $X_{V}$ for the vector space points of $X$.  For $Y$ a set of vectors, write $\langle Y \rangle$ for the $\mathbb{M}_{K}$-span of $V$.  
\begin{enumerate}
\item $T_{\infty}$ eliminates quantifiers after expanding the vector space sort with an $n$-ary predicate $\theta_{n}$ interpreted so that $\models \theta_{n}(v_{0},\ldots, v_{n-1})$ if and only if $v_{0}, \ldots, v_{n-1}$ are linearly independent for all $n \geq 2$ \cite[Theorem 9.2.3]{Granger}.
\item For any set $A \subseteq \mathbb{M}$, the field points of $\text{dcl}(A)$ contain the field generated by $A_{K}$, $\{[a,b] : a,b \in A_{V}\}$, and for each $n$, and every set $\{\alpha_{0}, \ldots, \alpha_{n-1}\}$ such that there are $v_{0}, \ldots, v_{n} \in A_{V}$ with $\mathbb{M} \models \theta_{n}(v_{0}, \ldots, v_{n-1})$ and $v_{n} = \alpha_{0}v_{0} + \ldots + \alpha_{n-1}v_{n-1}$.  The vector space points of $\text{dcl}(A)$ are the $(\text{dcl}(A))_{K}$-span of $A_{V}$.  The field points of $\text{acl}(A)$ are the algebraic closure of $(\text{dcl}(A))_{K}$ and the vector space points of $\text{acl}(A)$ are the $(\text{acl}(A))_{K}$-span of $A_{V}$ \cite[Proposition 9.5.1]{Granger}.
\end{enumerate}
\end{fact}  

\begin{defn}
Write $\ind^{ACF}$ to denote algebraic independence, which coincides with non-forking independence in the theory ACF.  Suppose \(A \subseteq B\) and \(c\) is a singleton.  Let \(c \ind^{\Gamma}_{A} B\) be the assertion that \((\text{dcl}(cA))_{K} \ind^{ACF}_{(\text{dcl}(A))_{K}} (\text{dcl}(B))_{K}\) and one of the following holds:
\begin{enumerate}
\item \(c \in \mathbb{M}_{K}\)
\item \(c \in \langle A_{V} \rangle\)
\item \(c \not\in \langle B_{V} \rangle\) and \([c,B]\) is \(\Phi\)-independent over \(A\), 
\end{enumerate}
where `\([c,B]\) is \(\Phi\)-independent over $A$' means that whenever \(\{b_{0}, \ldots, b_{n-1}\}\) is a linearly independent set in \(B_{V} \cap (\mathbb{M}_{V} \setminus \langle A \rangle)\) then the set \(\{[c,b_{0}], \ldots, [c,b_{n-1}]\}\) is algebraically independent over the compositum of \((\text{dcl}(B))_{K}$ and $(\text{dcl}(Ac))_{K}\).  

By induction, for \(c = (c_{0}, \ldots, c_{m})\) define \(c \ind_{A}^{\Gamma} B\) by 
\[
c \ind_{A}^{\Gamma} B \iff (c_{0}, \ldots, c_{m-1}) \ind^{\Gamma}_{A} B \text{ and } c_{m} \ind^{\Gamma}_{Ac_{0}\ldots c_{m-1}} B c_{0}\ldots c_{m-1}.
\]
\end{defn}

\begin{fact}\cite[Theorem 12.2.2]{Granger} \cite[Lemma 6.1]{ArtemNick}
The relation $\ind^{\Gamma}$ is automorphism invariant and symmetric.  Moreover, it satisfies extension, strong finite character, and the independence theorem over a model.  Consequently, $T_{\infty}$ is NSOP$_{1}$.    
\end{fact}

\begin{prop} \label{kimchar}
Suppose $M \models T_{\infty}$.  Then if $A = \text{acl}(A)$, $B = \text{acl}(B)$ and $A \cap B \supseteq M$, then $A \ind^{K}_{M} B$ if and only if $A \cap B = M$.  
\end{prop}

\begin{proof}
The right to left direction is trivial and holds in any theory.  Suppose $M$ is a model, $A = \text{acl}(A)$, $B = \text{acl}(B)$, and $A \cap B \subseteq M$.  Let $C = \text{acl}(AB)$ and let $(C_{i})_{i < \omega}$ be an $M$-invariant Morley sequence over $M$ with $C_{0} = C$.  Fix $\sigma \in \text{Aut}(\mathbb{M}/M)$ with $\sigma(C_{i}) = C_{i+1}$ for all $i < \omega$.  By restricting the sequence $(C_{i})_{i < \omega}$ to a subtuple, we obtain an $M$-invariant Morley sequence $(B_{i})_{i < \omega}$ with $B_{0} = B$.  Let $D = \text{acl}((B_{i})_{i < \omega})$.  Let $\tilde{K} = (\text{acl}((C_{i})_{i < \omega}))_{K}$.  Let $\{u_{i} : i < \alpha\}$ be a basis for $M_{V}$.  Let $\{v_{i} : i < \beta\}$ complete this set to a basis for $A_{V}$ and let $(w_{0,j})_{j < \gamma}$ complete it to a basis for $(B_{0})_{V}$, then let $(w_{i,j})_{j < \gamma}$ be the set of vectors completing $\{u_{i} : i < \alpha\}$ to a basis for $(B_{i})_{V}$ corresponding to the $(w_{0,j})_{j < \beta}$\textemdash i.e. $w_{i,j} = \sigma^{i}(w_{0,j})$.  By our assumptions, $\{u_{i} : i < \alpha\}  \cup \{v_{i} : i < \beta\} \cup \{w_{i,j} : i < \omega, j < \gamma\}$ is a set of linearly independent vectors in $\mathbb{M}_{V}$.  Let $\tilde{V}$ be the $\tilde{K}$-vector space with basis $\{u_{i} : i < \alpha\}  \cup \{v_{i} : i < \beta\} \cup \{w_{i,j} : i < \omega, j < \gamma\}$.  To define the model $N = (\tilde{V}, \tilde{K})$, we are left with definining the form on $\tilde{V}$\textemdash for this it suffices to define the form on a basis.  First, interpret the form so that $N$ extends the structure on $D$\textemdash i.e. 
$$
[u_{i},u_{i'}]^{N} = k \iff [u_{i},u_{i'}]^{D} = k
$$
$$
[u_{i},w_{i',j}]^{N} = k \iff [u_{i}, w_{i',j}]^{D} = k
$$
$$
[w_{i,j}, w_{i',j'}]^{N} = k \iff [w_{i,j},w_{i',j'}]^{D} = k.
$$
And likewise, interpret the structure so that it extends the structure on $A$\textemdash i.e. 
$$
[u_{i},v_{i'}]^{N} = k \iff [u_{i}, v_{i'}]^{A} = k
$$
$$
[v_{i},v_{i'}]^{N} = k \iff [v_{i}, v_{i'}]^{A} = k.
$$
Then finally, we interpret the form so that the structure generated by $AB_{i}$ does not depend on $i$: put $[v_{i},w_{0,j}]^{N} = k \iff [v_{i},w_{0,j}]^{C} = k$ and set 
$$
[v_{i},w_{i',j}]^{N} = \left\{ \begin{matrix}
k & \text{ if } [v_{i},w_{0,j}]^{C} = k \in A \\
\sigma^{i'}(k) & \text{ if } [v_{i},w_{0,j}]^{C} = k \not\in A
\end{matrix} \right.
$$
This defines $N$.  By quantifier-elimination, there is an embedding $\iota: N \to \mathbb{M}$ over $D$ into $\mathbb{M}$.  Let $A' = \iota(A)$.  By quantifier-elimination, we have $AB_{0} \equiv_{M} A'B_{i}$ for all $i$.  This shows $\text{tp}(A/B)$ does not Kim-divide over $M$.
\end{proof}       

\begin{prop}\label{kimandgammaprops}
Suppose $M \models T_{\infty}$.  Then 
\begin{enumerate}
\item $a \ind^{\Gamma}_{M} b \implies a \ind^{K}_{M} b$.  
\item There are $a$ and $b$ so that $a \ind^{K}_{M} b$ and $a \nind^{\Gamma}_{M} b$.  
\end{enumerate}
\end{prop}

\begin{proof}
(1)  Suppose $a \ind^{\Gamma}_{M} b$.  By transitivity of $\ind^{\text{ACF}}$, $(\text{dcl}(aM))_{K} \ind^{\text{ACF}}_{M_{K}} (\text{dcl}(bM))_{K}$ so 
$$
\text{acl}(aM)_{K} \ind^{\text{ACF}}_{M_{K}} (\text{acl}(bM))_{K}
$$ 
since the field points of the algebraic closure of any set $X$ are just the field-theoretic algebraic closure of $(\text{dcl}(X))_{K}$.  Similarly, transitivity of independence for vector spaces forces $\langle (aM)_{V} \rangle \cap \langle (bM)_{V} \rangle \subseteq \langle M \rangle$.  This shows $\text{acl}(aM) \cap \text{acl}(bM) = M$ so $a \ind^{K}_{M} b$, by Proposition \ref{kimchar}.     

(2)  Given any $M \models T_{\infty}$, choose two vectors $b_{1}, b_{2} \in \mathbb{M}_{V}$ that are $\mathbb{M}_{K}$-linearly independent over $M$.  By model-completeness, we can find some vector $a$ so that $\text{acl}(aM) \cap \text{acl}(b_{1}b_{2}M) \subseteq M$, so $a \ind^{K}_{M} b_{1}b_{2}$, and also $[a,b_{1}] = [a,b_{2}]$.  Then we clearly have $\{[a,b_{1}],[a,b_{2}]\}$ algebraically dependent, as they are equal, hence $a \nind^{\Gamma}_{M} b_{1}b_{2}$.  
\end{proof}

\begin{rem}\label{notequal}
This observation implies that axioms (1)-(5) in Theorem \ref{criterion} do not suffice to characterize $\ind^{K}$, since $\ind^{\Gamma}$ satisfies these axioms and $\ind^{\Gamma} \neq \ind^{K}$ by Proposition \ref{kimandgammaprops}(2).  
\end{rem}

\subsection*{Acknowledgements}

This work constitutes part of the dissertation of the second-named author.  He would like to thank Thomas Scanlon and Leo Harrington for many useful conversations.  Artem Chernikov also had a profound influence on this project and it is our pleasure to thank him.  Finally, we would like to thank the anonymous referee writing such a thorough report in so little time.  

\bibliographystyle{alpha}
\bibliography{ms.bib}{}

\begin{thebibliography}{KLM89}

\bibitem[Adl05]{adler2005explanation}
Hans Adler.
\newblock Explanation of independence.
\newblock {\em arXiv preprint math/0511616}, 2005.

\bibitem[Adl14]{adler2014kim}
Hans Adler.
\newblock Kim's lemma for ntp2 theories: a simpler proof of a result by
  chernikov and kaplan.
\newblock {\em Rendiconti del seminario matematico}, 72(3-4), 2014.

\bibitem[CH03]{cherlin2003finite}
Gregory~L Cherlin and Ehud Hrushovski.
\newblock {\em Finite structures with few types}.
\newblock Number 152. Princeton University Press, 2003.

\bibitem[Cha98]{chatzidakis1998model}
Zo{\'e} Chatzidakis.
\newblock Model theory of profinite groups having the iwasawa property.
\newblock {\em Illinois Journal of Mathematics}, 42(1):70--96, 1998.

\bibitem[Cha99]{ZoePAC2}
Zo{{\'e}} Chatzidakis.
\newblock Simplicity and independence for pseudo-algebraically closed fields.
\newblock In {\em Models and computability ({L}eeds, 1997)}, volume 259 of {\em
  London Math. Soc. Lecture Note Ser.}, pages 41--61. Cambridge Univ. Press,
  Cambridge, 1999.

\bibitem[Cha02]{chatzidakis2002properties}
Zo{\'e} Chatzidakis.
\newblock Properties of forking in $\omega$-free pseudo-algebraically closed
  fields.
\newblock {\em The Journal of Symbolic Logic}, 67(03):957--996, 2002.

\bibitem[Che14]{ChernikovNTP2}
Artem Chernikov.
\newblock Theories without the tree property of the second kind.
\newblock {\em Ann. Pure Appl. Logic}, 165(2):695--723, 2014.

\bibitem[CK12]{chernikov2012forking}
Artem Chernikov and Itay Kaplan.
\newblock Forking and dividing in {$NTP_2$} theories.
\newblock {\em The Journal of Symbolic Logic}, 77(01):1--20, 2012.

\bibitem[Con14]{conant2014forking}
Gabriel Conant.
\newblock Forking and dividing in henson graphs.
\newblock {\em arXiv preprint arXiv:1401.1570}, 2014.

\bibitem[CR16]{ArtemNick}
Artem Chernikov and Nicholas Ramsey.
\newblock On model-theoretic tree properties.
\newblock {\em Journal of Mathematical Logic}, page 1650009, 2016.

\bibitem[DS04]{dvzamonja2004maximality}
Mirna D{\v{z}}amonja and Saharon Shelah.
\newblock On◁∗-maximality.
\newblock {\em Annals of Pure and Applied Logic}, 125(1):119--158, 2004.

\bibitem[FJ08]{fried2008field}
Michael~D Fried and Moshe Jarden.
\newblock Field arithmetic. a series of modern surveys in mathematics, 11,
  2008.

\bibitem[GIL02]{grossberg2002primer}
Rami Grossberg, Jos{\'e} Iovino, and Olivier Lessmann.
\newblock A primer of simple theories.
\newblock {\em Archive for Mathematical Logic}, 41(6):541--580, 2002.

\bibitem[Gra99]{Granger}
Nicholas Granger.
\newblock {\em Stability, simplicity, and the model theory of bilinear forms}.
\newblock PhD thesis, University of Manchester, 1999.

\bibitem[Hod93]{hodges1993model}
Wilfrid Hodges.
\newblock {\em Model theory}, volume~42.
\newblock Cambridge University Press Cambridge, 1993.

\bibitem[HP94]{hrushovski1994groups}
Ehud Hrushovski and Anand Pillay.
\newblock Groups definable in local fields and pseudo-finite fields.
\newblock {\em Israel Journal of Mathematics}, 85(1):203--262, 1994.

\bibitem[Hru91]{hrushovski1991pseudo}
Ehud Hrushovski.
\newblock Pseudo-finite fields and related structures.
\newblock {\em Model theory and applications}, 11:151--212, 1991.

\bibitem[Hru12]{hrushovski2012stable}
Ehud Hrushovski.
\newblock Stable group theory and approximate subgroups.
\newblock {\em Journal of the American Mathematical Society}, 25(1):189--243,
  2012.

\bibitem[Kan03]{MR1994835}
Akihiro Kanamori.
\newblock {\em The higher infinite}.
\newblock Springer Monographs in Mathematics. Springer-Verlag, Berlin, second
  edition, 2003.
\newblock Large cardinals in set theory from their beginnings.

\bibitem[Kim98]{kim1998forking}
Byunghan Kim.
\newblock Forking in simple unstable theories.
\newblock {\em Journal of the London Mathematical Society}, 57(02):257--267,
  1998.

\bibitem[Kim01]{kim2001simplicity}
Byunghan Kim.
\newblock Simplicity, and stability in there.
\newblock {\em The Journal of Symbolic Logic}, 66(02):822--836, 2001.

\bibitem[Kim09]{KimNTP1}
Byunghan Kim.
\newblock $\text{NTP}_{1}$ theories.
\newblock Slides, Stability Theoretic Methods in Unstable Theories, BIRS,
  February 2009.

\bibitem[KKS14]{KimKimScow}
Byunghan Kim, Hyeung-Joon Kim, and Lynn Scow.
\newblock Tree indiscernibilities, revisited.
\newblock {\em Arch. Math. Logic}, 53(1-2):211--232, 2014.

\bibitem[KLM89]{KLM}
William~M Kantor, Martin~W Liebeck, and H~Dugald Macpherson.
\newblock $\aleph_{0}$-categorical structures smoothly approximated by finite
  substructures.
\newblock {\em Proceedings of the London Mathematical Society}, 3(3):439--463,
  1989.

\bibitem[KLS16]{18-Kaplan2014a}
Itay Kaplan, Noa Lavi, and Saharon Shelah.
\newblock The generic pair conjecture for dependent finite diagrams.
\newblock {\em Israel J. Math.}, 212(1):259--287, 2016.

\bibitem[KP97]{kim1997simple}
Byunghan Kim and Anand Pillay.
\newblock Simple theories.
\newblock {\em Annals of Pure and Applied Logic}, 88(2-3):149--164, 1997.

\bibitem[Mal12]{malliaris2012hypergraph}
ME~Malliaris.
\newblock Hypergraph sequences as a tool for saturation of ultrapowers.
\newblock {\em The Journal of Symbolic Logic}, 77(01):195--223, 2012.

\bibitem[MS15]{malliaris2015model}
M~Malliaris and S~Shelah.
\newblock Model-theoretic applications of cofinality spectrum problems.
\newblock {\em arXiv preprint arXiv:1503.08338}, 2015.

\bibitem[She80]{shelah1980simple}
Saharon Shelah.
\newblock Simple unstable theories.
\newblock {\em Annals of Mathematical Logic}, 19(3):177--203, 1980.

\bibitem[She90]{shelah1990classification}
Saharon Shelah.
\newblock {\em Classification theory: and the number of non-isomorphic models}.
\newblock Elsevier, 1990.

\bibitem[Sim15]{simon2015guide}
Pierre Simon.
\newblock {\em A guide to NIP theories}.
\newblock Cambridge University Press, 2015.

\bibitem[SU08]{shelah2008more}
Saharon Shelah and Alexander Usvyatsov.
\newblock More on sop 1 and sop 2.
\newblock {\em Annals of Pure and Applied Logic}, 155(1):16--31, 2008.

\bibitem[YC14]{yaacov2014independence}
Ita{\"\i}~Ben Yaacov and Artem Chernikov.
\newblock An independence theorem for ntp 2 theories.
\newblock {\em The Journal of Symbolic Logic}, 79(01):135--153, 2014.

\end{thebibliography}

\end{document}